\documentclass[11pt,reqno]{amsart}
\usepackage{xifthen,setspace,bm}
\usepackage[utf8]{inputenc}
\usepackage[T1]{fontenc} 
\usepackage{bm}
\usepackage{amsmath,amsthm,amsfonts,amssymb,mathrsfs,stmaryrd,bigints}
\usepackage{dsfont}
\usepackage{mathtools}
\usepackage[usenames]{color}
\usepackage{thmtools, thm-restate}

\usepackage[colorlinks=true,linkcolor=blue]{hyperref}
\usepackage{epstopdf} 
\usepackage{amssymb}
\usepackage{setspace}
\usepackage{enumerate}
\usepackage{bigstrut}
\usepackage{multirow}
\usepackage{mathtools,xparse}
\usepackage{mathrsfs}
\usepackage{hyperref}
\usepackage{xcolor}
\usepackage [autostyle, english = american]{csquotes}
\usepackage{pgfplots}
\newtheorem*{lemma*}{Lemma}
\newtheorem{theorem}{Theorem}[section]
\newtheorem{corollary}[theorem]{Corollary}
\newtheorem{lemma}[theorem]{Lemma}
\newtheorem{proposition}[theorem]{Proposition}
\theoremstyle{definition}
\newtheorem{definition}{Definition}
\newtheorem{remark}[theorem]{Remark}

\DeclareMathOperator*{\argmax}{arg\,max}
\DeclareMathOperator{\Binom}{Binom}
\allowdisplaybreaks
\usepackage{chngcntr}

\newcommand{\calG}{\mathcal{G}}
\renewcommand{\P}{\mathbb{P}}

\newcommand{\ri}{\rightarrow \infty}
\newcommand{\ro}{\rightarrow 0}

\newcommand{\E}{{\mathbb{E}}}
\newcommand{\1}{\mathds{1}}

\newcommand{\R}{\mathbb{R}}
\newcommand{\norm}[1]{\left\lVert#1\right\rVert}

\newcommand{\e}{\varepsilon}

	\renewcommand{\P}{\mathbb{P}}

\newcommand{\cB}{\mathcal{B}}
\newcommand{\cC}{\mathcal{C}}
\newcommand{\cD}{\mathcal{D}}
\newcommand{\cE}{\mathcal{E}}
\newcommand{\cF}{\mathcal{F}}
\newcommand{\cG}{\mathcal{G}}

\newcommand{\cJ}{\mathcal{J}}
\newcommand{\cK}{\mathcal{K}}
\newcommand{\cL}{\mathcal{L}}
\newcommand{\cM}{\mathcal{M}}

\newcommand{\cP}{\mathcal{P}}

\newcommand{\cR}{\mathcal{R}}
\newcommand{\cS}{\mathcal{S}}
\newcommand{\cT}{\mathcal{T}}

\newcommand{\cW}{\mathcal{W}}

\DeclareMathOperator{\argmin}{arg\,min}

\DeclareMathOperator{\sgn}{sgn}

\renewcommand{\setminus}{\backslash}

\def\ba{\begin{align}}
\def\ea{\end{align}}
\def\bs{\begin{split}}
\def\es{\end{split}}

\begin{document}

\title[Spectral large deviations of sparse random matrices]{Spectral large deviations of sparse random matrices}

\author{Shirshendu Ganguly, Ella Hiesmayr, and Kyeongsik Nam}

\begin{abstract}  
Eigenvalues of Wigner matrices has been a major topic of investigation.
A particularly important subclass of such random matrices, useful in many applications, are what are known as sparse or diluted random matrices, where each entry in a Wigner matrix is multiplied by an independent Bernoulli random variable with mean $p$. Alternatively, such a matrix can be viewed as the adjacency matrix of an  Erd\H{o}s-R\'{e}nyi graph  $\cG_{n,p}$ equipped with  i.i.d. edge-weights.

An observable of particular interest is the largest eigenvalue. 
In this paper, we study the large deviations behavior of the largest eigenvalue of such matrices, a topic that has received considerable attention over the years.
While certain techniques have been devised for the case when $p$ is fixed or perhaps going to zero not too fast with the matrix size, we focus on the case $p = \frac{d}{n}$, i.e. constant average degree regime of sparsity, which is a central example due to its connections to many models in statistical mechanics and other applications. Most known techniques break down in this regime and even the typical  behavior of the spectrum of such random matrices is not very well understood.
So far, results were known only for $\cG_{n,\frac{d}{n}}$  \emph{without} edge-weights   \cite{KS,ganguly1}  and with  \emph{Gaussian}  edge-weights \cite{gn}.

In the present article, we consider the effect of general  weight distributions. More specifically, we consider the  entries whose tail probabilities decay at rate $e^{-t^\alpha}$ with $\alpha>0$, where the regimes $0<\alpha < 2$ and $\alpha > 2$ correspond to tails heavier and lighter than the Gaussian tail respectively.  While in many natural settings the large deviations behavior is expected to depend crucially on the entry distribution, we establish a surprising and rare universal behavior showing that this is not the case when $\alpha > 2.$

In contrast, in the $\alpha< 2$ case, the  large deviation rate function is no longer universal and is given by the solution to a variational problem, the description of which involves a generalization of the Motzkin-Straus theorem, a classical result from spectral graph theory.

As a byproduct of our large deviation results, we also establish the  law of large numbers behavior for the largest eigenvalue, which also seems to be new and difficult to obtain using existing methods. In particular, we show that the typical value of the largest eigenvalue exhibits a phase transition at $\alpha = 2$, i.e. corresponding to the Gaussian distribution.
\end{abstract}

\address{ Department of Statistics,  University of California, Berkeley, CA
94720, USA} 

\email{sganguly@berkeley.edu }

\address{ Department of Statistics,  University of California, Berkeley, CA 94720, USA} 

\email{ella.hiesmayr@berkeley.edu}

\address{ Department of Mathematical Sciences, KAIST, South Korea} 

\email{ksnam@kaist.ac.kr }

\maketitle

\tableofcontents

\section{Introduction}

{
Random graphs are used to model many real world phenomena, like electrical and social networks. One of the oldest and most studied random graph models is the Erd\H{o}s-R\'enyi random graph $\cG_{n,p}$, where there is an edge between any two of $n$ vertices independently with probability $p$. How this graph looks like changes a lot according to the dependence of $p$ on $n$, for instance with high probability the graph is  connected if $p \gg \frac{\log n}{n}$ and  disconnected if $p \ll \frac{\log n}{n}$.    
Often for more effective encoding of physical phenomena, edges in a random network are equipped with random weights, which for instance could denote resistances in an electrical network. 
This leads one to consider  Erd\H{o}s-R\'enyi random graph  $\cG_{n,p}$ with random weights assigned to the edges.  Its adjacency matrix can  be regarded as \emph{sparse} or  \emph{diluted}  random matrices, where each entry of a Wigner matrix is multiplied by an independent Bernoulli random variable with mean $p$.

Some important properties of  the weighted random graph are captured by the spectrum of its adjacency matrix. For instance, the largest eigenvalue and the corresponding eigenvector can be used to measure the spread of diseases on graphs. Regarding this eigenvalue, a lot of research has been devoted towards studying its typical and atypical behavior: the former is concerned with the value and fluctuations that the largest eigenvalue has with high probability, and the latter is about the probability that the eigenvalue deviates significantly its typical value.  A particularly important question in this direction is how the spectral behavior depends on the precise matrix entry distribution, and we  say that the  \emph{universality} phenomenon holds if there is no such dependency in some asymptotical sense.

 In terms of many applications, the regime of $p$ of most interest is the sparse regime (i.e. $p \rightarrow 0$ as $n\rightarrow \infty$). In particular, the constant average degree regime of sparsity $p = \frac{d}{n}$ ($d>0$ is a constant), i.e. the typical number of connections of a single vertex tends to stay constant, arises naturally in  several models in the statistical mechanics. An important distinction from denser matrices is that it is believed that
the universality phenomenon breaks down for  the latter and the spectrum depends rather crucially on the entry distribution. However, despite such predictions, the precise spectral behavior of such random matrices has still been mostly out of reach of the known methods.\\

The present paper is aimed at advancing our understanding in this direction, focusing on the large deviation properties.
Towards this, we study  the large deviation properties of the largest eigenvalue of  the adjacency matrix of the  Erd\H{o}s-R\'enyi  random graph  $\cG_{n,\frac{d}{n}}$ with general  edge-weight distributions $(Y_{ij})_{1\leq i<j\leq n}$. The topic of large deviations of spectral observables of random matrices have attracted much interest over the last few decades leading to a considerable literature, some of which will be reviewed in Section \ref{related results}.

We will be working in the particular setting where the edge-weights $(Y_{ij})_{1\leq i<j\leq n}$ have  tails of the form
\begin{align*}
\mathbb{P}(|Y_{ij}| > t) \approx e^{-t^\alpha},\qquad t>1
\end{align*}
for some $\alpha>0$ (we will keep the notion of $\approx$ somewhat imprecise momentarily). It turns out that $\alpha=2$ is critical in a sense which will become apparent once we state our main results which will address both the $\alpha>2$ (light-tailed) and $\alpha<2$ (heavy-tailed) cases.\footnote{The special case of the Gaussian distribution which can be thought to be corresponding to $\alpha=2$  was studied previously in \cite{gn} by the first and the third named authors.}

 At this point, before embarking on stating our main results precisely in the next section, we choose to highlight some of their key features.
  
Perhaps the most interesting consequence of our main results is the surprising universality  phenomenon for the large deviation of the largest eigenvalue in the light-tailed case $\alpha>2$ where the deviation probabilities (ignoring smaller order terms)  does not depend on $\alpha$ and is also identical to the one for Erd\H{o}s-R\'enyi graphs without edge-weights, which essentially corresponds to  ``$\alpha = \infty$”. Our results also yield new law of large numbers (LLN) results about the largest eigenvalue, which seem to be rather challenging to obtain using previous methods. Interestingly, it turns out that the LLN behavior exhibits a transition as well at $\alpha = 2$.

~

Let us now precisely define our model and state our main results. A short summary of the history of this problem will be provided subsequently.

}

\subsection{Main results}
 
Throughout this paper,
for any symmetric matrix $Z$,  we denote by $\lambda_1(Z) \geq \lambda_2(Z) \geq \cdots \geq \lambda_n(Z)$ its eigenvalues in non-increasing order. We denote by $\cG_{n,p}$ the  Erd\H os-R\'enyi graph  with vertex set $[n]:= \{1,\cdots,n\}$ and edge density $p$,
which will be fixed to be $\frac{d}{n}$ for some constant $d>0$ throughout the article.  

Let  $X =  (X_{ij})_{i,j\in [n]} $ be the adjacency matrix of $\cG_{n, p}$ and $Y = (Y_{ij})_{i,j\in [n]}$ be a standard (symmetric) Wigner matrix, where we assume that $Y_{ji} = Y_{ij}$, $Y_{ii} = 0$ and $\{Y_{ij}\}_{1\leq i< j\leq n}$ are i.i.d random variables. The matrix of interest for us is $Z=X\odot Y,$ i.e.,  $Z_{ij}=X_{ij}Y_{ij}.$  This is a sparse random matrix which can be regarded as the adjacency matrix of a \emph{weighted} random graph, whose underlying random graph is  $\cG_{n,\frac{d}{n}}$ with i.i.d.  edge-weights coming from $\{Y_{ij}\}_{1\leq i< j\leq n}$. Throughout the paper, we  interchangeably use the notation $X$ both for the  Erd\H os-R\'enyi graph   $\cG_{n,\frac{d}{n}}$  as well as its adjacency matrix.

Throughout the paper, matrix entries will have the Weibull distribution:
 {
\begin{definition}
A random variable $W$ has the \emph{Weibull distribution} with a shape parameter $\alpha > 0$ if there exist constants $C_1,C_2>0$ such that for all $t>1$,
\begin{align} \label{weibull}
\frac{C_1}{2} e^{-   t^\alpha} \leq \P\big (W \geq t \big ) \leq \frac{C_2}{2} e^{-   t^\alpha} \quad  \text{ and } \quad  \frac{C_1}{2} e^{-   t^\alpha} \leq \P \big ( W \leq -t \big ) \leq \frac{C_2}{2} e^{-   t^\alpha}.
\end{align}
\end{definition}

The definition is chosen is such a way that \begin{equation}
C_1 e^{- t^\alpha} \leq \P \big ( |W| \geq  t \big ) \leq C_2 e^{-   t^\alpha},
\end{equation} 
which will be notationally convenient later.
}

The cause for assuming a symmetric tail behavior is that a much heavier lower tail causes the spectral norm of the random matrix to be governed by not the largest but the smallest eigenvalue which is a large negative number and hence the largest eigenvalue isn't an interesting object of study anymore.

We now state the main results of the paper.
 
\subsubsection{Light-tailed case, $\alpha>2$.}
Let
\begin{align}\label{typ1}
\lambda^{\textup{light}}_{\alpha  } := 2^\frac{1}{\alpha} \alpha^{-\frac{1}{2}} (\alpha - 2)^{\frac{1}{2}-\frac{1}{\alpha}}   \frac{(\log n) ^\frac{1}{2}}{ (\log \log n )  ^{\frac{1}{2}- \frac{1}{\alpha}}}.
\end{align}
We will shortly state (see Corollary \ref{cor:light_typical}) that the above is the typical value of $\lambda_1(Z)$ for $\alpha > 2$ .

\begin{restatable}{theorem}{lightupper} \label{thm:light_upper}
For any $\delta>0$,
\begin{align*}
\lim_{n \to \infty} - \frac{ \log \P \left  ( \lambda_{1}(Z)\geq (1+\delta) \lambda^{\textup{light}}_{\alpha  } \right )}{\log{n}} = (1+\delta)^2 - 1.
\end{align*}
\end{restatable}

Next, we establish the lower tail large deviation.
\begin{restatable}{theorem}{lightlower} \label{thm:light_lower}
For any $0<\delta<1$,
\begin{align*}
\lim_{n \to \infty}   \frac{1}{\log{n}} \Big( \log \log \frac{1}{\P   ( \lambda_{1}(Z) \leq (1 - \delta) \lambda^{\textup{light}}_{\alpha  }    )} \Big)= 1 - (1-\delta)^2.
\end{align*}
\end{restatable}

A crucial observation is that the upper and lower tail large deviation results establish a sense of \emph{universality} with the rate function not depending on $\alpha$.

Further, as a corollary we have the following law of large numbers.
\begin{corollary} \label{cor:light_typical}
We have
\begin{align*}
\lim_{n\to \infty}   \frac{ (\log \log n )  ^{\frac{1}{2}- \frac{1}{\alpha}}}{(\log n) ^\frac{1}{2}} \lambda_1(Z)= 2^\frac{1}{\alpha} \alpha^{-\frac{1}{2}} (\alpha - 2)^{\frac{1}{2}-\frac{1}{\alpha}}
\end{align*}
in probability.
\end{corollary}
{
Note that the case of no edge-weights can be thought of as corresponding to ``$\alpha=\infty$''. 
Now while this was treated in \cite[Theorem 1.1]{ganguly1} (see Theorem \ref{thm:eigenvalues_bbg} stated later), in some sense the same can be deduced in a limiting sense from  Theorem   \ref{thm:light_upper}.} In fact,  noting that
\begin{align*}
\lim_{\alpha\rightarrow \infty} \lambda^{\text{light}}_\alpha =  \lim_{\alpha\rightarrow \infty} \Big[   2^\frac{1}{\alpha} \alpha^{-\frac{1}{2}} (\alpha - 2)^{\frac{1}{2}-\frac{1}{\alpha}}   \frac{(\log n) ^\frac{1}{2}}{ (\log \log n )  ^{\frac{1}{2}- \frac{1}{\alpha}}}  \Big]=   \frac{(\log n) ^\frac{1}{2}}{ (\log \log n )  ^{\frac{1}{2}}} ,
\end{align*} 
one obtains  \cite[Theorem 1.1]{ganguly1}
by taking $\alpha \rightarrow \infty$ in Theorem    \ref{thm:light_upper}. Note that the large deviation rate function turns out to be not only the same for all light-tailed distribution, i.e. whenever $\alpha > 2$, but also in the limit, when ${}\alpha = \infty$.

\subsubsection{Heavy-tailed case, $\alpha<2$}
Counterpart to \eqref{typ1}, we define
\begin{align}\label{typ2}
 \lambda_{\alpha  }^{\textup{heavy}} :=  (\log n)^\frac{1}{\alpha}.
\end{align}

Stating the upper tail  large deviation needs the following definition.  For $\theta>1,$ let $\phi_\theta : \{2,3,\cdots\} \rightarrow \R$ be defined by 
\begin{align} \label{var}
\phi_{\theta} (k)  := \sup_ {v \in \R^k, v= (v_1,\cdots,v_k)\norm{v}_1=1} \sum_{i,j\in [k], i \neq j} |v_i| ^{\theta}  |v_j|^{\theta}.
\end{align}

The statement involves further considering two sub-cases. 

\begin{restatable}{theorem}{heavyupper} \label{thm:heavy_upper}
Let $\delta>0$.

1. In the case $1<\alpha<2$, let $\beta>2$ be the conjugate of $\alpha$ (i.e. $\frac{1}{\alpha}+\frac{1}{\beta} = 1$). For an integer $k\geq 2$,  define
\begin{align} \label{psi}
\psi _ {\alpha,\delta} (k):=  \frac{k(k-3)}{2}   + \frac{1}{2} (1+ \delta)^\alpha  \phi_{\beta/2}(k)^{1-\alpha } . 
\end{align}
Then,
\begin{equation} \label{rate}
\lim_{n \to \infty} - \frac{ \log \P \left ( \lambda_{1}(Z) \geq (1+\delta)   \lambda^{\textup{heavy}}_{\alpha  }  \right ) }{\log{n}}= \min_{k=2,3,\cdots}  \psi _  {\alpha,\delta}  (k).
\end{equation}

2. In the case $0<\alpha\leq 1$,
\begin{equation}
\lim_{n \to \infty} - \frac{ \log \P \left  ( \lambda_{1} (Z) \geq (1+\delta) \lambda^{\textup{heavy}}_{\alpha  }  \right  ) }{\log{n}}=    (1+\delta)^\alpha-1.
\end{equation}
\end{restatable}

The above theorem shows that once heavy-tailed edge-weights are induced on the \emph{sparse} graph $\cG_{n,\frac{d}{n}}$,   the large deviation rate function for the largest eigenvalue exhibits a phase transition: the  rate function  is a  \emph{piecewise smooth}  function whose behavior changes as $\argmax_k \psi_{\alpha,\delta}(k)$ (i.e. size of the clique needed to have atypically large  $\lambda_1$) varies. 
This is in a sharp contrast to the large deviation result for the  standard  (i.e. dense) Wigner matrices with heavy-tailed entries (i.e. edge-weights are induced on the complete  graph) \cite{augeri2}, where it is proved that there exists a \emph{smooth}  function $I(\delta)$ such that for all $\delta>0$,
\begin{align*}
\mathbb{P}(\lambda_1 > 2+\delta ) \approx e^{-I(\delta ) n^{\alpha/2}}.
\end{align*}

\begin{remark}
We do not have a closed expression for the quantity  $\phi_\theta ( k)$ defined in \eqref{var}, unless $\theta=1$ in which case we know $\phi_1(k) = \frac{k-1}{k}$ ({by the classical Motzkin-Straus theorem} \cite{turan}). While the definition of $\phi_{\theta}(\cdot),$ might seem somewhat unmotivated at this point, it appears quite naturally once we bound $\lambda_1(Z)$ in terms of the `entrywise' $L^p$-(quasi)norm of $Z$. See Section \ref{sec:spectral_prop} for details.
\end{remark}

Next, we obtain the lower tail large deviation.

\begin{restatable}{theorem}{heavylower}  \label{thm:heavy_lower}
For any $0<\delta<1$,
\begin{align*}
\lim_{n \to \infty } \frac{1}{\log n} \left ( \log  \log \frac{1}{\P \left  ( \lambda_{1}(Z) \leq (1 - \delta)  \lambda^{\textup{heavy}}_{\alpha  }  \right  )} \right ) = 1 - (1 - \delta)^\alpha.
\end{align*}
\end{restatable}

Note that unlike for light-tail edge weights, in this case the large deviation rate function \emph{does} depend on $\alpha.$

Finally, as a counterpart to Corollary \ref{cor:light_typical}, we have the following law of large numbers result.
\begin{corollary} \label{cor:heavy_typical}
We have
\begin{align*}
\lim_{n\to \infty}  \frac{\lambda_1(Z)}{(\log n)^\frac{1}{\alpha}} = 1
\end{align*}
in probability.
\end{corollary}

\begin{remark}\label{poly}While for the sake of simplicity as well as brevity we have chosen to work with distributions as in \eqref{weibull}, certain generalizations are not hard to make. E.g., a simple rescaling shows that similar large deviation results hold for more general Weibull distributions where for $t>1$,
\begin{align} \label{eta}
\frac{C_1}{2} e^{-  \eta t^\alpha} \leq \P\big (W  \geq  t \big ) \leq \frac{C_2}{2} e^{- \eta  t^\alpha} \quad  \text{ and } \quad  \frac{C_1}{2} e^{-  \eta t^\alpha} \leq \P \big ( W \leq -t \big ) \leq \frac{C_2}{2} e^{-\eta   t^\alpha}
\end{align}
for some additional scale parameter $\eta>0$.  A straightforward algebra shows that the rate function for  $\lambda_1(Z)$ does not change with $\eta$.

{Further,} a little additional work also allows one to consider an even more general class of distributions given by
\begin{align*}
\frac{C_1}{2}   t^{-c_1} e^{-\eta t^\alpha}  \leq \P\big (W   \geq  t \big ) \leq   \frac{C_2}{2} t^{-c_2} e^{-\eta t^\alpha}  \quad  \text{ and } \quad     \frac{C_1}{2}   t^{-c_1} e^{-\eta t^\alpha}  \leq \P \big ( W \leq  -t \big ) \leq  \frac{C_2}{2} t^{-c_2} e^{-\eta t^\alpha}
\end{align*}
for parameters $c_1,c_2\geq 0$. {Note that while this is not the focus of the paper, in particular, this includes the Gaussian distribution when $\alpha = 2$  and $\eta = c_1=c_2 = \frac{1}{2}$}.

\end{remark}

We next include a brief overview of the literature on large deviations of spectral observables of random matrices.

 {
\subsection{Related results} \label{related results}

Much effort has been devoted to proving whether spectral properties of random matrix ensembles exhibit  certain universality features i.e.  do not depend on the precise entry distribution as the matrix size goes to infinity. For instance, under some moment conditions, the appropriately scaled empirical distribution of the eigenvalues converges to the  Wigner’s semi-circle law \cite{wigner58, arnold67} and the largest eigenvalue lies near the edge of the  semi-circle distribution \cite{juhasz81, furedi81, bai_yin88}. 
On the other hand, large deviation behavior is generally far from universal making this an intriguing research direction which has also witnessed considerable activity.
 Large deviations for the empirical distribution and the largest eigenvalue were first derived for the Gaussian ensembles \cite{freeentropy,adg}. Progress stalled for a while, until a surprising recent result by Guionnet and Husson  \cite{guionnet1}, where the large deviation universality was established for  the largest eigenvalue of `sharp' sub-Gaussian matrices. Their proof relied on the method of spherical integrals.  Together with Augeri \cite{agh}, they also showed that for a more general class of sub-Gaussian matrices, the rate function is universal for small large deviations, but beyond that also depends on other properties of the moment generating function. 

The behavior changes a lot when the entry distribution possesses  tails   \emph{heavier} than Gaussian.  Bordenave and Caputo   \cite{bc}  proved a large deviation result for the empirical distribution of Wigner matrices with stretch-exponential entries whose tails decay at the rate $e^{-t^\alpha}$ with $\alpha \in (0,2)$. The proof  relies on the observation that the atypical behavior is the result of a few atypically large entries. Building on the same idea, Augeri  \cite{augeri2} subsequently obtained a large deviation principle for the largest eigenvalue. In both cases, the  large deviation speed as well as the rate function depend on $\alpha$ and thus on the precise tail behavior of the entries.

We next move on to the results on the spectral properties of the adjacency matrix of  Erd\H{o}s-R\'{e}nyi random graphs $\cG_{n,p}$.
When the edge density $p$ is allowed to decay fast enough in $n$, the aforementioned methods no longer apply and  the spectral behavior  heavily relies on the geometry of the graph. Krivelevich and Sudakov  \cite{KS} first showed that  for almost the entire regime of $p$, except for some threshold values, the largest eigenvalue is typically equal to $(1+o(1)) \max\{ \sqrt{\Delta}, np \}$, where $\Delta$ denotes the maximum degree of $\cG_{n,p}$. This in particular implies that for our case of interest, namely $p = \frac{d}{n}$,  the largest eigenvalue is typically $(1+o(1))\sqrt{\frac{\log n}{\log \log n}}$,  since it is not hard to show that  the largest degree of $\cG_{n,\frac{d}{n}}$ is $(1+o(1)) \frac{\log n}{\log  \log n}$ with high probability. The cases of the threshold values were later treated, for instance the case $np \asymp \log n$ was studied in \cite{adk,tikhomirov20}.

Beyond typical behavior, the study of the large deviation behavior for the spectrum of $\cG_{n,p}$   becomes much more delicate and  requires  new methods. In the dense case (i.e. $p$ fixed in $n$), Chatterjee and Varadhan \cite{CV11, CV12} established a large deviation principle using the powerful graph limit theory and  characterized the rate function in terms of the variational problem. This variational problem  was later analyzed by Lubetzky and Zhao  \cite{LZ-dense}. In the sparse case $p\rightarrow 0$,  the graph limit theory no longer applies.  In the breakthrough work by Chatterjee and Dembo \cite{CD16}, a general framework of \emph{nonlinear large deviations} was developed leading to a similar variational problem. This was later extended \cite{CD18,augeri1} and analyzed in \cite{uppertail_eigenvalue} to obtain large deviations for the largest eigenvalue in the sparsity regime $\frac{1}{\sqrt{n}} \ll p \ll 1 $. 
Finally, the sparsity we will consider in this paper, namely $p = \frac{d}{n}$ was covered in \cite{ganguly1}. It was shown that the large deviation behavior of edge eigenvalues is a consequence of the emergence of vertices of atypically large degree.

Finally, let us review the few existing results about the spectral behavior of the adjacency matrix of  \emph{weighted} Erd\H{o}s-R\'{e}nyi  graphs, or in other words, sparse or diluted  Wigner matrices.
The typical largest eigenvalue of dense graphs (i.e. $p$ fixed in $n$) with general edge-weights, under suitable moment conditions, belongs to the universality class of general Wigner matrices.
In the sparser regime $p\rightarrow 0$, Khorunzhy  \cite{khor01} proved that once   $\frac{\log n}{n} \ll p \ll n^{\beta}$ for some $\beta >0$, the largest eigenvalue is  asymptotically $2 \sqrt{np}$ with high probability and does thus not depend on the precise distribution of the edge-weights.
Recently,
\cite{bbk_dense, tikhomirov20} treated the typical behavior of diluted Wigner matrices,
 where some of the above assumptions are relaxed. However  all results mentioned before are in the regime $p \gg \frac{1}{n}$, which excludes the case $p=\frac{d}{n}$.

However, despite the above advances, significantly less is known about the spectral \emph{large deviation} behavior in the setting of diluted Wigner matrices.
Large deviations for the dense $\cG_{n,p}$ with Gaussian edge-weights  can be deduced from the aforementioned result for  the sub-Gaussian Wigner matrices \cite{agh}. This result together with \cite{guionnet1} show that even for the same weight distribution, the large deviation behavior for the cases  $p=1$ and $p <1$ ($p$ is fixed) are different.
No other regime of $p$, in particular when $p\rightarrow 0$, is covered by previous results. In this direction, only  very recently, the case $p=\frac{d}{n}$ with \emph{Gaussian} edge-weights was treated by the first and last named authors in  \cite{gn}.

Recall that Theorem \ref{thm:light_upper} implies a universal spectral large deviations behavior when $\alpha > 2.$ 
While the large deviation result for the `sharp' sub-Gaussian Wigner matrices \cite{guionnet1} can be considered as a universality result for dense Wigner matrices, it seems Theorem \ref{thm:light_upper} is the first universality result of its kind in the sparse regime. \\
 
}

{The rest of the paper is structured as follows: In Section \ref{sec:idea_proof}, we describe the general proof strategy and the main ingredients that we use. In Section \ref{sec:spectral_prop} we prove our new deterministic results relating the largest eigenvalue of weighted graphs to their $\alpha$-norm and state some classical spectral results that we will use in our proofs. Section \ref{sec:str_erd_ren_graphs} is about the  structural results of sparse Erd\H{o}s-R\'{e}nyi graphs, in particular we prove results about their degree profile and state relevant results about their connectivity structure. After this preparatory work, we move on to prove the large deviation theorems for the cases of light- and heavy-tailed edge-weights in Section \ref{sec:light} and \ref{sec:heavy} respectively. Appendix contains more technical estimates on tails of sums of i.i.d. Weibull random variables.}

Finally, let us define a few notations we use throughout the paper before we move on to explain our proof strategy. 

\subsection{Notations}

For any graph $G$, denote by $V = V(G)$ and $E = E(G)$  the set of vertices and edges in $G$ respectively. For any vertex $v\in V(G)$, define 
$d(v)$ to be the degree of a vertex $v$, and let
$d_1(G)$  be  the maximum degree of  $G$.  For any graph $G = (V,E)$ with a vertex set $V = [n] := \{1,2,\cdots,n\}$,  we write each \emph{undirected} edge joining two vertices $i$ and $j$  with $i<j$ as $(i,j)$. We use the notation $i \sim j$ for two vertices $i,j$ if $i$ and $j$ are connected by an edge.
 
Finally we use the  notation  $G = (V,E,A)$ for denoting a network with an underlying graph $G = (V,E)$ (we abuse the notation)  having $A$ as  a conductance matrix.

\subsection{Acknowledgement}
KN's research was supported by Basic Science Research Program through the National Research Foundation of Korea(NRF) funded by the Ministry of Education\\(2019R1A6A1A10073887). SG was partially supported by NSF grant DMS-1855688, NSF Career grant DMS-1945172, and a Sloan Fellowship. Part of this work was completed when he was participating in the fall
2021 semester program `Universality and Integrability in Random Matrix Theory and Interacting Particle Systems program' at MSRI.

\section{Idea of the proof} \label{sec:idea_proof}
The proofs rely on identifying the correct geometric mechanisms responsible for the largest eigenvalue which we describe next.

\subsection{Light-tailed weights}  \label{heur_typ_val_light}

We start by considering the adjacency matrix $X$ without edge-weights.  In the influential work \cite{KS} it was shown that the largest eigenvalue is determined essentially by the star, incident on the vertex with the largest degree, which is typically $\frac{\log n}{\log \log n}$. Since the largest eigenvalue of a star of degree $\ell$ is equal to $\sqrt{\ell}$, we have  $\lambda_1(X) \approx \frac{(\log n)^\frac{1}{2}}{(\log \log n) ^{\frac{1}{2}}}.$ 
It was subsequently established in \cite{ganguly1} that in this case, even in the large deviations regime, the largest eigenvalue is essentially determined by  an atypically large degree vertex.

 Now, let us consider the case when the edges are equipped with light-tailed edge-weights. Owing to the lightness of the tail,
one might naturally expect  that vertices with degree of order $\frac{\log n}{\log \log n}$ will continue to be the determining structure, but the existence of weights does add another element of randomness that can increase the largest eigenvalue. This calls for the need to balance the fact that there are more stars of lower degree, while high degree stars are more likely to have a big largest eigenvalue necessitating the estimation of the contributions from vertices of degree close to $\gamma \frac{\log n}{\log \log n}$ for $0<\gamma < 1$.

Using a binomial tail estimate, one can deduce that  the probability that a vertex has degree close to  $\gamma \frac{\log n}{\log \log  n}$ is roughly $n^{-\gamma}$ and hence approximately there are around $n^{1-\gamma}$ vertices of that degree.  Because $\cG_{n,\frac{d}{n}}$ is sparse, it is likely that a constant proportion of these vertices has distinct neighborhoods with no edges present within those neighborhoods and hence these high-degree vertices and their neighborhoods induce vertex-disjoint stars. 
The next step is to consider the contribution from the edge-weights. Here we need two ingredients. First, the fact that the largest eigenvalue of a weighted star is the square root of the sum of the squares of the weights. The second fact we use is that for light-tailed weights, the probabilistically optimal way to obtain a large squared sum of the weights is by all of them being uniformly large. This allows us to deduce that for any weighted star $S$ of degree $k$,
\begin{align} \label{00}
 \mathbb{P}( \lambda_1(S) > t ) \approx e^{-t^\alpha k^{1-\frac{\alpha}{2}}}.
\end{align}
Using this, 
 {to find a typical value $t$ of the largest eigenvalue of such  collection of stars, the probability in \eqref{00} with $k=\gamma \frac{\log n}{\log \log n}$  should balance out the number of such  stars which is roughly $n^{1-\gamma}$.   A further optimization over $\gamma$ indeed indicates that typical value of the largest eigenvalue is close to $\lambda^{\textup{light}}_{\alpha}$ (defined in \eqref{typ1}). 
 
The proof of the upper bound in Theorem \ref{thm:light_upper} now relies on establishing the above heuristic even in the large deviations regime. 
A key subtle distinction arises from the fact that the large deviation of $\lambda_1$ may induce atypical behavior of both the degrees as well as the edge weights which has to be taken into consideration as well.\\

The proof of the lower bound is much simpler and consists of ensuring that there is a star of degree $\gamma_\delta \frac{\log n}{\log \log n},$ for some $\delta$ dependent constant $\gamma_\delta,$ with large enough weights on the edges.

\subsection{Heavy-tailed weights} \label{heur_typ_val_heavy} 
The results of \cite{gn} indicate that typically the largest eigenvalue of $\cG_{n,\frac{d}{n}}$   with \emph{Gaussian} weights is determined by the maximal edge-weight in absolute value. Since the latter increases as the tail of individual random variables becomes heavier, this suggests that the same mechanism persists when the edge-weights have heavier tails (i.e. $\alpha<2$).

Note that there are on average $\frac{dn}{2}$ edges present in the graph $\calG_{n,\frac{d}{n}}$. Thus, the probability that the  largest entry is greater in absolute value than $t$ is 
\begin{align} \label{600}
1- \P \big(\max_{(i,j) \in E(X)}|Z_{ij}| < t\big) & \approx   1- (1-e^{-  t^\alpha})^{\frac{dn}{2}} \approx  \frac{dn}{2}e^{-t^\alpha}.
\end{align}
This shows that the typical value of the maximum entry (in absolute value) is  $ \lambda_{\alpha }^{\textup{heavy}}$ defined in \eqref{typ2}.

However, unlike typical behavior, it turns out that large deviations, i.e., $\{\lambda_1(Z) \geq (1+\delta)\lambda^{\textup{heavy}}_\alpha  \} $, is achieved by the emergence 
 of a clique with high  edge-weights on it. One can  first  estimate the probability that $\cG_{n,\frac{d}{n}}$, contains a clique of size $k$ for each integer $k\geq 2$. Next, one estimates the probability that the edge-weights on the clique are high.  One of our main results is that we identify how to induce these high edge-weights  in the most efficient way, which  turns out to   involve the  variational problem \eqref{var}. 
 
The optimal clique size depends on both $\alpha$ and $\delta$. In particular, it is $2$ when $\alpha \leq 1,$ i.e., large deviations is dictated by the existence of an atypically  large edge-weight.
 
 Again as before, while the proof of the upper bound involves several technical ingredients making the above heuristics precise, the lower bound is obtained rather quickly by planting a clique of an appropriate size with high edge weights.
 
We end this section with a brief overview of the variational problem described in \eqref{var} which we believe is of independent interest.

\subsection{$\alpha-$norm generalization of the {Motzkin-Straus} theorem} \label{sec 2.3}
When the edge-weights are given by heavy-tailed Weibull distributions ($0<\alpha<2$), we study the spectral behavior by relying on a new result relating the largest eigenvalue and the entrywise $L^\alpha$-(quasi)norm of any symmetric matrix $A$, which we define by
\begin{align*}
\norm{A}_{\alpha} :=  \Big(\sum_{i,j} |a_{ij}|^\alpha \Big)^{1/\alpha}.
\end{align*}
In particular we show that for any $0 < \alpha < 2$ and any integer $k \geq 2$,
there exists an explicit and \emph{sharp}  constant $C(\alpha,k)>0$ such that for any network $G = (V,E,A)$ with maximum clique size $k$, 
\begin{align} \label{ms theorem}
\lambda_1(A) \leq   C(\alpha, k)  \norm{A}_{\alpha}.
\end{align} 
The constant $C(\alpha,k)$ is expressed in terms of the function $\phi_\alpha(k)$ defined in \eqref{var} {(see Proposition \ref{prop:spectral_bound_L_p} for details). Moreover, it turns out that $C(\alpha,k)$ does not depend on $k$ when $0< \alpha \leq 1$. 
The special case  $\alpha=2$ had been studied earlier where the bound counterpart to \eqref{ms theorem} reads as
\begin{align} \label{p=2}
\lambda_1(A)  \leq  \sqrt{\frac{k-1}{k}}    \norm{A}_{2}, 
\end{align}
and can be obtained as a straightforward consequence of the Motzkin-Straus theorem (see \cite[Proposition 3.1]{gn}).
}

We finish this discussion with a brief comparison of the heavy tailed case to the Gaussian case $\alpha=2.$ For the latter, it is shown in \cite{gn} that if we set
\begin{align}\label{keydef}
\bar{\psi}_\delta(k) :=    \frac{k(k-3)}{2} + \frac{1+\delta}{2} \frac{k}{k-1}
\end{align}
for integers $k\geq 2$, then
\begin{align} \label{112}
\lim_{n\ri}   -\frac{1}{\log n} \log  \mathbb{P}\left ( \lambda_1(Z) \geq \sqrt{2(1+\delta) \log n} \right ) =   \min_{k=2,3,\cdots} \bar{\psi}_\delta(k).
\end{align} 
{
While, $\lim_{\delta \to \infty} \argmin_{k\in \mathbb{N}_{\geq 2}} \bar{\psi}_\delta(k)    = \infty,$
it turns out that  when $1<\alpha<2$,  the  function $\phi_{\beta/2} (k)$ that appears in   \eqref{psi} becomes constant for large $k$ (see \eqref{uniform} in  Lemma \ref{lemma 02}), which implies that $\argmin_{k\in \mathbb{N}_{\geq 2}} \psi_{\alpha,\delta}(k)  $ stays bounded in $\delta$. 
Further, our proof indicates that the optimal way for the largest eigenvalue to be greater than $(1+\delta)\lambda_\alpha$  is to have a clique of size $\argmin_{k\in \mathbb{N}_{\geq 2}}  \psi _ {\alpha,\delta} (k)$ (in particular, 2, which is just a single edge, when $0<\alpha\leq 1$) in  the random graph $X = \cG_{n,\frac{d}{n}}$, and then have high valued edge-weights on  this clique. A counterpart result in the Gaussian case was proven in \cite{gn}, where it was shown that the corresponding clique size is given by $\argmin_{k\in \mathbb{N}_{\geq 2}} \bar{\psi}_\delta(k).$ Thus, by the above discussion, the governing clique size for $\alpha<2$ stays bounded as the deviation increases in contrast to the Gaussian case where the same goes to infinity.
}

}

\section{Spectral properties of weighted graphs} \label{sec:spectral_prop}

In this section we collect various bounds and other facts about spectral properties of graphs that will be in force throughout the article.

We start with, as outlined in the idea of proof in Section \ref{sec:idea_proof}, a key ingredient in our proofs for heavy-tailed weights in the form of a new deterministic bound on the largest eigenvalue in terms of the `entry-wise' $L^p$-(quasi)norm of the matrix, generalizing the classical Motzkin-Straus theorem \cite{turan} corresponding to $p=2$ case.

\subsection{Spectral norm and $L^p$-(quasi)norm}\label{swg}

For $p>0$,
we denote by $\norm{A}_{p}$ the entry-wise $L^p$-(quasi)norm of the symmetric matrix $A$:
\begin{align*}
\norm{A}_{p} :=   \Big(\sum_{1\leq i,j\leq n} |a_{ij}|^p\Big)^{1/p}.
\end{align*}
To state our bound for the largest eigenvalue of the symmetric matrix $A$ in terms of $\norm{A}_{p}$, we first recall the following auxiliary function which appeared in the statement of Theorem \ref{thm:heavy_upper}:
For  $\theta > 1$ and any integer $k\geq 2$, let 
\begin{align} \label{phi}
\phi_\theta (k)  := \sup_ {f = (f_1,\cdots,f_k): \norm{f}_1=1} \sum_{i,j\in [k], i \neq j} |f_i| ^\theta  |f_j|^\theta.
\end{align}
We will assume without loss of generality that the vector $f$ appearing above is non-negative.

\begin{proposition} \label{prop:spectral_bound_L_p}
Suppose that $1 < p  < 2$ and let $k \geq 2$ be an integer. Then, for any network $G = (V,E,A)$  such that the maximum size of clique contained in $G$ is $k$,
\begin{align} \label{case1}
\lambda_1(A) \leq    \phi_{ \frac{p}{2(p-1)}} (k)^{\frac{p-1}{p}}  \norm{A}_{p}.
\end{align}
In the case $0<p\leq 1$, for any network $G = (V,E,A)$,
\begin{align} \label{case2}
\lambda_1(A) \leq     2^{-\frac{1}{p}} \norm{A}_{p}.
\end{align}
\end{proposition}

{Before proving this proposition, we state some useful  lemmas.}
 {First, the next lemma identifies the structure that leads to the equality in the above expressions. We will need those characterizations when planting the structures that lead to an atypically large eigenvalue in the  heavy-tailed edge-weights case.}

\begin{lemma} \label{equality_L_p}
Assume that $1<p < 2$.
Then, for any  integer $k\geq 2$, there exist $k_1,k_2 \geq 0$ with $k_1+k_2\leq k$ and  $x,y\geq 0$ such that if $G  = (V,E) $ is a clique with  $V  = [k]$ and $A = (a_{ij})_{i,j \in [k]}$ is  a block matrix given by
\begin{align} \label{block}
a_{ij}=
\begin{cases}
x^2 \qquad  &i\neq j, i,j\in \{1,\cdots,k_1\} =: V_1 , \\
y^2  \qquad &i\neq j,  i,j\in \{k_1+1,\cdots,k_1+k_2\} =: V_2,\\
xy \qquad & i\in V_1, j\in V_2 \text{ or } i\in V_2,j\in V_1,\\
0 \qquad & \text{otherwise},
\end{cases}
\end{align}
 then  $A$ satisfies the equality in \eqref{case1}.

\noindent
In the case $0<p\leq 1$, the equality in \eqref{case2} holds when $A$ is the adjacency matrix of a clique of size 2, i.e. a graph consisting of a single edge.

{ Note that this matrix is not unique, as for any tuple $(k_1,k_2,x,y)$ and any constant $c$, the tuple $(k_1,k_2,cx,cy)$ also satisfies the equality.}
 
\end{lemma}

{We defer the proof of this lemma to the end of this section.}
 To prove Proposition \ref{prop:spectral_bound_L_p}, we rely on an alternative characterization of $\phi_\theta(k)$, which is given in the following lemma. It turns out that  $\phi_\theta(k)$  is equal to the supremum of the same objective function {(which we call $\widehat{\phi}_\theta(k)$)} over all graphs  $G$ whose \emph{maximum clique size}  is $k$.  

\begin{lemma} \label{same}

For  $\theta >  1$ and an integer $k\geq 2$, define
\begin{align} \label{varphi}
\widehat{\phi}_\theta (k) := \sup_{G=(V,E) } \sup_{f = (f_1,\cdots,f_{|V|}): \norm{f}_1=1 } \sum_{i,j\in [|V|], i\sim j} |f_i |^\theta | f_j|^\theta.
\end{align}  
{Here, the first supremum is taken over all graphs $G$ whose maximum clique size is $k$. Recall that in the summation,} $i\sim j$ means that vertices $i$ and $j$ are connected by an edge. Then, we have that
\begin{align*}
\widehat{\phi}_\theta  (k)= \phi_\theta(k).
\end{align*}

\end{lemma}

We defer the proof to the end of this section.
Given this lemma, one can  conclude the proof of  Proposition \ref{prop:spectral_bound_L_p}.
\begin{proof}[Proof of Proposition \ref{prop:spectral_bound_L_p}]
Let $V = [n] $ and $a_{ij}$s denote the edge-weights.
By the variational characterization of the largest eigenvalue,
\begin{align} \label{variational}
\lambda_1(A) = \sup_{ \norm{f}_2  =1} \sum_{i\sim j}  a_{ij} f_if_j.
\end{align}
We now consider the different ranges of $p$. {In the case $1<p<2$, we apply  H\"older's inequality to bound the above quantity,  whereas in the case $0<p\leq 1$,  we simply use the monotomicity of $\ell^p$ norms.}

\textbf{Case 1: $\bm{1<p <  2}$.}
Setting $q=\frac{p}{p-1} >2$ to be the conjugate of $p$,  by H\"older's inequality,
\begin{align*}
 \sum_{i\sim j}  a_{ij} f_if_j \leq  \Big (  \sum_{i\sim j}  | a_{ij}|^p    \Big)^{\frac{1}{p}}  \Big( \sum_{i\sim j}   | f_i|^q|f_j|^q \Big)^{\frac{1}{q}} .
\end{align*}
By the definition of $\widehat{\phi}_{\frac{q}{2}}$ and since {$\phi_{ \frac{q}{2} }= \widehat{\phi}_{ \frac{q}{2} }$} (see Lemma \ref{same}),  for any vector $f$ such that $ \norm{f}_2  =1$,
\begin{align*}
\sum_{i\sim j}   | f_i|^q|f_j|^q \leq  \widehat{\phi}_{ \frac{q}{2} } (k)   =  \phi_{ \frac{q}{2} } (k)
\end{align*} 
(note that in \eqref{varphi}, supremum is taken over $\norm{f}_1=1$). Therefore,  we have 
\begin{align*}
\lambda_1(A)   \leq  \phi_{ \frac{q}{2}} (k)^{\frac{1}{q}} \norm{A}_p  =  \phi_{ \frac{p}{2(p-1)}} (k)^{\frac{p-1}{p}} \norm{A}_p.
\end{align*}  

\textbf{Case 2: $\bm{0<p\leq 1}$.}
Since $|f_if_j| \leq \frac{1}{2}$ for any $i\neq j,$ whenever $\norm{f}_2=1$,   by the monotonicity of $\ell^p$ norms, 
\begin{align*}
\lambda_1(A)  = 2 \sup_{ \norm{f}_2  =1} \sum_{i<j, i\sim j}  a_{ij} f_if_j \leq    \sum_{i<j, i\sim j}  |a_{ij} | \leq \Big( \sum_{i<j, i\sim j}  | a_{ij} |^p \Big )^{\frac{1}{p}} = 2^{-\frac{1}{p}} \Big ( \sum_{i\sim j}  | a_{ij} |^p   \Big)^{\frac{1}{p}}.
\end{align*} 
  
\end{proof}

We now establish some useful properties of the function $\phi_\theta$.

\begin{lemma} \label{lemma 02} Let $\theta > 1$. Then,
\begin{enumerate}
\item \label{phi:max_is_blocks} For each $k\geq 2$, there exists a $k$-dimensional vector $f$ of the form $(x,\cdots,x,y,\cdots,y,0,\cdots,0)$ which attains the maximum of $\phi_\theta(k)$  in  \eqref{phi}.  In other words, there exist $k_1,k_2\geq 0$ with $k_1+k_2\leq k$ and $x,y\geq 0$ such that $k_1 x + k_2 y=1$ and
\begin{align*}
\phi_\theta (k)  =  \sum_{i,j\in [k], i \neq j} f_i ^\theta  f_j^\theta
\end{align*}
holds with $f_1=\cdots= f_{k_1}  = x$,  $f_{k_1+1} = \cdots = f_{k_1+k_2} = y$ and $f_{k_1+k_2+1}  = \cdots  = f_k = 0$.
 {
\item \label{phi:general_bound} For any $k \geq 2$, 
\begin{equation} 
\phi_\theta(k) \leq \Big(\frac{2\theta-2}{2\theta-1}  \Big)^{2\theta-2}- \Big(\frac{2\theta-2}{2\theta-1}  \Big)^{2\theta-1}.
\end{equation}
}
\item \label{phi:bound_small_k} For  any
$k  \leq  \frac{2\theta-1}{2\theta-2}$,
\begin{align} \label{912}
 \phi_\theta(k) =  \frac{1}{k^{2\theta-2}} - \frac{1}{k^{2\theta-1}}.
\end{align}
In addition,
\begin{equation} \label{clique 2} 
\phi_\theta(2) = \frac{1}{2^{2 \theta - 1}}.
\end{equation}
\item \label{uniform} $ \phi_\theta(k)$ is non-decreasing and becomes constant for large enough $k$.
\end{enumerate}

\end{lemma}

We give a brief interpretation  of this lemma. The statement
 \eqref{phi:max_is_blocks} implies that for any $k\geq 2$, maximum of $\phi_\theta(k)$ in  \eqref{phi} is attained at the vector $f$ with at most two distinct non-zero elements. {\eqref{phi:general_bound} states a general upper bound for the function $\phi_\theta(k)$. \eqref{phi:bound_small_k} states that for   $k  \leq  \frac{2\theta-1}{2\theta-2}$, the maximum of $\phi_\theta(k)$ is attained at the $k$-dimensional vector $(\frac{1}{k},\cdots,\frac{1}{k})$. 
The fact that $\phi_\theta(k)$ becomes constant for large $k$, stated in \eqref{uniform}, is a crucial ingredient in our analysis.}
 
\begin{proof}[Proof of Lemma \ref{lemma 02}] $\empty$

\textbf{Proof of  \eqref{phi:max_is_blocks}.}
{We assume, without loss of generality, that the supremum is taken over all $k$-tuples} $(f_1,\cdots,f_k)$ with $\sum_{i=1}^k  f_i=1$ and $f_i\geq 0$. 
Since the collection of  such $k$-tuples is a compact set, the function   $(f_1,\cdots,f_k) \mapsto \sum_{i,j\in [k], i\neq j} f_i^\theta f_j^\theta $ over this set attains its maximum. {Note that there may be several $k$-tuples which attain the maximum, and in this case we arbitrarily choose one of them.} We further assume, without loss of generality, that for some integer $1\leq \ell \leq k$, the maximum is attained in the interior of the $(\ell-1)$-dimensional simplex $f_1+\cdots + f_\ell=1$ with $f_{\ell+1} =\cdots = f_k = 0$ (i.e. $0<f_1,\cdots,f_\ell<1$). {By the Lagrange multiplier theorem \cite[Chapter 14.8]{lag}, which states that  the gradient of the objective function at a local extreme point is a scalar multiple of the gradient of the constraint function,} applied to our maximization problem on this simplex, setting $s := f_1^\theta + \cdots  + f_\ell^\theta$, we have
\begin{align*}
(s - f_1^\theta ) f_1^{\theta-1} =  \cdots = (s - f_\ell^\theta ) f_\ell^{\theta-1}.
\end{align*}
Defining $g_i:=f_i^{\theta-1}$ and $\bar{\theta}:= \frac{2\theta-1}{\theta-1}>1 $, 
this implies that {the quantities
$
\frac{g_i^{\bar{\theta}}- g_j^{\bar{\theta}} }{g_i-g_j}
$
are all equal to $s$ for any $1\leq i<j\leq \ell$ for which $g_i\neq g_j$}. By the mean value theorem applied to the convex function $x \mapsto x^{\bar{\theta}}$, one can deduce that there are at most two distinct (non-zero) values that $g_i$s (and thus $f_i$s)  for $1\leq i \leq \ell$ can take.  {This concludes the proof of \eqref{phi:max_is_blocks}.}

\textbf{Proof of \eqref{phi:general_bound}.}
Applying H\"older's inequality  and using that $\sum_{i=1}^k f_i  = 1$, 
\begin{align} \label{01}
\sum_{i=1}^k f_i^\theta \leq   \left ( \sum_{i=1}^k f_i  \right )^ \frac{\theta}{2\theta-1}  \left ( \sum_{i=1}^k f_i^{2\theta} \right )^ \frac{\theta-1}{2\theta-1} = \left ( \sum_{i=1}^k f_i^{2\theta} \right )^ \frac{\theta-1}{2\theta-1}.
\end{align}
Thus, setting $r :=   \sum_{i=1}^k f_i^{2\theta}$, we have
\begin{align} \label{353}
\sum_{i,j\in [k], i \neq j} f_i^\theta f_j^\theta  = \left (\sum_{i=1}^k f_i^\theta  \right )^2 -  \sum_{i=1}^k f_i^{2\theta}\leq r^{\frac{2\theta-2}{2\theta-1}} - r.
\end{align}
{The function $t \mapsto t^{\frac{2\theta-2}{2\theta-1}} - t$ is increasing} on $\Big (0, \big (\frac{2\theta-2}{2\theta-1} \big)^{2\theta-1} \Big )$ and decreasing on $\Big ( \big (\frac{2\theta-2}{2\theta-1} \big )^{2\theta-1}, \infty \Big )$.
 Thus, 
\begin{align} \label{02}
\sum_{i,j\in [k], i \neq j} f_i^\theta f_j^\theta  \leq   \left ( \frac{2\theta-2}{2\theta-1}  \right )^{2\theta-2}- \left ( \frac{2\theta-2}{2\theta-1}  \right )^{2\theta-1}.
\end{align}

Note that the equality above may not be attained in general. In fact, if the equality is attained, then by the equality condition in the H\"older's inequality \eqref{01}, $f_1=\cdots=f_m = \frac{1}{m}$ and $f_{m+1} = \cdots  = f_k = 0$ for some integer $1\leq  m\leq  k$. Also, in order that \eqref{02} becomes an equality,  $  \sum_{i=1}^k f_i^{2\theta} = r = \big (\frac{2\theta-2}{2\theta-1} \big )^{2\theta-1} $. Since $  \sum_{i=1}^k f_i^{2\theta} =m\cdot (\frac{1}{m})^{2\theta}  =   (\frac{1}{m})^{2\theta-1}$,  the equality in \eqref{02} is  possible only when  $\frac{2\theta-1}{2\theta-2}$ is a positive integer.

\textbf{Proof of \eqref{phi:bound_small_k}.}
We first prove the first part of the statement. By  H\"older's inequality,
\begin{align} \label{holder}
1 = \sum_{i=1}^k f_i \leq \Big  (\sum_{i=1}^k f_i^{2\theta}\Big)^{\frac{1}{2\theta}} k^{1-\frac{1}{2\theta}}.
\end{align}
Hence, if $k\leq \frac{2\theta-1}{2\theta-2}$, then $ r= \sum_{i=1}^k f_i^{2\theta} \geq \left ( \frac{1}{k} \right )^{2\theta-1} \geq \big ( \frac{2\theta-2}{2\theta-1} \big)^{2\theta-1}$. Thus, recalling that the function $t \mapsto t^{\frac{2\theta-2}{2\theta-1}} - t$ is decreasing on $\Big(\big(\frac{2\theta-2}{2\theta-1}\big )^{2\theta-1},\infty \Big)$ and using $r \geq (\frac{1}{k})^{2\theta  - 1}$, 
\begin{align*}
\sum_{i,j\in [k], i \neq j}  f_i^\theta f_j^\theta \overset{\eqref{353}}{\leq}   r^{\frac{2\theta-2}{2\theta-1}} -r \leq  \frac{1}{k^{2\theta-2}} - \frac{1}{k^{2\theta-1}}.
\end{align*}
One can also deduce that the maximum of $\phi_\theta(k)$ is attained when  $f_1=\cdots = f_k = \frac{1}{k}$. To see this, by the above inequality, if $f$ attains the maximum, then $r  =  (\frac{1}{k})^{2\theta  - 1}$, which implies that $f$ satisfies the equality in   \eqref{holder}. By the equality condition in the H\"older's inequality, all the $f_i$s are same and thus  $f_1=\cdots = f_k = \frac{1}{k}$.

The second statement \eqref{clique 2} follows from the fact that $|xy| \leq \frac{1}{4}$ whenever $|x|+|y|=1$.

\textbf{Proof of \eqref{uniform}.}
 It is straightforward to observe that  $\phi_\theta(k)$ is non-decreasing in $k$ since as $k$ increases, the supremum is taken over a larger class of vectors $f$. We now show that $ \phi_\theta(k)$ becomes constant for large enough $k$. By considering $f_1=\cdots = f_k = \frac{1}{k}$, we have
\begin{align*}
\phi_\theta( k) \geq   \frac{1}{k^{2\theta-2}} - \frac{1}{k^{2\theta-1}}.
\end{align*}
Since $\phi_\theta(k)$ is non-decreasing in $k$ and the function $t \mapsto  \frac{1}{t^{2\theta-2}} - \frac{1}{t^{2\theta-1}}$ is decreasing for any large enough $t$, we have  that for large enough $k$,
\begin{align*}
\phi_\theta( k) \geq  \sup_n \Big (  \frac{1}{n^{2\theta-2}} - \frac{1}{n^{2\theta-1}} \Big).
\end{align*}
If the equality holds for all large enough $k$, we are done.
Otherwise, there is $K_0>0$ such that
\begin{align}\label{355}
 \phi_\theta(k) > \sup_n  \Big ( \frac{1}{n^{2\theta-2}} - \frac{1}{n^{2\theta-1}} \Big)
\end{align}  for  all $k \geq K_0$. We show that this  implies that the the support of the maximizer $f$ of $\phi_\theta(k)$  in \eqref{phi} is uniformly bounded in $k$.

 By the  statement \eqref{phi:max_is_blocks} of this lemma, for any $k$, the maximum of the objective function $\phi_\theta(k)$ is attained at some vector $f =  (f_i)_i$ with $f_1 = \cdots = f_{k_1} = x$, $f_{k_1+1}= \cdots = f_{k_1+k_2} = y$, and $f_{k_1+k_2+1} = \cdots = f_k  =  0$ for some $k_1,k_2\geq 0$ with $k_1+k_2 \leq k$ and $x,y\geq 0$. Thus, if the support of $f$ is not bounded in $k$ (without loss of generality we assume $k_2\rightarrow \infty$), then for sufficiently small $\iota>0$,
\begin{align*}
\phi_\theta(k) = 2{k_1 \choose 2}x^{2\theta } +  2{k_2 \choose 2}y^{2\theta }  +2 k_1k_2  x^\theta y^\theta & \leq  2 {k_1 \choose 2}\frac{1}{k_1^{2\theta}}  + \Big[ 2 {k_2 \choose 2} \frac{1}{k_2^{2\theta}}  +2 k_2   \frac{1}{k_2^\theta}  \Big] \\
& \leq  \frac{1}{k_1^{2\theta-2}} - \frac{1}{k_1^{2\theta-1}}  + \iota \\
& \leq  \sup_n  \Big ( \frac{1}{n^{2\theta-2}} - \frac{1}{n^{2\theta-1}} \Big) + \iota \overset{\eqref{355}}{<} \phi_\theta(k),
\end{align*}  
{
where we used $k_1 x^\theta \leq k_1 x \leq 1, y\leq \frac{1}{k_2}$ in the first inequality and $k_2\rightarrow \infty,\theta>1$ in the second inequality. The final RHS bound yields a contradiction.}  Hence, we conclude that  the support {of the vector $f$ which maximizes} the objective function  $\phi_\theta(k)$  is uniformly bounded in $k$.  This implies that $\phi_\theta(k)$ becomes constant for large $k$.
\end{proof}
 
Using this lemma, one can  establish Lemma \ref{equality_L_p} which provides the equality condition of the inequalities in Proposition \ref{prop:spectral_bound_L_p}.

\begin{proof}[Proof of Lemma \ref{equality_L_p}]
 
{Let us first consider the case  $1<p <  2$.} Let $q > 2$ be the conjugate of $p$.  By   Lemma \ref{lemma 02},  there exist $k_1,k_2\geq 0$ with $k_1+k_2\leq k$  and $x,y\geq 0$ such that the vector $f = \left (f_1,\cdots,f_{[k]} \right )$ defined by
  \begin{align*}
  f_i = 
  \begin{cases}
  x \qquad &i \in   \{1,\cdots,k_1\} =: V_1 , \\
  y \qquad &i \in \{k_1+1,\cdots,k_1+k_2\} =: V_2, \\
  0 \qquad & \text{otherwise}
  \end{cases}
\end{align*}
satisfies  $\norm{f}_1=1$ and
\begin{align*}
\phi_{\frac{q}{2}} (k)  =  \sum_{i,j\in [k], i \neq j} f_i ^\frac{q}{2}  f_j^ \frac{q}{2}.
\end{align*}

Now define the vector $\tilde{f} = \left ( \tilde{f}_1,\cdots, \tilde{f}_{k} \right )$ by setting $\tilde{f}_i = \sqrt{f_i}$ so that $\lVert \tilde{f} \rVert_2 = \norm{f}_1=1$.
Next, define the $k \times k $ matrix $A = (a_{ij})_{i,j\in [k]}$ by
\begin{align*} 
a_{ij}:=
\begin{cases}
x^\frac{1}{p-1} \qquad  &i\neq j,i,j\in   V_1 ,\\
y^\frac{1}{p-1} \qquad &i\neq j,  i,j\in  V_2,\\
x^\frac{1}{2(p-1)} y^\frac{1}{2(p-1)}  \qquad & i\in V_1, j\in V_2 \text{ or } i\in V_2,j\in V_1,\\
0 \qquad & \text{otherwise}.
\end{cases}
\end{align*}
Note that we defined $A$ so that $a_{ij}^p = f_i ^\frac{q}{2}  f_j^ \frac{q}{2} =  \tilde{f}_i^q \tilde{f}_j^q$ for $i\neq j$. This implies that
\begin{align} \label{03} 
\norm{A}_{p} =  \Bigg(  \sum_{i,j\in  [k] , i\neq j}   \tilde{f}_i^q \tilde{f}_j^q \Bigg )^\frac{1}{p} =   \Bigg(  \sum_{i,j\in  [k] , i\neq j}  f_i^{\frac{q}{2}}   f_j^{\frac{q}{2}}   \Bigg)^\frac{1}{p}  = \phi_{\frac{q}{2}} (k)^\frac{1}{p} = \phi_{\frac{p}{2(p-1)}}(k)^\frac{1}{p}
\end{align}
and
\begin{align} \label{04}
\sum_{i,j\in  [k] , i\neq j}  a_{ij} \tilde{f}_i \tilde{f}_j =   \sum_{i,j\in  [k] , i\neq j}   \tilde{f}_i^q \tilde{f}_j^q  =     \sum_{i,j\in  [k] , i\neq j}   f_i^{\frac{q}{2}}   f_j^{\frac{q}{2}}  =  \phi_{\frac{q}{2}} (k) = \phi_{\frac{p}{2(p-1)}}(k) .
\end{align}
By the variational characterization of the largest eigenvalue (recall that $\lVert \tilde{f} \rVert_2 =1$) and Proposition \ref{prop:spectral_bound_L_p}, 
\begin{align*}
 \phi_{\frac{p}{2(p-1)}}(k) 
 \overset{\eqref{04}}{=} 
 \sum_{i,j\in  [k] , i\neq j}  a_{ij} \tilde{f}_i \tilde{f}_j  
\leq
 \lambda_1(A)
\leq
  \phi_{\frac{p}{2(p-1)}}(k)^\frac{p-1}{p} \norm{A}_{p} 
\overset{\eqref{03}}{=}
 \phi_{\frac{p}{2(p-1)}}(k),
\end{align*}
which establishes that equality holds in \eqref{case1}.

{Now, let us consider the case $0<p\leq 1$}. Note the largest eigenvalue of a network with a single edge is nothing other than the edge-weight which we call $a$.  Since  $\norm{A}_{p} = 2^{\frac{1}{p}}a$ for a $1\times 1$ matrix $A = (a)$,  we obtain the equality in \eqref{case2}.
\end{proof}

 We now establish Lemma \ref{same}, which claims that our two characterizations of $\phi_\theta$ and $\varphi_\theta$ are equivalent.
\begin{proof} [Proof of Lemma \ref{same}]
Assuming that $G$ is not a clique of size $k$, we can choose two vertices $v_1$ and $v_2$ that are not connected by an edge  in $G$. Without loss of generality, we assume  that $\sum_{i\sim v_1}  f_i ^\theta\geq \sum_{j \sim v_2}  f_j^\theta$. Since 
\begin{align*}
\sum_{ i\sim j} f_i^\theta f_j^\theta =\Big ( \sum_{i \sim v_1} f_i^\theta \Big  ) f_{v_1}^\theta + \Big ( \sum_{j \sim v_2} f_j ^\theta \Big  ) f_{v_2} ^\theta + \sum_{ i,j\neq v_1,v_2, i\sim j } f_i^\theta f_j^\theta,
\end{align*}
the objective function does not decrease when we move the weight from $v_2$ to $v_1$ by replacing $f=(\cdots, f_{v_1},\cdots, f_{v_2},\cdots)$ by $ f^{(1)}=( \cdots, f_{v_1}+f_{v_2}, \cdots, 0,\cdots ) $. 
 This follows from the fact that  for $\theta >  1$, if $a\geq b\geq 0$, then
$ \max_{x+y=s,x,y\geq 0} ax^\theta + by^\theta    = a \cdot s^\theta + b \cdot 0^\theta =  as^\theta.$

After removing the zero at $v_2$, we obtain a new vector $f ^{(1)}$ on the new graph $G_1$ obtained by a  deletion of the vertex $v_2$ and the edges incident on it. We repeat this procedure to get a sequence of vectors $f  ^{(1)},\cdots,f ^{(m)}$ and graphs $G_1,\cdots,G_m$ such that $G_{i+1}$ is  obtained by the removal of some vertex $w_{i+1}$ together with the edges incident on it in  $G_i$. This procedure can be repeated until $G_m$ becomes a clique, showing that the maximum  of  $\widehat{\phi}_\theta(k)$ is attained at a clique of size $k$.
\end{proof}

We finish this section by introducing a technical lemma that will be useful when estimating the largest eigenvalue of tree-like networks.

\begin{lemma} \label{lemma 03}
Suppose that $G$ is a tree with a  vertex set $[n]$.  Let $s, \xi>0$ and $\theta \geq 1$ be constants. Then, for any vector $f = (f_1,\cdots,f_n)$ with $\sum_{i=1}^n f_i = s$ and $0\leq f_i\leq \xi $ for all $i,$
\begin{align} \label{030}
\sum_{i<j, i\sim j} f_i^\theta f_j^\theta  \leq 
\begin{cases}
\frac{1}{4}s^{2\theta}  \quad  &\textup{if } s< 2\xi , \\
  \xi^\theta (s-\xi )  ^\theta  \quad &\textup{if } s\geq 2\xi .
\end{cases}
\end{align}
In particular, if $0<\xi\leq \frac{1}{2}$, $\theta\geq 1$, $\sum_{i=1}^n f_i ^2 = 1$ and $0\leq f_i\leq \xi  $, then
\begin{align} \label{lemma 36}
\sum_{i<j, i\sim j} f_i^{2\theta} f_j^{2\theta} \leq \xi^{2\theta}.
\end{align}
\end{lemma}
{Note  that the above estimates still  hold even when $G$ is  a  forest (i.e. vertex-disjoint union of trees), since $f_i$s are all  non-negative and one can easily construct a tree on the same vertex set as $G$ with the latter as a subgraph and apply the above result.}

\begin{proof}
{Let {$\rho:=\argmax_{i} f_i$ (there may be several vertices which attain the maximum, and in this case we choose any of them)} and regard $G$ as a tree rooted at $\rho$. Then,
\begin{align*}
\sum_{i<j, i\sim j}f_i^\theta  f_j^\theta \le \sum_{i\neq \rho} f_{\rho}^\theta f_i^\theta   \leq  f_\rho^\theta (s-f_\rho)^\theta,
\end{align*} 
where the last inequality follows from the fact that $x_1^\theta + \cdots + x_m^\theta \leq (x_1+\cdots+x_m)^\theta$ whenever $\theta \geq   1$ and $x_1,\cdots,x_m\geq 0$.  Since the function $ x\mapsto  x(s-x)$ is increasing on $[0,\frac{s}{2}]$ and  $f_{\rho}\le \xi ,$ \eqref{030} follows.

 \eqref{lemma 36} is a direct consequence of \eqref{030}, by replacing $f_i$ with $f_i^2$ and setting $s=1$. } 
\end{proof}

\subsection{Basic spectral properties of weighted graphs}
We end this section by introducing basic but crucial spectral properties of general weighted graphs. 
 
\begin{lemma}[Proposition 3.2 in \cite{KS}] \label{spectrum_graphs}
Let $G = (V,E,A)$ be any network. Suppose that $G_1,\cdots,G_k$  are subgraphs of $G$ and let $A_1,\cdots,A_k$ be the corresponding networks.
If $E(G) = \cup_{i=1}^k E(G_i)$, then $\lambda_1(A) \leq \sum_{i=1}^k \lambda_1(A_i)$. In particular, if the graphs $G_1,\cdots,G_k$ are vertex disjoint,  then $\lambda_1(A) = \max_{i=1,\cdots,k} \lambda_1(A_i)$.

\end{lemma}
{Note that \cite[Proposition 3.2]{KS} was only stated for unweighted graphs, but it is  straightforward that a weighted version holds as well.}

The next lemma characterizes the largest eigenvalue of the  weighted star graph in terms of the Frobenius norm of its conductance matrix.
\begin{lemma} \label{spectrum_star_weights}
If $G = (V,E,A)$ is a weighted star graph with $d+1$ vertices and edge-weights $w_1,\cdots,w_d$, then $\lambda_1(A) = \sqrt{\sum_{j = 1}^d {w_j^2}}.$
\end{lemma}
\begin{proof}
Let $v_0$ be the root and  $v_1,\cdots,v_d$ be vertices connected to the root $v_0$, and assume that   $w_i$  is a weight on the edge connecting vertices $v_0$ and $v_i$. Setting $w := \sqrt{\sum_{j = 1}^d w_j^2}$,  the vector $ \textbf{v} =  \left (1, \frac{w_1}{w}, \dots, \frac{w_d}{w} \right )$  satisfies  $A \textbf{v} = w \textbf{v}$. This shows that $w$ is an eigenvalue of $A$.

 On the other hand, since $G$ is a bipartite graph, $\lambda_1 (A) =  - \lambda_{d+1}(A)$ ($\lambda_{d+1}(A)$ denotes the smallest eigenvalue of $A$). Combining this with the fact $$\sum_{i=1}^{d+1} \lambda_i(A)^2  =  \text{Tr}(A^2) =  2\sum_{j=1}^d w_j^2 = 2w^2,$$ we obtain
$
\lambda_1(A) \leq w.$ Thus, we conclude the proof.
\end{proof}

 We conclude this section by stating a lower bound for the largest eigenvalue of symmetric matrices.
 
\begin{lemma} \label{graphs:max_entry}
For any symmetric matrix $A = (a_{ij})_{i,j}$  whose diagonal entries are all zero,
\begin{align}
  \lambda_1(A)  \geq   \max_{i \neq j} |a_{ij}|.
\end{align}
\end{lemma}
\begin{proof}
Let $|a_{k \ell}|$ be the maximal value, i.e. $|a_{k \ell}| = \max_{i \neq j} |a_{ij}|$.  Let the vector $\textbf{v} = (v_i)_i$  be defined by  
$v_k = \frac{1}{\sqrt{2}}$, $v_\ell =\frac{\sgn(a_{k \ell})}{\sqrt{2}}$ and $ v_i =
0$ otherwise, then $\norm{\textbf{v}}_2=1$ and  $\textbf{v}^T A \textbf{v} = |a_{k \ell}|$. The result now follows by the  variational formulation for $\lambda_1(A)$ applied to the vector $\textbf{v}$.
\end{proof}

\section{Structure of $\calG_{n,\frac{d}{n}}$} \label{sec:str_erd_ren_graphs}
{
In this section, we record various properties of the  Erd\H{o}s-R\'{e}nyi graph. In Section \ref{sec 4.1}, we study the degree profile of $\calG_{n,\frac{d}{n}}$, which will be a crucial input when the edge-weights are light-tailed. In Section \ref{sec 4.2}, we state a simple probability bound for the event that  $\calG_{n,\frac{d}{n}}$ contains a clique, which will be used in the heavy-tailed case. Finally, we present results regarding the connectivity structure of sparser  Erd\H{o}s-R\'{e}nyi graphs  in Section \ref{connectivity}. }

\subsection{Degree profiles}  \label{sec 4.1}
{We start by recording a  result regarding the degree distribution. }

\begin{proposition}\cite[Proposition 1.3]{ganguly1} 
{Let us denote by $d_s,$ the $s$-th largest degree of    $G=\cG_{n, \frac{d}{n}}$. Then, setting
\begin{align*}
t_n := \frac{\log n}{\log \log n},
\end{align*}
we have}
\begin{equation} \label{eq:degrees_bbg}
\lim_{n \to  \infty} \frac{- \log \P \left (d_1 \geq (1+\delta_1) t_n, \dots, d_r \geq  (1+\delta_r) t_n \right ) }{ \log n} = \sum_{s = 1}^r \delta_s.
\end{equation} 
\end{proposition} 
 
 An important input in our arguments will be that, for any  constant $\kappa>0$:
\begin{enumerate}
\item  \label{goal1}  For any fixed $0 < \gamma < 1$, with high probability, there exist  $n^{1-\gamma-\kappa}$  vertices having at least $\gamma \frac{\log n}{\log \log n}$ neighbors with no edges between each other.
\item  \label{goal2} For a suitable discretization $\{\gamma_i\}_{i=1,2,\cdots}$ of $(0,1)$, with high probability, the number of vertices of degree between $\gamma_i \frac{\log n}{\log \log n}$ and $\gamma_{i+1} \frac{\log n}{\log \log n}$ is at most $n^{1-\gamma_i+\kappa}$ for all $i=1,2,\cdots$.
\end{enumerate}
In order to establish{ \eqref{goal1}}, we first estimate the probability that a vertex  has a large degree in $\calG_{n,\frac{d}{n}}$. 
For a subset $L \subseteq V = V(G)$, we denote by $d_L(v)$ the number of vertices in $L$ connected to $v$.
Throughout the paper, to simplify the notation,  for  $\gamma \geq 0$, define
\begin{align} \label{g}
g(\gamma) := \Big \lceil \gamma \frac{ \log n }{ \log \log n } \Big \rceil.
\end{align}

By a well-known binomial tail estimate (see Lemma \ref{binomial_tails} in the Appendix),  for any vertex $v$,
\begin{align} \label{05}
\P (d(v) \geq g(\gamma)) = n^{-\gamma + o(1)}.
\end{align}
In the next lemma, we state a simple extension of \eqref{05}, i.e. $d(v)$ replaced with $d_L(v)$ for general subsets $L\subseteq V$.  Although a straightforward consequence of a binomial tail estimate, we provide a proof, which consists of straightforward but tedious algebra, for the sake of completeness.

\begin{lemma} \label{degree_tails}
{For  $0<\rho<1$, let $L$ be  any subset of $V$ of size $\lfloor \rho n \rfloor$}. Then for any vertex $v \in L^c$ and $\gamma>0$,
\begin{equation}
\P \left ( d_L(v) \geq  g(\gamma)  \right ) = n^{-\gamma + o(1)}.
\end{equation}
\end{lemma}
 Note that   this probability does not depend on the parameter $\rho$, which shows that as long as $|L|$ is of order $n$, the probability is of the same order.

\begin{proof} 

Since $d_L(v)$ is distributed as $\text{Bin} \left (\lfloor \rho n \rfloor, \frac{d}{n} \right )$,  by the mentioned bound in Lemma \ref{binomial_tails}, setting $\theta := \frac{ 1 }{\lfloor \rho n \rfloor  } g(\gamma) =   \frac{ 1 }{\lfloor \rho n \rfloor  } \lceil \gamma \frac{ \log n }{ \log \log n }  \rceil$, 
$$ \frac{1}{\sqrt{8 \lfloor \rho n \rfloor \theta (1-\theta) }} e^{ -\lfloor \rho n \rfloor I_{\frac{d}{n}}\left ( \theta \right )} \leq \P \left ( d_L(v) \geq g(\gamma) \right ) \leq e^{ -\lfloor \rho n \rfloor I_{\frac{d}{n}}\left ( \theta \right )}.$$
Since
$\frac{d}{n}  = o \left (   \frac{ 1 }{\lfloor \rho n \rfloor  } \lceil \gamma \frac{ \log n }{ \log \log n }  \rceil \right )$,  by the relative entropy estimate (see Lemma \ref{entropy bound}),
\begin{align*}
I_{\frac{d}{n}}  (\theta) 
 & = (1+o(1))   \frac{ 1 }{\lfloor \rho n \rfloor  } \Big \lceil \gamma \frac{ \log n }{ \log \log n } \Big \rceil   \log \left ( \frac{n}{d}    \frac{ 1 }{\lfloor \rho n \rfloor  } \Big \lceil \gamma \frac{ \log n }{ \log \log n } \Big \rceil   \right ). 
\end{align*}
Thus, 
$\lfloor \rho n \rfloor I_{\frac{d}{n}}   (\theta)  = \lfloor \rho n \rfloor \frac{ (1+o(1))\gamma  \log n }{\lfloor \rho n \rfloor}
= (\gamma + o(1)) \log n .$

The correction term in the lower bound, namely $\frac{1}{\sqrt{8 \lfloor \rho n \rfloor \theta(1-\theta)} },$ is $n^{ o(1) },$ which implies the matching  bound    $ n^{- \gamma + o(1)}$.

\end{proof}
 
We now proceed to estimate the number of such high-degree vertices satisfying additional useful properties.
 
\begin{proposition} \label{sequential_revealing}
 
For $0 < \gamma,\rho < 1$, let $\mathcal{A}_{\gamma,\rho}$ be the event that there exist {$m:=\lceil  n^{1-\gamma -\rho }\rceil $} vertices $v_1,\cdots,v_m$ and $m$ subsets $W_1,\cdots,W_m\subseteq V$ of size  $g(\gamma)$  satisfying the following properties:
\begin{enumerate}
\item Vertices $v_1,\cdots,v_m$ and elements in $W_1,\cdots,W_m$ are  all  distinct.
\item For each $i$, the vertex $v_i$ is connected to all the elements in  $W_i$.
\item For each $i$, there are no edges within  $W_i$.
\end{enumerate}
Then, 
\begin{align*}
\mathbb{P}( \mathcal{A}_{\gamma,\rho}) \geq 1 -  e^{- n^{1 -\gamma  - \rho  +  o(1) }}.
\end{align*}
 
\end{proposition}

\begin{proof}
Let us partition the set of vertices into two subsets $S:=\{s_1, \dots, s_{ \lceil   \frac{n}{2} \rceil}\}$, which are the potential centers of the stars, and $L := S^c =  \{\ell_1, \dots, \ell_{\lfloor \frac{n}{2} \rfloor}\}$, which will be the potential leaves of the stars.  The ordering on these two sets of vertices is arbitrary and only necessary so that the following algorithm is well-defined.
Now we sequentially reveal the neighbors of vertex $s_1$ in $L$, by first checking whether $\ell_1$ is its neighbor, and so on. Then, 
\begin{enumerate}
\item We either obtain $g(\gamma)$ neighbors of $s_1$ before all edges from $s_1$ to vertices in $L$ are revealed, or
\item There are less than $g(\gamma)$  neighbors of $s_1$ in $L$.
\end{enumerate}

In the first case, we \emph{mark} $s_1$ and define $L_1$ to be the collection of the first $g(\gamma) $ revealed vertices connected to $s_1$. In the second case, we do not mark $s_1$ and set $L_1 = \varnothing$.

Assume that we implemented the above process up to the $k$-th vertex $s_k$  in $S$ and obtained  subsets $L_1,\cdots,L_k \subseteq L$.
We then proceed similarly for $s_{k+1}$, but we only reveal edges from $s_{k+1}$ to vertices in $L \setminus \cup_{i = 1}^k L_i$. {This guarantees that $L_i$s are all disjoint.} As before, we mark $s_{k+1}$ and define $L_{k+1}$ to  be the collection of the  first $g(\gamma)$ revealed vertices connected to $s_{k+1}$ in the former case, and set $L_{k+1} =\varnothing$ in the latter case. We stop this process either once $\lceil n^{1-\gamma -\rho }\rceil$ vertices  in $S$ are marked, in which case we consider the process to be \emph{successful}, or once we revealed edges to vertices in $L$ for all vertices in $S$.

Let $\mathcal{B}$ be the event that the this revealing process is successful. We now  show that this event happens with high probability. Since we discard exactly  $g(\gamma)$  vertices in $L$  for each marked vertex, at each $k$-th step, the set $L \setminus \cup_{i = 1}^{k-1} L_i $ contains at least $\big \lfloor \frac{n}{2} \big \rfloor - n^{1-\gamma -\rho } g(\gamma) \geq \frac{n}{4}$ vertices, as long as $n$ is large enough. Since the edges we reveal at each step are independent of any edges that have been revealed before,  by Lemma \ref{degree_tails}  the probability that $s_k$ has at least $g(\gamma)$  neighbors in $L \setminus \cup_{i = 1}^{k-1} L_i $ is $n^{-\gamma + w}$  for some $w=o(1)$. Hence 
 \begin{align*}
\P(\cB)\ge \P\left (\Binom \left (\Big \lceil \frac{n}{2} \Big  \rceil, n^{-\gamma + w} \right ) \geq n^{1-\gamma -\rho} \right),
\end{align*}  and thus, by Lemma \ref{binomial_tails},
$$ 
\P(\cB) \geq    1 - \exp\Big(-\Big\lceil \frac{n}{2} \Big\rceil   I_{n^{-\gamma + w }} \Big(\frac{n^{1-\gamma -\rho}}{\lceil \frac{n}{2} \rceil } \Big )   \Big ).$$
Since $\frac{n^{1-\gamma -\rho}}{\lceil \frac{n}{2} \rceil } \leq \frac{1}{2} n^{-\gamma + w}$ for large $n$, by  Lemma \ref{rel_ent_bern}, there  exists  a constant $c>0$ such that 
{$$I_{n^{-\gamma + w }}\left (\frac{n^{1-\gamma -\rho}}{\lceil \frac{n}{2} \rceil } \right ) \geq c n^{-\gamma + w} = n^{-\gamma+o(1)}.$$} Thus,
\begin{equation*} 
\P(\cB) \geq  1 - e^{-\lceil \frac{n}{2} \rceil n^{-\gamma + o(1)} } \geq  1 - e^{-n^{1 -\gamma + o(1)} }.
\end{equation*}

Conditioned on the event  $\mathcal{B}$, let us enumerate the marked vertices  by $v_1,\cdots, v_{  \big \lceil n^{1-\gamma-\rho } \big \rceil }$ and the collection of $g(\gamma)$ neighbors that we revealed by  $W_1,\cdots, W_{  \big \lceil n^{1-\gamma-\rho } \big \rceil }$ respectively.  We call a vertex $v_i$ \emph{good} if there are no edges within $W_i$. {Since} having edges within $W_i$ is independent of the revealing process, for large enough $n$,
\begin{equation} \label{no_edges_neighbours}
\mathbb{P}(v_i  \ \text{is good} | \mathcal{B} ) =  \left ( 1- \frac{d}{n} \right )^{\binom{  g(\gamma) }{2}}  \geq  \frac{1}{2}.
\end{equation}

Since  $W_i$s for $i = 1, \dots, \big \lceil n^{1-\gamma-\rho } \big \rceil$  are disjoint,  by independence, the number of good vertices  stochastically dominates $\Binom \left ( \lceil n^{1-\gamma-\rho } \rceil, \frac{1}{2} \right )$. Thus,  again by Lemma \ref{binomial_tails}, for  some  constant $c'>0$,
\begin{equation*} 
\mathbb{P}\Big(\text{There exist at least} \    \frac{1}{4} n^{1-\gamma-\rho } \ \text{good vertices} \big\vert \mathcal{B} \Big) \geq  1- e^{- c'n^{1-\gamma-\rho }}.
\end{equation*}
Hence, putting things together yields that the probability that there exist at least {$ \frac{1}{4} n^{1-\gamma-\rho } $} good marked vertices is at least
\begin{align*}
\left (   1 - e^{- n^{1 -\gamma + o(1) }} \right ) \left (1- e^{- c'n^{1-\gamma-\rho }}  \right ) \geq 1 - e^{- n^{1 -\gamma - \rho + o(1) }}.
\end{align*}
{Since one can absorb the factor $\frac{1}{4}$ into the exponent of $n$ by adjusting the parameter $\rho>0$, we are done.}
\end{proof}

Now we focus on the unlikely appearance of vertices of degree close to     $\gamma \frac{\log n}{\log \log n}$ for $\gamma>1$.  By \eqref{eq:degrees_bbg},  the  probability  of the existence of   such vertex is  $n^{1-\gamma + o(1)}$. In the next proposition, we  improve this statement by further requiring the absence of edges between the neighbors of such vertex.

\begin{proposition}  \label{reveal2}
For  $\gamma>1$,  let  $\mathcal{A}_\gamma'$ be the event that there exists a vertex  $v$ and a subset $W\subseteq V$  of size   $g(\gamma) = \big \lceil \gamma \frac{ \log n}{\log \log n} \big \rceil$ satisfying the following properties:
\begin{enumerate}
\item The vertex  $v$ is connected to all elements in $W$.
\item There are no edges within $W$.
\end{enumerate} 
Then,
  \begin{align*}
  \mathbb{P} \left (\mathcal{A}_\gamma' \right ) =  n^{1-\gamma+o(1)}.
  \end{align*}
\end{proposition}

\begin{proof}
Since the upper bound immediately follows from \eqref{eq:degrees_bbg} with $r=1$ and $\delta_1 = \gamma-1$,  we only prove the lower bound.
We again partition the vertices into the subsets $S:=\left \{s_1, \dots, s_{ \lceil   \frac{n}{2} \rceil} \right \}$ and $L := S^c =  \left \{ \ell_1, \dots, \ell_{\lfloor \frac{n}{2} \rfloor} \right \}$. Then, by Lemma \ref{degree_tails}, for each $s_k\in S$,
\begin{align*}
\mathbb{P} \left ( d_L(s_k) \geq g(\gamma)  \right ) = n^{-\gamma  + o (1) }.
\end{align*}
Let $\mathcal{B}$ be the event that there exists a vertex $s_k \in S$ such that $d_L(s_k) \geq g(\gamma)  $. 
 Since $|S| = \lceil   \frac{n}{2} \rceil $,  
\begin{align}  \label{1000}
\mathbb{P} ( \mathcal{B} ) \geq   1- (1-n^{-\gamma  - o (1) } )^{  \lceil   \frac{n}{2} \rceil } \geq 1- e^{-n^{1-\gamma+o(1)}} \geq n^{1-\gamma+o(1)},
\end{align}   
where we used the fact that $1-e^{-x} \geq \frac{x}{2}$ for small $x>0$.
Given  the event $\mathcal{B}$, let us take any subset $W \subseteq L$ of size $g(\gamma) $ consisting of neighbors of $s_k$. Then, as in \eqref{no_edges_neighbours} in the previous proof, {using the independence between different edges},
\begin{align} \label{1001}
\mathbb{P} \left (\text{There are no edges within } W | \mathcal{B} \right )  \geq  \frac{1}{2}.
\end{align}
{Therefore, the statement follows by multiplying  \eqref{1000} and \eqref{1001}.}
\end{proof}

The next result concerns the degree profile of high-degree vertices of $G$.
To this end we define, for $\gamma \geq 0$, \begin{align} \label{d}
D_\gamma :=\Big \{v\in V: d(v) \geq g(\gamma)   \Big\}.
\end{align}
By \eqref{05},  {if $0\leq  \gamma<1$, then} the expected number of elements in $D_\gamma $ is  of order  $n^{1-\gamma }$. In Proposition \ref{sequential_revealing}, we established that with high probability, $|D_\gamma| \geq  n^{1-\gamma- \rho }$ for any $\rho>0$, which corresponds to a bound on the lower tail of $|D_\gamma|$ for $0\leq \gamma<1$. 

~

{Now, we establish \eqref{goal2} mentioned in the beginning of Section \ref{sec 4.1}.}
To accomplish this, we start by estimating the moments of $|D_\gamma| $.

\begin{lemma} \label{lem:moments_high_degree}
For any $0\leq \gamma <1$,  $\e>0$ and a positive integer $j$,   for sufficiently large $n$,
\begin{equation}
\E \left [ \left | D_\gamma \right |^j\right ]  \leq n^{j-j\gamma + j \e} .
\end{equation}
\end{lemma}

\begin{proof}
 { First note that the bound is trivial for $\gamma = 0$, since $|D_\gamma|^j \leq |V|^j \leq n^j$ for any integer $j$.
As before, let us arbitrarily label all vertices and partition the set of vertices $V = \{v_1,\cdots,v_n\}$ into two subsets $S:= \{v_1, \dots, v_j\}$ and $ L := V \setminus  S$. This definition makes $d_{L}(v_1),\cdots,d_{L}(v_j)$ independent (recall that $d_L(v)$ denotes the number of vertices in $L$ connected to $v$), which we will crucially use in the following moment calculations. First note that 
\begin{align} 
\E \left [ \left | D_\gamma \right |^j \right ] & = \E \left [ \left ( \sum_{i=1}^n \mathbf{1} \left (d(v_i) \geq  g(\gamma) \right ) \right)^j \right ] \nonumber \\
&  \leq  \sum_{k = 1}^j  \binom{n}{k} k^j \P  \left (d(v_1), \dots, d(v_k) \geq  g(\gamma)  \right ), \label{moment_D_g}
\end{align}
where we used the symmetry of the distribution of the vertex degrees.

We now proceed to estimate the joint probabilities of interest. Since each vertex $v_i\in S$ can have at most $j-1$ edges into $S$, for any $1\leq k\leq j$, for large enough $n$,
\begin{align*}
 \P \left (d(v_1), \dots, d(v_k) \geq g(\gamma) \right )
&\leq \ \P  \left (d_{L}(v_1), \dots, d_{L}(v_k) \geq g(\gamma)  - j +1 \right )\\
& \leq \P  \left (d_{L}(v_1), \dots, d_{L}(v_k) \geq g\Big(\gamma - \frac{\e}{2}  \Big)  \right ) \leq \left ( \frac{1}{n^{\gamma - \frac{\e}{2} +o(1)}} \right )^k,
\end{align*}
where we used the independence of $d_{L}(v_1),\cdots,d_{L}(v_k)$ and  a tail probability estimate for  the vertex degree  (Lemma \ref{degree_tails})  in the last inequality.
Applying this estimate to each term in \eqref{moment_D_g}, for all $1\leq k \leq j$,
\begin{align*}
\binom{n}{k} k^j \P  \left (d(v_1), \dots, d(v_k) \geq g(\gamma)  \right ) 
& \leq \frac{1}{k!} k^j n^{k \left (1 - \gamma + \frac{\e}{2} \right ) + o(1) } \leq n^{j(1-\gamma + \frac{\e}{2}) + o(1)},
\end{align*}
since $1-\gamma + \frac{\e}{2} > 0$. Therefore,  \eqref{moment_D_g} is bounded by
$
j \cdot  n^{j(1-\gamma + \frac{\e}{2}) + o(1)} \leq n^{j (1-\gamma + \e)}
$ for large $n$, which concludes the proof. 
}
\end{proof}

The moment bound yields the following.

\begin{proposition} \label{degree_distribution}
 For any $0< \kappa <1 $,   let $m$ be an integer  such that $m \kappa < 1 \leq (m+1)\kappa $. Then, for any $\mu > 0$ and sufficiently large $n$,
\begin{equation}
\P \big( |D_{ i \kappa} |\leq n^{1- i \kappa  + \kappa} \textup{ for all} \   i =0,1,\cdots,m  \big) \geq 1 - n^{-\mu \kappa}.
\end{equation}
\end{proposition}
\begin{proof}
By a union bound and the Markov's inequality combined with Lemma \ref{lem:moments_high_degree}, for any $\e>0$ and a positive integer $j$,
\begin{align*}
\P \left (| D_{i \kappa} | \geq n^{1-   i \kappa  + \kappa} \text{ for some } i =0,1,\cdots,m  \right )&\leq \sum_{i=0}^m \P \left ( |D_{ i \kappa}| \geq n^{1-  i \kappa  + \kappa} \right)  \\
&\leq   \sum_{i=0}^m  \frac{ \E\left [|D_{ i \kappa}|^j \right ]}{n^{j- j  i \kappa + j\kappa }} \leq   (m+1) n^{j \e -j \kappa}.
\end{align*}
Taking $\e = \frac{\kappa}{2},$ and noticing that $m < \frac{1}{\kappa}$,  the above expression is bounded by
$ (m+1) n^{ -j\frac{\kappa}{2}}  \leq  \Big( \frac{1}{\kappa}  + 1\Big) n^{-j\frac{\kappa}{2}} .$
Since $j$ is an arbitrary integer, this concludes the proof.
\end{proof}

 Note that the previous  lemma and proposition could be stated with an asymptotic notation rather than the additional constants $\e$ and $\mu$, but the current phrasing will make it easier to use the results in our main proofs.

 \subsection{Probability of the existence of a clique} \label{sec 4.2}
As mentioned in the idea of proof section, {atypically} high degree stars (with high edge-weights on them) are the driving mechanism behind the largest eigenvalue in the case of light-tailed weights. In the case of heavy-tailed weights on the other hand, a similar role is played by cliques. In  this short section, we  state a bound for the probability that $G = \calG_{n, \frac{d}{n}} $ contains a clique of size $k$.

\begin{lemma}[{\cite[Lemma 4.3]{gn}}] \label{lemma clique}
For any integer $k \geq 3$, there exists a constant $C = C(k,d)>0$ such that {
\begin{equation*}
C n^{ -{k \choose 2} + k} \leq \mathbb{P}(G  \ \textup{contains a clique of size} \  k)  \leq    d^{{k \choose 2}} n^{ -{k \choose 2} + k} .
\end{equation*}
}
\end{lemma}

\subsection{Connectivity structures of highly  sub-critical  Erd\H{o}s-R\'{e}nyi graph} \label{connectivity}
{
As a key step in our proofs, we will decompose the underlying graph $X$ into  graphs with low and high edge-weights respectively.} Because of the threshold we choose, the latter graph turns out to be a highly subcritical graph $\cG_{n,q}$ with \begin{align} \label{q}
q \leq  \frac{d'}{n (\log n)^{\e}}
\end{align}  for some constants $\e,d'>0$.  In this section, we record some properties of such graphs, namely that all connected components look like trees and that their sizes are well-controlled.

Throughout this section, we assume that the edge density $q$ satisfies \eqref{q}.
First, we have a bound on the largest degree denoted by $d_1(\cG_{n,q})$, as a direct consequence of  \eqref{eq:degrees_bbg}. 
\begin{lemma} 
 \label{lemma 33}
For $\delta_1>0$,
define the event
\begin{align}\label{degsize}
\cD_{\delta_1} :=  \left \{d_1(\cG_{n,q}) \leq  (1+\delta_1) \frac{\log n}{\log \log n}  \right \}.
\end{align}
Then,
\begin{align*}
 \liminf_{n\ri} \frac{-\log \mathbb{P} \left (  \cD_{\delta_1}^c \right )  }{\log n}  \ge \delta_1.
\end{align*}
\end{lemma}

\begin{proof}
Note that $\cD_{\delta_1}$ is a decreasing event and that $\cG_{n,q}$ is stochastically dominated by $\cG_{n,\frac{d}{n}}$.  Hence, by \eqref{eq:degrees_bbg}, we obtain the result.
\end{proof}

Next, we state a quantitative bound on the size of largest connected component. 

\begin{lemma}[{\cite[Lemma 5.4]{gn}}] \label{lem:biggest_component}
Let  $C_1,\cdots,C_L$ denote the connected components of $\cG_{n,q}$. For $\delta_2>0$,
define the event
\begin{align*}
\mathcal{C}_{\e,\delta_2} :=
 \left \{ \max_{i=1, \dots, L} |C_i|  \leq  ( 1 +\delta_2 )  \frac{1}{\e} \frac{\log n}{\log \log n}   \right  \} .
\end{align*}
Then,
\begin{align}
\liminf_{n \to \infty} \frac{ - \log \P\left (\mathcal{C}_{\e,\delta_2}^c \right )}{\log{n}} \geq {\delta_2}.
\end{align}

\end{lemma}
 
 To conclude this section we state two results about the structure of the connected components. The first one quantifies how similar all connected components are to trees, in the sense that they have a small number of tree-excess edges. The second one concerns the event that all connected components are trees.

\begin{lemma}[{\cite[Lemma 5.6]{gn}}]\label{lem:excess_edges}
Let    $C_1,\cdots,C_L$ denote the connected components of $\cG_{n,q}$.
For $ \delta_3 > 0 $, define the event 
\begin{equation}
\mathcal{E}_{\delta_3} := \left \{\max_{i=1, \dots, L} \big \{|E(C_i) |- |V(C_i)| \big  \} \leq \delta_3 \right \}.
\end{equation}
Then, 
\begin{equation}
\liminf_{n \to \infty} \frac{-\log \P \left ( \mathcal{E}_{\delta_3}^c \right ) }{\log n} \geq \delta_3.
\end{equation}
In addition,  define the event  
\begin{align*}
\cT:=\big \{ |E(C_i)| = |V(C_i)| - 1,  \ \forall i = 1,\cdots,L \big \}.
\end{align*}
In other words, $\cT$ is the event that  all the connected components of $\cG_{n,q}$ are trees. Then,   there is a constant $C>0$ that depends on $d$, such that
\begin{align} \label{351}
\mathbb{P} \big (\cT^c  \big ) \leq \frac{C}{(\log n)^{2\e}}.
\end{align}
 
\end{lemma}

\section{Light-tailed weights} \label{sec:light}
We consider $\alpha>2$ in this section and prove Theorems \ref{thm:light_upper} and \ref{thm:light_lower}. 
For the reader's benefit let us recall that 
\begin{align}\label{typ11} \lambda^{\textup{light}}_{\alpha  }  = 2^\frac{1}{\alpha} \alpha^{-\frac{1}{2}} (\alpha - 2)^{\frac{1}{2}-\frac{1}{\alpha}}   \frac{(\log n) ^\frac{1}{2}}{ (\log \log n )  ^{\frac{1}{2}- \frac{1}{\alpha}}}.
\end{align} 
For notational brevity, in this section we will denote $\lambda^{\textup{light}}_{\alpha  }$ simply by $\lambda_{\alpha  }.$
{To simplify the notation further, we set
\begin{align}\label{b_alpha}
B_\alpha := 2^\frac{1}{\alpha} \alpha^{-\frac{1}{2}} (\alpha - 2)^{\frac{1}{2}-\frac{1}{\alpha}} .
\end{align}
}

\subsection{The upper tail} \label{5.1}

Let us recall the theorem that we will prove in this section.

\lightupper*

{The governing structure for the upper tail of $\lambda_1(Z)$ will turn out to be  a star of degree $\lceil \gamma_\delta \frac{\log n}{\log \log n} \rceil$ with
\begin{align} \label{31}
\gamma_\delta := (1+\delta)^2 \left ( 1 - \frac{2}{\alpha} \right)
\end{align}
  and high edge-weights on the edges. This stems from the following optimization problem. The probability that the maximum of the largest eigenvalue among all the typically present $n^{1-\gamma}$ stars of degree  $\lceil\gamma \frac{\log n}{\log \log n}\rceil$ (see e.g., Lemma \ref{lem:moments_high_degree} and Proposition \ref{degree_distribution}) is greater than  $ (1+\delta) \lambda_\alpha$  is maximized at  $\gamma = \gamma_\delta$.}

\subsubsection{Lower bound for the upper tail}

  The strategy will change depending on whether $\gamma_\delta$ is less or greater than 1.

\textbf{Case 1:  $\bm{\gamma_\delta<1}$.}
For small enough $\rho>0$, we condition on the event $\mathcal{A}_{\gamma_\delta,\rho}$ measurable with respect to $X$, defined in Proposition \ref{sequential_revealing}. Conditioned on that event, there exist $m: = \left \lceil  \frac{1}{4} n^{1-\gamma_\delta  - \rho} \right \rceil $ vertices with  $ g(\gamma_\delta) =  \left \lceil   \gamma_\delta \frac{ \log n}{\log \log n} \right \rceil  $  disjoint   neighbors  with no edges between each neighbors.
Denote by $S_1,\cdots,S_m$ the vertex-disjoint stars induced by these vertices and their $g(\gamma)$ neighbors. Then,  by Lemma \ref{spectrum_graphs},
\begin{equation*}
\lambda_1(Z) \geq \max_{k=1,\cdots,m} \lambda_1 (S_k).
\end{equation*}
Thus, conditioned on the event $\mathcal{A}_{\gamma_\delta,\rho}$,
by the characterization of the largest eigenvalue of stars in Lemma \ref{spectrum_star_weights},
\begin{align} \label{20}
 \mathbb{P}\big(\lambda_1(Z) \geq (1+\delta) \lambda_{\alpha } \mid  X \big )   
 \geq  
 \P \left (\max_{k= 1, \dots, m} \sum_{(i,j)\in E(S_k)} Z_{ij}^2 \geq (1+\delta)^2 \lambda_{\alpha } ^2  \mid   X \right )
\end{align}
(recall that  $(i,j)$  denotes the  \emph{undirected} edge joining vertices $i$ and $j$ with $i<j$).
In the appendix, we derive a tail bound \eqref{common_tail_terms:upper}  for the sum of squares of Weibull random variables. Plugging in the bound  with $d=1+\delta$ and $b=\gamma_\delta$, under the event  $\mathcal{A}_{\gamma_\delta,\rho}$, for each $k=1,\cdots,m$,
\begin{align*}
\P  \left ( \sum_{(i,j)\in E(S_k)} Z_{ij}^2\geq (1+\delta)^2 \lambda_{\alpha } ^2 \mid  X  \right )  
 \geq  n^{-(1+\delta)^\alpha  \frac{2}{\alpha-2}(1-\frac{2}{\alpha})^{\frac{\alpha}{2}} \gamma_\delta^{1-\frac{\alpha}{2}} + o(1)} =    n^{-(1+\delta)^2 \frac{2}{\alpha}  + o(1)},
\end{align*}
where we used $\gamma
_\delta = (1+\delta)^2 \left ( 1 - \frac{2}{\alpha} \right)$ in the last equality.
Using the independence of edge-weights and recalling $m = \left \lceil  \frac{1}{4} n^{1-\gamma_\delta  - \rho} \right \rceil $, under the event  $\mathcal{A}_{\gamma_\delta,\rho}$,
\begin{align*}
\P \left (\max_{k= 1, \dots, m} \sum_{(i,j)\in E(S_k)} Z_{ij}^2\geq (1+\delta)^2 \lambda_{\alpha } ^2  \mid  X \right )
& \geq 1 - \left ( 1 - n^{-(1+\delta)^2 \frac{2}{\alpha}+ o(1)}  \right )^{ m} \\
& \geq  1 - e^{-  n^{1- \gamma_\delta -\rho   - (1+\delta)^2 \frac{2}{\alpha} + o(1)}}  \geq \frac{1}{2} n^{1-(1+\delta)^2 - \rho+ o(1)},
\end{align*}
where we used  the fact that $1-e^{-x} \geq \frac{1}{2}x$ for small $x>0$ in the last inequality.

Therefore, applying this to  \eqref{20} and using that $\mathbb{P}(\mathcal{A}_{\gamma_\delta,\rho}) \geq \frac{1}{2}$ (see Proposition \ref{sequential_revealing}), 
 \begin{align*}
\P \big ( \lambda_1(Z) \geq (1+\delta) \lambda_{\alpha } \big )  
& \geq  
\mathbb{E} \left [ \mathbb{P} \big (\lambda_1(Z) \geq (1+\delta) \lambda_{\alpha } \mid  X \big ) 1_{ \mathcal{A}_{\gamma_\delta,\rho} } \right ] 
\geq 
  \frac{1}{4} n^{1-(1+\delta)^2 - \rho+ o(1)}.
\end{align*}
Due to the arbitrariness of $\rho>0$,  
\begin{align*}
\limsup_{n \to \infty} - \frac{ \log \P \big ( \lambda_{1} (Z) \geq (1+\delta) \lambda_{\alpha }  \big )}{\log{n}} \leq  (1+\delta)^2 - 1.
\end{align*}

\textbf{Case 2: $\bm{\gamma_\delta \geq 1}$.}
For any $\rho>0$, we condition on the event $\mathcal{A}_{ (1+\rho)\gamma_\delta}'$ measurable with respect to $X$, defined in Proposition \ref{reveal2}.
 Under this event there exists a vertex $v$ with $ \left \lceil (1+\rho) \gamma_\delta \frac{\log n}{\log \log n} \right  \rceil $  neighbors with no edges between them.  Denote by  $S$ {the star induced by $v$} and these neighbors.
As before, using the tail bound \eqref{common_tail_terms:upper} with $d=1+\delta$ and $b=(1+\rho) \gamma_\delta$, under the event $\mathcal{A}_{ (1+\rho)\gamma_\delta}'$,
\begin{align*}
\P \big ( \lambda_1(Z) \geq (1+\delta) \lambda_{\alpha }  \mid   X \big) 
& \geq 
\P \left( \sum_{(i,j)\in E(S)} Z_{ij}^2 \geq (1+\delta)^2 \lambda_{\alpha } ^2  \mid X \right )  \\
& \geq n^{-(1+\delta)^\alpha  \frac{2}{\alpha-2}(1-\frac{2}{\alpha})^{\frac{\alpha}{2}}  ((1+\rho)\gamma_\delta)^{1-\frac{\alpha}{2}} + o(1)}  =  n^{-(1+\delta)^2 \frac{2}{\alpha} (1+\rho)^{1-\frac{2}{\alpha}} + o(1) }.
\end{align*}
By Proposition \ref{reveal2},
\begin{align*}
\P \left (\mathcal{A}_{ (1+\rho)\gamma_\delta}' \right ) =  n^{1- (1+\rho)\gamma_\delta + o(1)}.
\end{align*}
Thus, as above,
\begin{align*}
\limsup_{n \to \infty} \frac{ - \log \P \big ( \lambda_1(Z) \geq (1+\delta) \lambda_{\alpha } \big ) }{\log n} 
& \leq 
-1 + (1+\rho)  \gamma_\delta + (1+\delta)^2 \frac{2}{\alpha} (1+\rho)^{1-\frac{2}{\alpha}} \\
& = -1+(1+\delta)^2  \Big[ (1+\rho) \Big(1-\frac{2}{\alpha}\Big) +   \frac{2}{\alpha}(1+\rho)^{1-\frac{2}{\alpha}}\Big].
\end{align*}
Since $\rho>0$ is arbitrary, this implies the desired bound.

\qed

\begin{remark} \label{remark critical}
If $\gamma_\delta = 1$, then the precise behavior of the number of vertices of degree close to $ \gamma_\delta \frac{\log n}{\log \log n}= \frac{\log n}{\log \log n}$ {or the probability of the existence of such  a vertex is somewhat delicate to track}. Hence, we considered vertices with a slightly larger degree instead to exploit the large deviation bound for atypically large degrees from \eqref{eq:degrees_bbg}. This shortens the proof  by not dealing with the case $\gamma_\delta = 1$ separately.

\end{remark}
\subsubsection{Upper bound of the upper tail} 
 
We proceed in a sequence of steps:
\begin{enumerate}[(1)]
\item \label{outline_light_step_decompose} We first truncate the weights $Y$ and then accordingly decompose $Z$ into $Z^{(1)} + Z^{(2)}$:
\begin{align} \label{100}
Z^{(1)}_{ij} = X_{ij} Y^{(1)}_{ij}  \quad  \text{and} \quad  Z^{(2)}_{ij} = X_{ij} Y^{(2)}_{ij},
\end{align} 
where
\begin{align}  
  Y^{(1)}_{ij}  = Y_{ij}\1_{|Y_{ij}| > (\e \log \log n )^\frac{1}{\alpha}}    \quad  \text{and} \quad   Y^{(2)}_{ij}  = Y_{ij}\1_{|Y_{ij}| \leq (\e \log \log n )^\frac{1}{\alpha}} .
\end{align}
Similarly, write $ X=X^{(1)}+X^{(2)}$ with 
\begin{align} \label{101}
X^{(1)}_{ij}=X_{ij}\1_{|Y_{ij}| > (\e \log \log n )^\frac{1}{\alpha}}   \quad  \text{and} \quad   X^{(2)}_{ij}=X_{ij}\1_{|Y_{ij}| \leq  (\e \log \log n )^\frac{1}{\alpha}}.
\end{align} 
{This particular threshold is chosen so that, as we will soon see, $Z^{(2)}$ is spectrally negligible.}

 By the tail decay of Weibull distributions, $X^{(1)}$  is distributed as   $\calG_{n,q}$ with 
\begin{align} \label{qq}
q \leq  \frac{d'}{n (\log n)^{  \e}}
\end{align} 
for some constant $d'>0$.  Also, given  $X^{(1)}$, the edge-weights on the network $Z^{(1)}$ can be regarded as  i.i.d. Weibull distributions conditioned to be greater than $(\e \log \log n)^{\frac{1}{\alpha}}$ in absolute value.

\item \label{outline_light_step_component_analysis} 
 We analyze the component structure of $X^{(1)}$, the underlying graph of the network  $Z^{(1)}$. {The sparsity of $X^{(1)}$ allows the results in Section \ref{connectivity} to be applicable. In particular, connected components of $X^{(1)}$ are tree-like (i.e. the number of tree-excess edges are small)} and their sizes are relatively small with high probability. 
\item \label{outline_light_step_decompose2} We further decompose the network $Z^{(1)}$ into  $Z^{(1)}_1$, consisting of vertex-disjoint weighted stars, and  $Z^{(1)}_2$, whose degrees are well-controlled.
 
\item \label{outline_light_step_Z^1_2_negligible} 
Using the results in Step \eqref{outline_light_step_component_analysis} and the fact that the maximal degree in $Z^{(1)}_2$ is relatively small, we prove  that  $Z^{(1)}_2$ is spectrally negligible as well.
\item \label{outline_light_step_stars} 
 We analyze the spectral contribution of $Z^{(1)}_1$ by grouping stars according to their degrees.  Since we have a complete characterization of the largest eigenvalue of a (weighted) star graph (see Lemma \ref{spectrum_star_weights}),  one can explicitly compute the contribution from the collections of stars of a given degree. It turns out that the main contribution, which leads to the large deviation probability, comes from the stars of degree close to $\gamma_\delta \frac{\log n}{\log \log n}$ (see \eqref{31}).

\end{enumerate}

Given the truncation in Step \eqref{outline_light_step_decompose} above, 
we first estimate $\lambda_1(Z^{(2)})$.

\begin{lemma} \label{lem:bulk_negligible}
{For any $\delta,\e>0$,}
\begin{equation}
\liminf_{n \to \infty} \frac{-\log \P\Big (\lambda_1\left (Z^{(2)} \right ) \geq \e^{\frac{1}{\alpha}} (1+\delta)  \frac{\lambda_\alpha}{B_\alpha} \Big )}{\log n} 
\geq
 (1+\delta)^2 - 1.
\end{equation}
\end{lemma}

We start by stating the following theorem from  \cite{ganguly1}. While the latter covers a varied range of values for $p,$ we state it only in the sparsity regime considered in this paper. {Recall that we set $t_n = \frac{\log n}{\log \log n}$.}
 
\begin{theorem}[{\cite[Thm. 1.1]{ganguly1}}] \label{thm:eigenvalues_bbg}
For any   $\delta > 0$,
\begin{equation}
\lim_{n \to \infty} \frac{ -\log \P ( \lambda_1(X) \geq (1+\delta) t_n^\frac{1}{2}  ) }{ \log n} = (1+\delta)^2 - 1.
\end{equation}
\end{theorem}

\begin{proof}[Proof of Lemma \ref{lem:bulk_negligible}]
Since $\big|Y_{ij}^{(2)} \big | \leq  (\varepsilon \log \log n)^{\frac{1}{\alpha}}$ {for all $ i, j$,}
\begin{equation} \label{Z2_wrt_X2}
\lambda_1(Z^{(2)})   \leq (\e \log \log n)^\frac{1}{\alpha} \lambda_1(X).
\end{equation} 
{By Theorem \ref{thm:eigenvalues_bbg} and recalling $\frac{\lambda_\alpha}{B_\alpha} =  \frac{(\log n) ^\frac{1}{2}}{ (\log \log n )  ^{\frac{1}{2}- \frac{1}{\alpha}}}$ (see \eqref{typ11} and  \eqref{b_alpha}), this immediately concludes the proof.}
 
\end{proof}

{Now we analyze $Z^{(1)}$, the main spectral part for $\lambda_1(Z)$. Since the edge density of its underlying graph $X^{(1)} \overset{\text{d}}{\sim} \cG_{n,q}$ satisfies \eqref{qq}, the  results in Section \ref{connectivity} hold for $X^{(1)}$ as indicated in Step \eqref{outline_light_step_component_analysis}.}

We now implement Step \eqref{outline_light_step_decompose2}, i.e. we decompose $X^{(1)}$ into two parts, one of which consists of a \emph{vertex-disjoint}  union of stars while the other one has a relatively small maximum degree. The latter part will be spectrally negligible and the dominating factor will be the former part. 
 {For this we rely on a result from \cite[Lemma 3.5]{ganguly1}, that we simplified slightly for our setting.}

\begin{lemma}[{\cite[Lemma 3.5]{ganguly1}}] \label{BBG_decomposition} 
{There exists  an event $\mathcal{W}$ measurable with respect to  $X^{(1)}$}   that happens with probability at least \footnote{Here, the quantity $x=w(\log n)$ means that $\lim_{n \rightarrow \infty} \frac{x}{\log n} = \infty$. }$ 1 - e^{-\omega(\log n)}$ under which $X^{(1)}$  can be decomposed into a graph $X^{(1)}_1$  which is a vertex-disjoint union of stars, and a graph $X^{(1)}_2$  whose maximum degree is $o \left ( \frac{\log n}{\log \log n} \right )$.
 
\end{lemma}
{The decomposition in \cite[Lemma 3.5]{ganguly1} is stated for $\cG_{n,p}$ with $p = O(\frac{1}{n})$ which is applicable for $X^{(1)}$ by \eqref{qq}.}

 From now on, we condition on the high probability event $\cW$. Let $Z^{(1)}_1$ and $Z^{(1)}_2$ be the corresponding networks of $X^{(1)}_1$ and $X^{(1)}_2$ respectively.  
We will first focus on the spectral behavior of  $   Z^{(1)}_2$ by analyzing its underlying graph $X^{(1)}_2$.  {By Lemmas \ref{lem:biggest_component}, \ref{lem:excess_edges} and \ref{BBG_decomposition}, each connected component $C_\ell$ of   $X^{(1)}_2$}
satisfies  the following properties {with high probability}  for any $\delta_1,\delta_2>0$:
\begin{enumerate}
\item $d_1(C_\ell) = o \left ( \frac{\log n}{\log \log n} \right )$,
\item $|V(C_\ell)| \leq ( 1 + \delta_1 )  \frac{1}{   \e} \frac{\log n}{\log \log n}$,
\item $|E(C_\ell)| \leq |V(C_\ell)| + \delta_2$
\end{enumerate}
(recall that for any graph $G$, $d_1(G)$ denotes the maximum degree of $G$).  
 
We now state the following key general proposition, which claims that it is  costly that any connected network satisfying the above three properties has the largest eigenvalue of order $\lambda_\alpha$ {(recall that $\lambda_\alpha$ denotes the quantity which turns out to be the typical largest  eigenvalue, as defined in \eqref{typ11}).}

\begin{proposition}  \label{prop:eigenvalue_components_Z^1_2}
Assume that $\alpha>2$. For any $n \in \mathbb{N}$ and $\e, \delta_1, \delta_2 >0$, {let $\cG := \cG_{n,\e,\delta_1,\delta_2}$ be the set of}   connected networks $G=(V,E,A)$ ($A= (a_{ij})_{i,j\in V}$ denotes the conductance matrix) such that
\begin{enumerate}[(1)]
\item $d_1(G) = o \left ( \frac{\log n}{\log \log n} \right )$,
\item $|V| \leq ( 1 + \delta_1 )  \frac{1}{   \e} \frac{\log n}{\log \log n}$,
\item $|E| \leq |V| + \delta_2$.
\end{enumerate} Assume that the edge-weights  are i.i.d.   Weibull distributions with a shape parameter $\alpha>2$ conditioned to be greater than $(\e \log \log n )^\frac{1}{\alpha}$ in absolute value. Then, for any constant $c > 0$,
\begin{align*}
\lim_{n \to \infty} & \frac{ - \log \sup_{G \in \cG} \P \left ( \lambda_1(A) \geq c \lambda_{\alpha }   \right ) }{ \log n }  = \infty.
\end{align*}
 
\end{proposition}

\begin{proof}
 
 The general strategy is to bound $\lambda_1(A)$ by expressing it in terms of the corresponding (random) top eigenvector and then analyzing the contributions from  the high and low values of the entries separately.  To make this precise, 
suppose that $V  = [m]$, and let $ f= (f_i)_{ i \in [m]}$ with $||f||_2 =1$ be any (random) eigenvector of $A$ such that
\begin{align} \label{410}
 \lambda_1 (A)= f^T A f  = \sum_{i,j = 1}^m a_{ij} f_i f_j = 2\sum_{ (i,j) \in E} a_{ij}f_if_j
\end{align}  
 (recall that  $(i,j)$  denotes  the  \emph{undirected} edge joining two vertices $i<j$).

For a constant $\xi  \in \left (0, \frac{1}{2}\right)$ which will be chosen sufficiently small later, define
\begin{align} \label{400}
V_S:= \{i \in [m]: |f_i| < |\xi|  \},\qquad V_L := \{i \in [m]:  |f_i| \geq  |\xi|  \},
\end{align}
where the indices stand for \emph{small} and \emph{large} respectively.
Since $\sum_{i=1}^m f_i^2 = 1$, we have
\begin{align} \label{411}
|V_L| \leq \left \lfloor \frac{1}{\xi ^2}  \right \rfloor.
\end{align}

We also partition the set of edges into two parts, those that are incident on a vertex in $V_L$ and the rest:
\begin{align}\label{401}
E_S := \left \{ (i,j) \in E: i<j, i, j \in V_S \right \}, \qquad E_L := \left \{ (i,j) \in E:i < j,  i \in V_L \text{ or } j \in V_L \right \}.
\end{align}
We now decompose the summation in \eqref{410} into two parts $\lambda_S$  and $\lambda_L$:
\begin{align} \label{413}
\lambda_1 (A)= 2\sum_{ (i,j) \in E} a_{ij}f_if_j = 2 \sum_{(i,j)\in E_S} a_{ij} f_i f_j   + 2 \sum_{(i,j)\in E_L} a_{ij} f_i f_j  =: {2 \lambda_S + 2\lambda_L}.
\end{align}
{The above is expressed in a way such that both $\lambda_S$ and $\lambda_L$ can be bounded by sums of i.i.d. random variables which will be convenient.}\\
 
Thus, for  any constant $0 \leq \tau \leq 1$ which will be chosen later, 
\begin{equation} \label{decompose_eigenvalue}
\mathbb{P}( \lambda_1 (A) \geq  c \lambda_{\alpha } ) \leq \P( 2 \lambda_S \geq \tau c \lambda_{\alpha })  + \P(2 \lambda_L \geq (1-\tau)c \lambda_{\alpha } ).
\end{equation} 
We now analyze these two probabilities separately. 

\textbf{Bounding $\bm{\lambda_S}$.}
We apply the Cauchy-Schwarz inequality to $\sum_{(i,j)\in E_S}a_{ij}f_i f_j$, and then use a bound on $\sum_{(i,j)\in E_S} f^2_i f^2_j$ which we now derive.
Let $T$ be a spanning tree of  $G$ (recall that $G$ is connected) and $E(T)$ be the collection of   edges in $T$. Since $|E(T)| = |V| - 1$, by our assumption on the number of tree-excess edges, {$|E_S \setminus E(T)| \leq |E| - |E(T)| \leq  \delta_2 +1$}. Hence,  
\begin{align} \label{416}
\sum_{(i,j)\in E_S} f^2_i f^2_j 
& = \sum_{(i,j)\in E_S \cap E(T)} f^2_i f^2_j + \sum_{(i,j)\in E_S \setminus E(T)} f^2_i f^2_j   \leq \xi ^2  + (\delta_2 +1)\xi ^4 \leq (2+\delta_2)\xi ^2,
\end{align}
{where we used   \eqref{lemma 36}  with $\theta=1$ to bound the first term (note that $|f_i| \leq \xi < \frac{1}{2}$).}
Thus, setting
\begin{align} \label{theta}
\tau: =  (2+\delta_2)^\frac{1}{4} \xi^\frac{1}{2},
\end{align}  
 by the Cauchy-Schwarz inequality together with the bound \eqref{416},
\begin{align*}
\lambda_S 
\leq \Bigg ( \sum_{(i,j) \in E_S} f_i^2f_j^2 \Bigg )^{\frac{1}{2}} \Bigg ( \sum_{(i,j) \in E_S} a_{ij}^2 \Bigg )^{\frac{1}{2}} 
\leq \tau^2 \Bigg ( \sum_{(i,j) \in E_S} a_{ij}^2 \Bigg )^{\frac{1}{2}} \leq \tau^2 \Bigg ( \sum_{(i,j) \in E} a_{ij}^2 \Bigg )^{\frac{1}{2}}.
\end{align*}
By  assumptions on the network size and the number of tree-excess edges, $\sum_{(i,j) \in E} a_{ij}^2$ is the sum of at most $\left \lfloor  ( 1 + \delta_1 )  \frac{1}{   \e} \frac{\log n}{\log \log n} + \delta_2  \right \rfloor$ many squares of Weibull random variables   conditioned to be greater than $\left ( \e \log \log n \right )^\frac{1}{\alpha}$ in absolute value. Hence, by the tail estimate for   such sum of squares ({the bound \eqref{cond1} with $d=\frac{c}{2\tau} = \frac{c}{2 (2+\delta_2)^\frac{1}{4} \xi^\frac{1}{2}}$ and $b = \frac{1+\delta_1}{\e}$}),
\begin{align} \label{415}
\liminf_{n\rightarrow \infty} \frac{-\log \P (2\lambda_S \geq \tau c \lambda_{\alpha }  )  }{\log n}& \geq  \liminf_{n\rightarrow \infty} \frac{-\log \P \Big ( \sum_{(i,j) \in E} a_{ij}^2 \geq \frac{ c^2 }{4 \tau^2} \lambda_{\alpha } ^2 \Big )  }{\log n} \nonumber \\
& \geq   \frac{c^\alpha}{2^\alpha (2+\delta_2)^\frac{\alpha}{4} \xi^\frac{\alpha}{2}}   \frac{2}{\alpha-2} \Big ( 1 - \frac{2}{\alpha}  \Big )^\frac{\alpha}{2} \Big(  \frac{1 + \delta_1}{  \e}\Big)^{1-\frac{\alpha}{2}}  - (1+\delta_1)  .
\end{align}

\textbf{Bounding $\bm{\lambda_L}$.}
By  the Cauchy-Schwarz inequality and the fact that $$\sum_{(i,j) \in E_L} f_i^2f_j^2 \leq \big (\sum_{i \in V} f_i^2 \big ) \big ( \sum_{j \in V} f_j^2 \big ) = 1,$$ we have 
$$\lambda_L \leq \Bigg ( \sum_{(i,j) \in E_L} f_i^2f_j^2 \Bigg )^{\frac{1}{2}} \Bigg ( \sum_{(i,j) \in E_L} a_{ij}^2 \Bigg )^{\frac{1}{2}} 
\leq \Bigg ( \sum_{(i,j) \in E_L} a_{ij}^2 \Bigg )^{\frac{1}{2}}.$$
Since $ |V_L| \leq \left \lfloor \frac{1}{\xi ^2}  \right \rfloor$ by \eqref{411}, the event $\{2\lambda_L \geq (1-\tau)c \lambda_{\alpha } \} $ implies the existence of a  random subset $J \subseteq V$ with $|J| \leq \left \lfloor \frac{1}{\xi ^2}  \right \rfloor$ such that
\begin{equation*}
\sum_{(i,j) \in E_J} a_{ij}^2  \geq \frac{(1-\tau)^2 c^2 }{4} \lambda_{\alpha}^2,
\end{equation*}
where
$E_J:= \left \{ (i,j) \in E:i < j,  i \in J \text{ or } j \in J \right \}$.
For any deterministic subset $J'$ with $|J'| \leq  \left \lfloor \frac{1}{\xi ^2}  \right \rfloor$,  by our assumption on the maximum degree, $|E_{J'}|  =  o \left (  \frac{1}{\xi ^2}  \frac{\log n}{\log \log n} \right )$.
Thus, by the tail probability  estimate for the sum of squares  of Weibull random variables (the bound \eqref{cond2} with $d= \frac{(1-\tau)c}{2}$),
\begin{align*}
\lim_{n \to \infty} \frac{ - \log \P \left ( \sum_{(i,j) \in E_{J'}} a_{ij}^2  \geq \frac{(1-\tau)^2 c^2 }{4} \lambda_{\alpha } ^2  \right) }{\log n} = \infty.
\end{align*}
By the  assumption on the component size,
the cardinality of different values that  a random subset $J$ with $|J| \leq \left \lfloor \frac{1}{\xi ^2}  \right \rfloor$ can  take is bounded by $  \left ( ( 1 + \delta_1 )  \frac{1}{   \e} \frac{\log n}{\log \log n} \right )^\frac{1}{\xi ^2} = n^{o(1)}$.
Thus, by a union bound,
\begin{align}\label{419}
\lim_{n \to \infty} \frac{ - \log \P (2\lambda_L \geq (1-\tau)c \lambda_{\alpha }   ) }{\log n}
= \infty.
\end{align}
Therefore,  applying \eqref{415} and \eqref{419} to  \eqref{decompose_eigenvalue},
 \begin{align*}
 \liminf_{n\rightarrow \infty} \frac{-\log \P ( \lambda_1(A) \geq  c \lambda_{\alpha }  )  }{\log n} \geq  \frac{c^\alpha}{2^\alpha (2+\delta_2)^\frac{\alpha}{4} \xi^\frac{\alpha}{2}}    \frac{2}{\alpha-2} \Big ( 1 - \frac{2}{\alpha}  \Big )^\frac{\alpha}{2} \Big(  \frac{1 + \delta_1}{  \e}\Big)^{1-\frac{\alpha}{2}}  - (1+\delta_1)  .
 \end{align*}
Since $\xi>0$ is arbitrary, the RHS above  can be made arbitrarily large, concluding the proof.
\end{proof}

{Since each connected component of   $X^{(1)}_2$ satisfies the  conditions in Proposition \ref{prop:eigenvalue_components_Z^1_2} by the discussion following Lemma \ref{BBG_decomposition}    with high probability, this completes Step \eqref{outline_light_step_Z^1_2_negligible}. }Therefore, it remains to analyze the spectral behavior of $Z^{(1)}_1$, a collection of disjoint stars.
We will  group these stars according to their sizes and then show that the main contribution comes from the group of stars with degrees close to $\gamma_\delta \frac { \log n}{\log \log n}$.

As a preparation, we now introduce some notations and a few lemmas.
The first lemma concerns the spectral behaviour of a single weighted star. {Recall from \eqref{g} that we set
$
g(\gamma) = \left \lceil  \gamma \frac{\log n}{\log \log n} \right \rceil
$
}
and let us  define, for a star graph $S$, $d(S)$ to be the degree of the root vertex of $S$.

\begin{lemma} \label{lem:prob_one_star}
Suppose that $S$ is a weighted star graph such that $d(S) \leq g(\gamma)$ for some $\gamma>0$, with i.i.d. weights given by the Weibull distributions with  a shape parameter $\alpha>2$ conditioned to be greater than $(\e \log \log n )^\frac{1}{\alpha}$ in absolute value.
Then, for any $\rho>0$,
\begin{align*}
\liminf_{ n \to \infty} \frac{ - \log \P ( \lambda_1( S) \geq (1+\rho) \lambda_{\alpha }  ) }{\log n } \geq (1+\rho)^\alpha \frac{2}{\alpha-2} \left (1 - \frac{2}{\alpha}  \right )^\frac{\alpha}{2} \gamma^{1- \frac{\alpha}{2}} -   \e \gamma.
\end{align*}
\end{lemma}

\begin{proof}
 Let $ \left \{ \tilde{Y}_i \right \}_{i=1,2,\cdots}$ be i.i.d.   Weibull random variables with a  shape parameter $\alpha>2$ conditioned to be greater than $  (\e \log \log n)^{ \frac{1}{\alpha}}$ in absolute value.  Since the largest eigenvalue of a weighted star is nothing other than the square root of the sum of squares of edge-weights (see Lemma \ref{spectrum_star_weights}),
\begin{align*}
\P \big ( \lambda_1( S) \geq (1+\rho) \lambda_{\alpha } \big ) &\leq  \P \left ( \tilde{Y}_1^2 + \dots +\tilde{Y}_{ g(\gamma)} ^2 \geq (1+\rho)^2 \lambda_{\alpha } ^2 \right ).
\end{align*}
By the  tail bound \eqref{cond1} with $d=1 + \rho$ and $b=\gamma$, we are done.
 
\end{proof}

Next, we  estimate the spectral contribution from the group of stars with degree close to $g(\gamma)$. For this we first introduce some additional notations.
Let {$d \left ( X^{(1)}, v \right )$} be the degree of $v$ in the graph $X^{(1)}$, and for  $\gamma \geq 0$, define
 \begin{align} \label{degree}
 D^{(1)}_\gamma  = \bigg \{v\in V : d \big ( X^{(1)}, v \big ) \geq  g(\gamma) \bigg \}.
\end{align}
For small enough constant $\kappa>0$ which will be chosen later, 
define $m$ to be an integer such that $m \kappa  < 1 \leq (m+1)\kappa $. 
Then, define the event measurable with respect to  $X^{(1)}$:
\begin{equation} \label{eq:disc_deg_bounded}
{\cP_{\kappa}} := \Big \{ \big | D^{(1)}_{i \kappa} \big | \leq n^{1- i\kappa + \kappa } \text{ for all }  i  = 0 ,1,\cdots,m\Big \},
\end{equation} 
{which guarantees that for the discretization $\{\kappa, 2\kappa, \dots, m\kappa \}$ of the interval $(0,1),$ there are not unusually many vertices whose degrees fall} into any bin of degree range given by the discretization.

Additionally we define the event   measurable with respect to  $X^{(1)}$: 
\begin{equation} \label{eq:big_deg_bounded}
\cL_{\delta,\kappa} := \bigg \{ \big | D^{(1)}_{1+\kappa} \big | \leq \frac{(1+\delta)^2}{\kappa} \bigg \},
\end{equation} 
which  guarantees that there are uniformly bounded   number of vertices of unusually large degree.

Using the estimate for the contribution of a single star (Lemma \ref{lem:prob_one_star}), we now prove a lemma that captures the contribution from the group of stars of degree close to $g(\gamma)$.  {Recall from Lemma \ref{BBG_decomposition} that  $X^{(1)}_1$  is  a vertex-disjoint union of stars, and $\kappa>0$ is a given   constant.}

\begin{lemma}  \label{lem:stars_contribution} Let $\mathcal{S}$ be the collection of stars in $X^{(1)}_1$. Moreover, define, for any $h>0$ and $\gamma = i \kappa < 1$  ($i$ is a non-negative integer) or $\gamma \geq 1 +\kappa$, 
\begin{equation}
\lambda_{\max}(\gamma, h) := \underset{S \in \mathcal{S}, d(S) \in   (g(\gamma), g(\gamma + h)]}{\max} \{\lambda_1(S)\}
\end{equation}
if there is a star $S \in \mathcal{S}$ satisfying $d(S) \in   (g(\gamma), g(\gamma + h)]$, and set  $\lambda_{\max}(\gamma, h) := 0$ otherwise.
Then, for any $\rho > 0$,
\begin{align} \label{570}
 \liminf_{n \to \infty} &\frac{ - \log  \E  \Big[ \P \Big ( \lambda_{\max}(\gamma, h) \geq (1+\rho) \lambda_{\alpha }  \mid  X^{(1)} \Big ) \1_{   {\cP_{\kappa} \cap \cL_{\delta,\kappa}} } \Big] } {\log n}  \nonumber \\
& \geq  - f_{\alpha,\rho} (\gamma + h) - h - \kappa -   \e (\gamma + h),
\end{align}
where the function $f_{\alpha,\rho}: (0,\infty)  \rightarrow \mathbb{R} $ is defined by
\begin{align} \label{f}
f_{\alpha,\rho}(x) := 1 -x -  (1+\rho)^\alpha \frac{2}{\alpha-2} \left (1 - \frac{2}{\alpha}  \right )^\frac{\alpha}{2} x^{1- \frac{\alpha}{2}}.
\end{align}
\end{lemma}
{In \eqref{570}, the quantity $f_{\alpha,\rho} (\gamma + h)$ should be thought of as the dominant term with the rest being error terms.}
\begin{proof}
The proof  depends on whether $\gamma<1$ or $\gamma \geq 1+\kappa$.
In the case $\gamma = i\kappa<1$,  the number of stars in  $X^{(1)}_1$ of degree at least $g(\gamma)$ {is bounded by $ n^{1-\gamma+\kappa}$ under the event $\cP_{\kappa}$}. Hence, by union bound and    Lemma \ref{lem:prob_one_star}  with   $\gamma+h$ in place of $\gamma$,
\begin{align*}
\liminf_{n\rightarrow \infty} &\frac{-\log \E \Big[ \P  \Big ( \lambda_{\max}(\gamma, h) \geq (1+\rho) \lambda_{\alpha }  \mid X^{(1)} \Big ) \1_{\cP_\kappa}  \Big]}{\log n}
\\
 &\geq  - ( 1-\gamma+  \kappa) + 
 (1+\rho)^\alpha \frac{2}{\alpha-2} \left (1 - \frac{2}{\alpha}  \right )^\frac{\alpha}{2} ( \gamma +  h)^{1- \frac{\alpha}{2}} -   \e (\gamma+h)\\
&= - f_{\alpha,\rho}(\gamma+h) - h - \kappa  -   \e (\gamma+h).
\end{align*}

Let us now consider the second case $\gamma \geq  1 + \kappa$.   Note that by Lemma \ref{lem:prob_one_star}, for any star $S$ with $d(S) \leq g(\gamma+h)$,
\begin{align*}
\P (\lambda_1(S) \geq   (1+\rho) \lambda_{\alpha }  ) \leq  n^{ -  (1+\rho)^\alpha \frac{2}{\alpha-2} \left (1 - \frac{2}{\alpha}  \right )^\frac{\alpha}{2} ( \gamma + h) ^{1- \frac{\alpha}{2}}  +    \e(\gamma+h)+o(1)}.
\end{align*} 
Thus, since $\lambda_{\max}(\gamma, h)=0$ if there is no star  $S$ in $\cS$ satisfying $d(S) > g(\gamma)$,   we have
{\begin{align*}
\E & \Big[ \P  \Big ( \lambda_{\max}(\gamma, h) \geq (1+\rho) \lambda_{\alpha }  \mid  X^{(1)} \Big ) \1_{\cL_{\delta,\kappa}}  \Big]\\
&=
 \E \Big[ \P  \Big ( \lambda_{\max}(\gamma, h) \geq (1+\rho) \lambda_{\alpha }  \mid  X^{(1)} \Big ) \1_{\cL_{\delta,\kappa}} \1_{\{ \exists S \in \mathcal{S}: d(S) > g(\gamma )\}}  \Big]\\
& \leq  \frac{(1 + \delta)^2}{\kappa} \cdot  n^{ -  (1+\rho)^\alpha \frac{2}{\alpha-2} \left (1 - \frac{2}{\alpha}  \right )^\frac{\alpha}{2} ( \gamma + h) ^{1- \frac{\alpha}{2}}  +    \e(\gamma+h)+o(1)} \P  (  d_1(X^{(1)}) \geq g(\gamma)  ),
\end{align*}}
where the last inequality follows from  the  union bound (under the event $\cL_{\delta, \kappa}$, the number of stars in  $X^{(1)}_1$ of degree at least $g(\gamma) \geq  g(1+\kappa)  $ is bounded by $\frac{(1 + \delta)^2}{\kappa}$).  Also, since $\gamma>1$, by \eqref{eq:degrees_bbg} and the fact that $X^{(1)}$ is distributed as $\calG_{n,q}$ with $q \leq  p = \frac{d}{n}$,
\begin{align*}
\P  (  d_1(X^{(1)}) \geq g( \gamma )   ) \leq  n^{1-\gamma+o(1)}.
\end{align*}  
Therefore, 
\begin{align*}
 \liminf_{n \to \infty} & \frac{ - \log\E  \Big[ \P  \Big (  \lambda_{\max}(\gamma, h)  \geq (1+\rho) \lambda_{\alpha }  \mid  X^{(1)} \Big ) \1_{\cL_{\delta,\kappa}}  \Big] } {\log n} \\
& \geq  - (1- \gamma )  + (1+\rho)^\alpha \frac{2}{\alpha-2} \left (1 - \frac{2}{\alpha}  \right )^\frac{\alpha}{2} (\gamma + h)^{1-\frac{\alpha}{2}} -   \e (\gamma + h)\\
& = - f_{\alpha,\rho}(\gamma + h) - h -   \e (\gamma + h).
\end{align*} 
\end{proof}

Having the  expression for the contribution of any group of stars under the assumption that the underlying graph is reasonably nice, we now identify the group of stars for which this contribution is maximized. This is done in the following technical lemma by optimizing the value of the function $f_{\alpha,\rho}$. {Note that $f_{\alpha,\rho}$ was originally defined  for $\rho>0$, but below we consider the wider range $\rho>-1$  for a later application.}

\begin{lemma} \label{star_probability_max}
For $\alpha>2$ and $\rho>-1$,
{recall the  function  $f_{\alpha,\rho}: (0,\infty) \rightarrow \mathbb{R}$  in  \eqref{f}:}
\begin{align*}
f_{\alpha,\rho}(\gamma) = 1 -\gamma -  (1+\rho)^\alpha \frac{2}{\alpha-2} \left (1 - \frac{2}{\alpha}  \right )^\frac{\alpha}{2} \gamma^{1- \frac{\alpha}{2}}.
\end{align*}
Then,
\begin{equation}
\max_{\gamma>0} f_{\alpha,\rho}(\gamma) = 1 - (1+\rho)^2\quad 
\text{ and }  \quad 
\gamma_\rho := \argmax_{\gamma > 0} f_{\alpha,\rho}(\gamma) =  (1+\rho)^2 \left (  1- \frac{2}{\alpha} \right ).
\end{equation}
 
\end{lemma}

\begin{proof}
For the sake of  readability, we will drop the subscripts of $f_{\alpha,\rho}$ in the proof. Note that 
\begin{align*}
\frac{d}{d\gamma} f(\gamma) 
= -1 + (1+\rho)^\alpha \left (1 - \frac{2}{\alpha}  \right )^\frac{\alpha}{2} \gamma^{-\frac{\alpha}{2}},
\end{align*}
and thus $f$ is maximized at  $ \gamma=(1+\rho)^2 \left (  1- \frac{2}{\alpha} \right )$.
Plugging this back into $f(\gamma)$, we get
$
\max_{\gamma >0} f(\gamma)  =  1 - (1+\rho)^2.
$
 
\end{proof}

We are now ready to put all of this together to prove the upper bound of the upper tail. 

\begin{proof}[Proof of the upper bound of the upper tail.]
By using the decomposition \eqref{100}-\eqref{101}, we  write $Z = Z^{(1)} + Z^{(2)}$. {Since $\lambda_1(Z) \leq \lambda_1 \big ( Z^{(1)} \big ) + \lambda_1 \big (Z^{(2)} \big )$ by Lemma \ref{spectrum_graphs},   we have
\begin{align} \label{311}
& \P  (\lambda_1(Z)\geq (1+\delta) \lambda_{\alpha } )  \leq 
\P \Big (\lambda_1(Z^{(1)}) \geq (1+\delta) \Big ( 1 - \frac{\e^\frac{1}{\alpha}}{B_{\alpha }} \Big ) \lambda_{\alpha }  \Big )  +
\P \Big (\lambda_1(Z^{(2)}) \geq \e^\frac{1}{\alpha} (1+\delta)  \frac{\lambda_\alpha}{B_\alpha} \Big ).
\end{align}
}
By Lemma \ref{lem:bulk_negligible},  the second term above can be bounded by
{
\begin{align} \label{310}
\P \Big (\lambda_1 \big (Z^{(2)} \big ) \geq \e^\frac{1}{\alpha} (1+\delta) \frac{\lambda_\alpha}{B_\alpha}  \Big ) \leq n^{1-(1+\delta)^2 + o(1)}.
\end{align}}
Hence, it suffices to  bound the probability
\begin{align} \label{312}
\P \Big (\lambda_1(Z^{(1)}) \geq (1+\delta) \Big ( 1 - \frac{\e^\frac{1}{\alpha}}{B_{\alpha }} \Big ) \lambda_{\alpha }  \Big ) . 
\end{align}  

\textbf{Step 1.}  
{Given the previous results, we will work on the event ensuring:}
\begin{enumerate}[(1)]
\item   Existence of the decomposition of  $X^{(1)}$ into  $X^{(1)}_1$ (vertex-disjoint union of  stars) and $X^{(1)}_2$ (relatively small maximum degree).
\item   All connected components of $X^{(1)}_2$ are relatively small and tree-like.
\item  $X^{(1)}_1$ {has a controlled number of  stars of each given degree.}
\end{enumerate}
{The first condition is achieved by the event $\cW$  in Lemma \ref{BBG_decomposition},  and the second one  is fulfilled by the series of events  in  Section \ref{connectivity} (applied to $X^{(1)}  \overset{\text{d}}{\sim} \cG_{n,q} $). For the last condition,   we  consider the events $\cP_{\kappa}$ and $\cL_{\delta,\kappa}$  in \eqref{eq:disc_deg_bounded} and \eqref{eq:big_deg_bounded} respectively.}\\
\noindent
The events above make up {the event} $\cK_0$  measurable with respect to $X^{(1)}$:
\begin{align} \label{nice_event_upper_light}
\mathcal{K}_0 &:=  \mathcal{W} \cap \cD_{(1+\delta)^2-1} \cap  \mathcal{C}_{\e,(1+\delta)^2-1}  \cap   \mathcal{E}_{(1+\delta)^2-1}  \cap \cP_\kappa \cap  \cL_{\delta,\kappa}.
\end{align}
Using the previously proven or cited results, we have
\begin{align} 
&\P( \mathcal{W}^c ) \leq e^{-\omega(\log n )} \quad \text{ by Lemma \ref{BBG_decomposition}},  \nonumber\\
& 
\P  (   \cD^c_{(1+\delta)^2-1} ) \leq n^{ 1 - (1+\delta)^2  + o(1)} \quad  \text{ by Lemma \ref{lemma 33}}, \nonumber\\
& \P ( \mathcal{C}_{\e,(1+\delta)^2-1}^c )
 \leq n^{1 - (1+\delta)^2+o(1)}  \quad \text{ by Lemma  \ref{lem:biggest_component},} \nonumber\\
& \P(\mathcal{E}_{(1+\delta)^2-1}^c) \leq n^{1 - (1+\delta)^2+o(1)} \quad  \text{ by Lemma  \ref{lem:excess_edges},}
 \label{light_upper_nice_event_prob}\\ 
& \P(\cP_\kappa^c) \leq n^{ - (1+\delta)^2} \quad  \text{ by Proposition \ref{degree_distribution} with } \mu = \frac{(1+\delta)^2}{\kappa},  \nonumber\\
&
\P (\cL_{\delta, \kappa}^c) \leq n^{-(1+\delta)^2 + o(1) }  \quad \text{ by \eqref{eq:degrees_bbg}}. \nonumber
\end{align}
Note that Proposition \ref{degree_distribution} was proven for the random graph  $\calG_{n,\frac{d}{n}}$,  whereas the events $\cP_\kappa$ and $\cL_{\delta,\kappa}$ are defined in terms of the sparser graph $X^{(1)} 
\overset{\text{d}}{\sim} \cG_{n,q}$. 
{However, since these events are decreasing the same bounds hold.}
Combining these  together,
\begin{align} \label{314}
\P ( \mathcal{K}_0^c) \leq  n^{1-(1+\delta)^2+o(1)}.
\end{align}
Since
$\lambda_1(Z^{(1)}) \leq \lambda_1(Z^{(1)}_1) + \lambda_1(Z^{(1)}_2)$, defining $\delta'$ as
 \begin{align} \label{delta'}
 (1 + \delta) \Big ( 1 -  \frac{\e^\frac{1}{\alpha}}{B_{\alpha }} \Big ) =  (1+\delta' ) + \e (1+\delta),
\end{align} we have
\begin{align} \label{313}
\P \left ( \lambda_1 \big (Z^{(1)} \big ) \geq (1+\delta) \left ( 1 - \frac{\e^\frac{1}{\alpha}}{B_{\alpha }} \right ) \lambda_{\alpha }  \right ) & \nonumber  \leq   \E \left  [ \P  \left  (\lambda_1 \big (Z^{(1)}_1 \big ) \geq (1+\delta') \lambda_{\alpha }  \mid  X^{(1)} \right ) \1_{\mathcal{K}_0}  \right ]  \\
& \hspace{-1cm}  + \E  \left  [ \P \left   ( \lambda_1 \big (Z^{(1)}_2 \big ) \geq \e(1+\delta) \lambda_{\alpha }    \mid  X^{(1)}   \right ) \1_{\mathcal{K}_0} \right  ]  + \P(\mathcal{K}_0^c).
\end{align}
From now on, we estimate the quantities in the RHS above.

\textbf{Step 2. Contribution from $\bm{Z^{(1)}_2}$.} 
Under the event $\cC_{\e,(1+\delta)^2-1} \cap \cE_{(1+\delta)^2-1}$, and hence under  the event $\mathcal{K}_0$, each connected component  of  $X^{(1)}$, and thus of  its subgraph $X^{(1)}_2$, satisfies the conditions in Proposition \ref{prop:eigenvalue_components_Z^1_2} with $\delta_1 = \delta_2 = (1+\delta)^2 - 1$ (recall that the largest degree of $X^{(1)}_2$ is $o(\frac{\log n}{\log \log n})$). Since the number of connected components is bounded by $n$,  by Proposition \ref{prop:eigenvalue_components_Z^1_2}  combined with a union bound,  
 whenever the event $\mathcal{K}_0$ holds, 
\begin{align}\label{315}
\lim_{n \to \infty} & \frac{ - \log \P  ( \lambda_1(Z^{(1)}_2) \geq\e (1+\delta) \lambda_{\alpha }   \mid  X^{(1)}  )   }{ \log n }  = \infty. 
\end{align}

\textbf{Step 3. Contribution from $\bm{Z^{(1)}_1}$.} 
Let  $M$ be the smallest integer such that $(1+\delta)^2  <  M \kappa$. Note that this in particular implies that
\begin{align} \label{bound m}
M \leq  \frac{(1+\delta)^2}{\kappa}+1.
\end{align} 
Note that  under the event $\cD_{(1+\delta)^2-1}$, $  d_1(X_1^{(1)})\leq d_1(X^{(1)}) \leq (1 + \delta)^2 \frac{\log n}{\log \log n}  $. Thus, the degree of any star $S$  in $X_1^{(1)}$  falls into in one of the following (not necessarily disjoint) categories:
\begin{enumerate}
\item $d(S) \leq g( \kappa)$.
\item  $d(S) \in ( g( i \kappa), g( (i+2) \kappa ) ] $ for  $i = 1,\cdots, M-1$ and $i \neq m+1$ (recall that  $m$ is a unique integer such that $ m \kappa < 1 \leq  (m+1)\kappa $),
\end{enumerate}
{Reiterating Remark \ref{remark critical}, the reason  why we exclude $i =  m+1$ is essentially because we do not have a precise understanding on the number of vertices of degree close to $\frac{\log n}{\log \log n}$. }

To bound the contribution of the first category, by Lemma \ref{lem:prob_one_star} (with $\gamma = \kappa$ and $\rho = \delta'$) combined with a union bound, for sufficiently small $\kappa,\e>0$,
{\begin{align} \label{321}
 \liminf_{n \to \infty}& \frac{ - \log  \E \Big[ \P \Big ( \max_{\substack{ S \in \mathcal{S} \\ d(S) \leq g(\kappa)  }} \{ \lambda_1 (S) \} \geq (1+ \delta')  \lambda_{\alpha } \mid  X^{(1)} \Big ) \1_{\mathcal{K}_0} \Big]  }{ \log n } \nonumber  \\
 & \qquad \geq
 -1 + (1+\delta')^\alpha \frac{2}{\alpha -2 } \left ( 1 - \frac{2}{\alpha} \right )^\frac{\alpha}{2}\kappa^{1-\frac{\alpha}{2}} -   \e \kappa  > (1+\delta)^2 -1.
\end{align}}
Note that the additional `$-1$' term in the middle quantity {comes from a union bound  (the  number of stars is bounded by $n$)}, and
the last inequality holds once $\kappa>0$ is sufficiently small (recall that $\alpha>2$).

For the  group of stars in the second category,
by Lemma \ref{lem:stars_contribution} (with $\gamma=i\kappa$, $h=2\kappa$ and $\rho = \delta'$), for each $i=1,\cdots,M-1$ with $i\neq m+1$,
\begin{align} \label{323}
& \liminf_{n \to \infty} \frac{ - \log \E \Big[  \P \Big ( \max_{\substack { S \in \mathcal{S} \\ d(S) \in ( g( i \kappa), g( (i+2) \kappa ) ] }} \{ \lambda_1 (S)\} \geq (1+ \delta')  \lambda_{\alpha }  \mid  X^{(1)} \Big ) \1_{\mathcal{K}_0} \Big]  }{ \log n } \nonumber \\
&\geq - f_{\alpha,\delta'}( (i + 2) \kappa) - 2 \kappa - \kappa -   \e  (i + 2) \kappa    \nonumber  \\
&\geq   (1+\delta')^2 -1 - 3 \kappa -   \e (M + 1) \kappa  \overset{\eqref{bound m}}{\geq}  (1+\delta')^2 -1 - 3 \kappa -   \e \Big ( \frac{(1+\delta)^2}{\kappa}  + 2\Big) \kappa   =: L,
\end{align}
where we used Lemma \ref{star_probability_max} to bound $f_{\alpha,\delta'}( (i + 2) \kappa)$ in the second inequality.

Since the categories considered in \eqref{321} and \eqref{323} make up the total contribution of the stars  in the network $Z^{(1)}$, by a  union bound,
\begin{align} \label{325}
\liminf_{n \to \infty} 
&\frac{ - \log   \E  [ \P   (\lambda_1(Z^{(1)}_1) \geq (1+\delta') \lambda_{\alpha }    )  \mid  X^{(1)}   ) \1_{\mathcal{K}_0}   ]   }{\log n} \geq \min \{ (1+\delta)^2 - 1, L \}.
\end{align} 
Applying  the bounds \eqref{314}, \eqref{315} and \eqref{325} to  \eqref{313}, 
\begin{align*}
\liminf_{n \to \infty} 
&\frac{ - \log   \P \Big (\lambda_1(Z^{(1)}) \geq (1+\delta) \Big ( 1 - \frac{\e^\frac{1}{\alpha}}{B_{\alpha }} \Big ) \lambda_{\alpha }  \Big ) }{\log n}  \geq   \min \{ (1+\delta)^2 - 1, L \}.
\end{align*}
Applying this with \eqref{310} to \eqref{311}, we obtain
 \begin{align*}
\liminf_{n \to \infty} 
&\frac{ - \log   \P (\lambda_1(Z) \geq (1+\delta)  \lambda_{\alpha }  ) }{\log n}  \geq   \min \{ (1+\delta)^2 - 1, L \}.
\end{align*}
Since $\lim_{\e \rightarrow 0} \delta' = \delta$ (see \eqref{delta'} for the definition of $\delta'$), the quantity $L$  defined in \eqref{323} becomes sufficiently close to $(1+\delta)^2-1$  for  small enough $\e,\kappa>0$, which completes the proof.

\end{proof}
 
 {
\begin{remark}
The above proof indicates that the following structural result  holds: {Conditioned on the upper tail event $  \{\lambda_1(Z) \geq (1+\delta) \lambda_{\alpha }\}$}, with high probability, $X$ contains a star of size roughly   $\gamma_\delta \frac{\log n}{\log \log n}$ with edge-weights  greater than  $ ( \frac{2}{\alpha-2} \log \log n )^{\frac{1}{\alpha}}$ in absolute value. {Indeed,  since $ f_{\alpha,\rho}(\gamma)$ at  $ \gamma = \gamma_\rho $ is a strict maximum, by Lemma \ref{lem:stars_contribution},  the contribution from the stars of degree $ g(\gamma) $ with $\gamma \notin (\gamma_\delta - \chi, \gamma_\delta + \chi)$ for $\chi>0$ is negligible compared to that from the  stars  of  degree $ g(\gamma_\delta) $.}

This, combined with Remark \ref{condition1} in the Appendix, which says that if the sum of squares of light-tailed random variables is large then these random variables tend to be uniformly large, implies that with high probability conditionally on the upper tail event, there exists a star of degree {close to} $ \gamma_\delta \frac{\log n}{\log \log n}$ with {edge-weights close to}
\begin{align*}
 \left ( \frac{(1+\delta)^2 \lambda_\alpha^2 }{ \gamma_\delta  \frac{\log n}{\log \log n}} \right )^{\frac{1	}{2}} = \Big ( \frac{2}{\alpha-2} \log \log n \Big)^{\frac{1}{\alpha}}
\end{align*}
in absolute value. In other words, the optimal size of the star increases in $\delta$, whereas the edge-weights on the star, while atypically large, do not depend on the amount of deviation  $\delta$. 

\end{remark}
 }
 
\subsection{The lower tail} \label{5.2}
 
We now move on to prove a large deviation result for the lower tail, that we restate here for the reader's convenience.

\lightlower*

{Analogous to  the upper tail case, the governing structure in this case will turn out to be the collection of $n^{1-\gamma_\delta' }$  vertex-disjoint  stars of degree  close to $\gamma_{\delta}' \frac{\log n}{\log \log n}$ with}
\begin{align} \label{gammaprime}
\gamma_{\delta}' := (1-\delta)^2 \left ( 1 - \frac{2}{\alpha} \right).
\end{align}
Note that $\gamma'_\delta$ is nothing but $\gamma_{-\delta}$ from \eqref{31}.

\subsubsection{Lower bound for the lower tail}
 
Before embarking on the proof, we establish a lemma about the lower tail behavior of the maximum  among the largest eigenvalue of $n^{1-\gamma +o(1)}$ weighted stars of degree close to $g(\gamma) = \left \lceil  \gamma \frac{\log n}{\log \log n} \right \rceil$.
 
\begin{lemma} \label{lower lower}

Suppose that  $\gamma>0$, $h\geq 0$ and $\kappa \geq \gamma-1$. Let $\mathcal{S}$ be a collection of at most $n^{1-\gamma + \kappa}$  vertex-disjoint weighted stars of size less than $g(\gamma + h)$. Assume that  edge-weights are  i.i.d.  Weibull distributions with a shape parameter $\alpha>2$  conditioned to be greater than $(\e \log \log n )^\frac{1}{\alpha}$ in absolute value. Then, for any $0<\rho<1$,
{\begin{equation} \label{510}
 \limsup_{n \to \infty} \frac{  1}{\log n}  \left ( \log \log \frac{1}{\P  \left ( \max_{ S \in \mathcal{S} } \{ \lambda_1(S) \} \leq ( 1- \rho )\lambda_{\alpha } \right  )} \right )
 \leq 
 f_{\alpha,-\rho}(\gamma +h ) + \kappa  + h +   \e (\gamma + h),
\end{equation}}
where the function $ f_{\alpha,-\rho}$ is as defined in \eqref{f}:
\begin{align*}
f_{\alpha,-\rho}(x) = 1 -x -  (1-\rho)^\alpha \frac{2}{\alpha-2} \left (1 - \frac{2}{\alpha}  \right )^\frac{\alpha}{2} x^{1- \frac{\alpha}{2}}.
\end{align*}
\end{lemma}

{Again, as before, above $\kappa  + h +   \e (\gamma + h)$ should be thought of as an error term.}

\begin{proof}
{We use the notation $ \left \{ \tilde{Y}_i \right \}_{i=1,2,\cdots}$ from the proof of  Lemma \ref{lem:prob_one_star}}. By  Lemma \ref{spectrum_star_weights},
\begin{align*}
\P \big ( \lambda_1( S) \geq (1-\rho) \lambda_{\alpha } \big ) &\leq  \P \left ( \tilde{Y}_1^2 + \dots +\tilde{Y}_{ g(\gamma+h)} ^2 \geq (1-\rho)^2 \lambda_{\alpha } ^2 \right ).
\end{align*}
By the tail estimate \eqref{cond1} with $d=1-\rho$ and $b=\gamma+h$, this probability is {upper bounded} by
\begin{align*}
n^{  \e   (\gamma +h)  - (1-\rho)^\alpha\frac{2}{\alpha - 2} \left ( 1 - \frac{2}{\alpha} \right )^\frac{\alpha}{2}  (\gamma  + h)^{1-\frac{\alpha}{2} } + o(1) },
\end{align*}
Thus, using that the number of stars in $\mathcal{S}$ is bounded by $ n^{1-\gamma + \kappa}$, by the independence of edge-weights,
\begin{align*}
\P \left  ( \max_{ S \in \mathcal{S} } \{ \lambda_1(S) \} \leq ( 1- \rho) \lambda_{\alpha }   \right )
& \geq 
\left (  1 - n^{  \e (\gamma +h)   -(1-\rho)^\alpha  \frac{2}{\alpha - 2} \left ( 1 - \frac{2}{\alpha} \right )^\frac{\alpha}{2}   (\gamma + h)^{1-\frac{\alpha}{2} } + o(1) } \right )^{n^{1-\gamma + \kappa }}\\
& \geq 
\exp \left ( - n^{1- \gamma+\kappa+    \e (\gamma  + h)   - (1-\rho)^\alpha  \frac{2}{\alpha - 2} \left ( 1 - \frac{2}{\alpha} \right )^\frac{\alpha}{2}  (\gamma  +h) ^{1-\frac{\alpha}{2} } + o(1) } \right ) \\
&= \exp \big(-n^{ f_{\alpha,-\rho}(\gamma +h ) + \kappa  + h +   \e (\gamma + h) + o(1)}\big),
\end{align*}
where the second inequality follows since $1 - x > e^{-2x}$ for small $ x>0$ and the constant can be absorbed into $n^{o(1)}$. 
\end{proof}

\begin{proof}[Proof of the lower bound of the lower tail]

$\empty$

\textbf{Step 1.}
Using the decomposition \eqref{100}-\eqref{101}, we write $Z = Z^{(1)} + Z^{(2)}$. First, we define the  event measurable with respect to $X$:
\begin{align} 
\mathcal{B}_\delta := \Big \{ \lambda_1\left (X \right ) \leq (1+\delta) \frac{(\log n)^\frac{1}{2}}{(\log \log n)^{\frac{1}{2}} } \Big \}.
\end{align}

As before, conditioning on the event $\cW$ defined in Lemma \ref{BBG_decomposition}, allows us to decompose $Z^{(1)}$ into  $Z^{(1)}_1$ (vertex-disjoint union of  stars) and $Z^{(1)}_2$ (relatively small maximum degree). 
 Let
\begin{align*}
\cR_\kappa :  = \{ \text{Maximum degree in}  \  X^{(1)}  \ \text{is less than}  \ g(1+\kappa) \}.
\end{align*}
We now define an event similar to \eqref{nice_event_upper_light} by additionally excluding  the existence of an  atypically large degree vertex using the above event, by defining the event $\cK_1$ which is measurable with respect to $\{X,X^{(1)}\}$:
\begin{equation}
\cK_1 : = \cB_\delta \cap \mathcal{W}  \cap   \mathcal{C}_{\e,(1+\delta)^2 -1} \cap \mathcal{E}_{(1+\delta)^2 -1} \cap \cP_{\kappa} \cap \cR_\kappa.
\end{equation}

By Theorem \ref{thm:eigenvalues_bbg}, $ \lim_{n\to \infty} \mathbb{P}(\mathcal{B}_\delta) = 1$. Also, by  \eqref{eq:degrees_bbg},  $ \lim_{n\to \infty} \mathbb{P}(\mathcal{R}_\kappa) = 1$ (note that  although \eqref{eq:degrees_bbg} is stated for the random graph $\cG_{n,\frac{d}{n}}$, since $\cR_\kappa$ is a decreasing event and $X^{(1)}$ is sparser than $\cG_{n,\frac{d}{n}}$, we still have this estimate).    Together with the analysis  in \eqref{light_upper_nice_event_prob}, we have
\begin{align} \label{k1}
 \lim_{n\to \infty} \mathbb{P}(\mathcal{K}_1) = 1.
\end{align}
{Since $\lambda_1(Z) \leq   \lambda_1 \big ( Z^{(1)}_1 \big ) + \lambda_1 \big ( Z^{(1)}_2 \big ) + \lambda_1 \big ( Z^{(2)} \big ) $}, setting $\delta'$ via
\begin{align} \label{delta'''}
 1-\delta  = (1-\delta' ) +  (1+\delta) \left(  \e + \frac{ \e^\frac{1}{\alpha}}{B_\alpha}  \right ),
\end{align}
we have
\begin{align} \label{lower_tail_lower_bound_decomp}
 \P \big ( \lambda_1 ( Z ) \leq (1-\delta) \lambda_{\alpha }  \big ) \geq 
\E \Bigg  [ \;  \P \bigg ( & \;  \lambda_1 \big ( Z^{(1)}_1 \big ) \leq  (1-\delta') \lambda_{\alpha }  , \quad \lambda_1 \big ( Z^{(1)}_2 \big ) \leq  (1+\delta)  \e \lambda_{\alpha } , \nonumber \\
&  \lambda_1 \big ( Z^{(2)} \big ) \leq (1+\delta)  \frac{ \e^\frac{1}{\alpha} }{B_\alpha}  \lambda_\alpha  \; \Big |  \; X, X^{(1)} \bigg )  \; \1_{\mathcal{K}_1} \; \Bigg ].
\end{align}

First of all, under  the event $\mathcal{B}_\delta$, and hence under the event $\mathcal{K}_1$, by the same argument as in \eqref{Z2_wrt_X2},
\begin{align} \label{573}
\lambda_1 \big ( Z^{(2)} \big ) \leq (1+\delta)  \frac{ \e^\frac{1}{\alpha} }{B_\alpha}   \lambda_\alpha.
\end{align}

Furthermore, note that   $ Z^{(1)}_1 ,Z^{(1)}_2 $ and $X$  are conditionally independent given $X^{(1)}$. 
Thus by \eqref{573}, under the event $\mathcal{K}_1$, the conditional probability inside the expectation in  \eqref{lower_tail_lower_bound_decomp} is written as 
{ \begin{align} \label{577}
\P & \left (  \lambda_1 \big ( Z^{(1)}_1 \big ) \leq  (1-\delta') \lambda_{\alpha }  , \lambda_1 \big ( Z^{(1)}_2 \big ) \leq  (1+\delta) \e \lambda_{\alpha }   \mid  X, X^{(1)} \right  ) \nonumber \\
=\P & \left (  \lambda_1 \big ( Z^{(1)}_1 \big ) \leq  (1-\delta') \lambda_{\alpha }  , \lambda_1 \big ( Z^{(1)}_2 \big ) \leq  (1+\delta) \e \lambda_{\alpha }   \mid   X^{(1)} \right  ) \nonumber \\
=  \P & \left (  \lambda_1 \big ( Z^{(1)}_1 \big ) \leq  (1-\delta') \lambda_{\alpha }    \mid  X^{(1)} \right ) \P \left ( \lambda_1 \big ( Z^{(1)}_2 \big ) \leq (1+\delta) \e \lambda_{\alpha }   \mid  X^{(1)} \right ).
\end{align}}
Therefore, by \eqref{lower_tail_lower_bound_decomp}-\eqref{577},
\begin{align} \label{30}
 \P \big ( \lambda_1 ( Z ) \leq (1-\delta) \lambda_{\alpha }  \big ) \geq 
\E \Big[ \P \left (  \lambda_1 \big ( Z^{(1)}_1 \big ) \leq  (1-\delta') \lambda_{\alpha }    \mid  X^{(1)} \right ) \P \left ( \lambda_1 \big ( Z^{(1)}_2 \big ) \leq (1+\delta) \e \lambda_{\alpha }   \mid  X^{(1)} \right )     \1_{\mathcal{K}_1} \Big].
\end{align}

We now estimate the two conditional probabilities above.

\textbf{Step 2. Contribution from $ Z_2^{(1)} $.}
 Note that, by definition, under the event $\cK_1,$ all components of $ X_2^{(1)} $ as well as of  $ X^{(1)} $ satisfy {the properties described in the events $\mathcal{C}_{\e,(1+\delta)^2 -1} $ and $ \mathcal{E}_{(1+\delta)^2 -1} $}. Hence by Proposition \ref{prop:eigenvalue_components_Z^1_2} together with a union bound (the number of connected components is bounded by $n$),   for large  enough $n$, under the event $\mathcal{C}_{\e,(1+\delta)^2 -1} \cap \mathcal{E}_{(1+\delta)^2 -1} $, and hence under the event $\mathcal{K}_1$,
\begin{align} \label{574}
\P \left  (\lambda_1 \big ( Z^{(1)}_2 \big ) \leq  (1 + \delta) \e \lambda_{\alpha }   \mid   X^{(1)} \right  ) \geq \frac{1}{2}.
\end{align}

\textbf{Step 3. Contribution from $ Z_1^{(1)} $.}
We proceed by considering groups of stars of similar degrees. Let $\mathcal{S}$ be the collection of stars in its underlying graph $X^{(1)}_1$ given the latter, by construction, is a vertex-disjoint union of stars. We define the events capturing the contributions of {the small and large stars}, by 
\begin{align*}
\cJ_0 & := \Bigg \{  \max_{\substack { S \in \mathcal{S}   \\  d(S) \leq g(\kappa) } } \{ \lambda_1(S) \} \leq (1-\delta') \lambda_{\alpha }  \Bigg \}  ,  \\
\mathcal{J}_m &:=  \Bigg \{ \max_{\substack { S \in \mathcal{S} \\ d(S) \in ( g(  m \kappa),g( (m+2) \kappa) ] }} \{ \lambda_1 \left (S \right ) \} \leq (1 - \delta')  \lambda_{\alpha }   \Bigg \},
\end{align*}
and the events capturing the contribution of the stars with intermediate degree, i.e. for $i = 1,\cdots,m-1$ (recall that $m$ is an integer such that $m\kappa  < 1 \leq (m+1)\kappa  $), by
\begin{align*}
 \mathcal{J}_i &:=   \Bigg \{ \max_{\substack { S \in \mathcal{S} \\ d(S) \in (g( i \kappa),  g((i+1) \kappa )] }} \{ \lambda_1 \left (S \right ) \} \leq (1 - \delta')  \lambda_{\alpha }   \Bigg \} .
\end{align*}  
  Since $ (m+2)\kappa >1+\kappa$, under the event $\cR_\kappa$ and thus under $\mathcal{K}_1$, there are no stars in  $X^{(1)}_1$ of degree at least $g((m+2)\kappa)$. Thus, conditioned on $\mathcal{K}_1$, the event $\bigcap_{i=0}^m \mathcal{J}_i$ implies $\lambda_1 \big ( Z^{(1)}_1 \big ) \leq  (1-\delta') \lambda_{\alpha } $.

We will now lower bound the probabilities of the events $\cJ_i$, conditioned on $\mathcal{K}_1$.
{{By   Lemma \ref{lower lower} with  $\rho = \delta',\gamma=\kappa$ and $ h=0$ (the number of stars  in $X_1^{(1)}$  is bounded by $n$)},  under the event $\mathcal{K}_1$,}
\begin{align*}
 \P  \left (  \mathcal{J}_0  \big | X^{(1)} \right )& \geq  \exp \left ( - n^{f_{\alpha,-\delta'}(\kappa) + \kappa  +   \e \kappa  + o(1)} \right ).
\end{align*}
{We will lower bound the probability of   the remaining events.} Recall that under the event $\cP_{\kappa}$, and thus under the event $\mathcal{K}_1$, for $\gamma =  i\kappa$ with $i=1,\cdots,m$, the number of stars in $X_1^{(1)}$ of degree at least $g(\gamma)$ is bounded by $n^{1-\gamma+\kappa}$. Thus,
by Lemma \ref{lower lower} with $\rho = \delta', \gamma = i\kappa$ and $h = \kappa$ or $2\kappa$, under the event $\mathcal{K}_1$,
\begin{align*}
  \P  \left (  \mathcal{J}_i \big  | X^{(1)} \right ) & \geq  \exp \left (-n^{ f_{\alpha,-\delta'}( (i+1) \kappa) + 2\kappa +   \e ( (i+1) \kappa ) + o(1)} \right ) \quad \text{ and } \\
  \P  \left (  \mathcal{J}_m \big | X^{(1)} \right ) & \geq  \exp \left (- n^{ f_{\alpha,-\delta'}(  (m+2) \kappa  ) + 3\kappa +   \e ( (m+2) \kappa ) + o(1)} \right ).
\end{align*} 
 Now note that  the events $\cJ_i$ are conditionally independent given  $X^{(1)}$. Since    $ m  \kappa<1$  and $f_{\alpha,-\delta'} (\gamma) \leq 1 - (1-\delta')^2$ for any $\gamma>0$ by Lemma  \ref{star_probability_max}, all exponents of $n$ in the above lower bounds  for $ \P  \left (  \mathcal{J}_i \big  | X^{(1)} \right ) $ with  $i=0,\cdots,m$  are less than  
 \begin{align*}
 1-(1-\delta')^2 + 3\kappa +  \e(1+2\kappa) +o(1).
 \end{align*}
Thus,
whenever the event $\mathcal{K}_1$ holds,
\begin{align} \label{575}
& \P \left  ( \lambda_1 \big ( Z^{(1)}_1 \big )  \leq (1-\delta') \lambda_{\alpha }  \big |  X^{(1)} \right ) 
\geq 
\P \left  (  \bigcap_{i=0}^m \mathcal{J}_i \big | X^{(1)} \right ) =  \prod _{i = 0}^{m}  \P \left  (   \mathcal{J}_i \big | X^{(1)} \right ) \nonumber \\
& \geq \left ( \exp \left (-n^{1-(1-\delta')^2 + 3\kappa +  \e(1+2\kappa) +o(1) } \right ) \right )^{ m + 2 } \geq \exp\Big(- \Big ( \frac{1}{\kappa} + 2 \Big ) n^{1-(1-\delta')^2 + 3\kappa +  \e(1+2\kappa) +o(1) }\Big),
\end{align}
where  we used $m \leq \frac{1}{\kappa} $ in the last inequality.

Therefore,  applying \eqref{574} and \eqref{575} to \eqref{30},
\begin{align*}
\P \left ( \lambda_1 ( Z ) \leq (1-\delta) \lambda_{\alpha }  \right )  \geq \frac{1}{2}  \exp\Big(- \Big ( \frac{1}{\kappa} + 2 \Big ) n^{1-(1-\delta')^2 + 3\kappa +  \e(1+2 \kappa) +o(1) }\Big)  \mathbb{P}(\mathcal{K}_1).
\end{align*}
Since $\P(\mathcal{K}_1) \geq \frac{1}{2}$ for large enough $n$ (see \eqref{k1}) and $\lim_{\e \rightarrow 0} \delta' = \delta$ (see \eqref{delta'''} for the definition of $\delta'$),  by taking  $\kappa,\e>0$ small enough, we establish the desired bound.
\end{proof}

We now move on to the final part of our analysis of light-tailed weights.

\subsubsection{Upper bound for the lower tail}
 
We show it is unlikely that all stars induced by vertices of degree close to $g(\gamma_\delta') = \lceil \gamma_\delta' \frac{\log n}{\log \log n} \rceil$, where $\gamma_\delta' = (1-\delta)^2 \left ( 1 -\frac{2}{\alpha} \right )$ was defined in \eqref{gammaprime}, have a largest eigenvalue  less than $(1-\delta)\lambda_\alpha$.
 
To prove this, we condition, for small enough $\rho>0$, on the event $\mathcal{A}_{\gamma_\delta',\rho}$ defined in Proposition \ref{sequential_revealing}, i.e. there exist $m: = \left \lceil  \frac{1}{4} n^{1-\gamma_\delta'  - \rho} \right \rceil $ vertices  having $g(\gamma_\delta') $ disjoint neighbors with no edges between each neighbors.   Let $ S_1,\cdots,S_m$  be  the  vertex-disjoint stars induced by these vertices and   their  $g(\gamma_\delta') $ neighbors.
By Proposition \ref{sequential_revealing},
\begin{align} \label{580}
\P \left ( \mathcal{A}_{\gamma_{\delta}',\rho}^c \right ) \leq e^{ - n^{1-\gamma_\delta'  - \rho + o(1)}}.
\end{align}
Since 
\begin{align*}
\lambda_1^2(Z) \geq \max_{k=1,\cdots,m} \lambda_1^2(S_k) =  \max_{k=1,\cdots,m} \sum_{(i,j) \in E(S_k)} Z_{ij}^2,
\end{align*} 
we have
\begin{align} \label{582}
\P \big ( \lambda_{1} (Z)  \leq (1 - \delta) \lambda_{\alpha } \big )  
\leq   
 \E\bigg[ \P \Big (\max_{k=1,\cdots,m}    \sum_{(i,j) \in E(S_k)} Z_{ij}^2 \leq (1 - \delta) ^2 \lambda_{\alpha } ^2  \mid  X \Big )\1_{ \mathcal{A}_{\gamma_{\delta}',\rho} } \bigg] +\P  (\mathcal{A}_{\gamma_{\delta}',\rho}^c  ) .
\end{align}
Using the   tail estimate \eqref{common_tail_terms:upper} with $d = 1-\delta$ and $b=\gamma_\delta'$, under the event  $\mathcal{A}_{\gamma_{\delta}',\rho}$,
\begin{align} \label{581}
 \P &\Big (\max_{k=1,\cdots,m}    \sum_{(i,j) \in E(S_k)} Z_{ij}^2 \leq (1 - \delta) ^2 \lambda_{\alpha }  ^2 \mid  X \Big ) \nonumber  \\
& \leq 
\Big( 1 - n^{ -(1-\delta)^\alpha \frac{2}{\alpha-2} \left (1 - \frac{2}{\alpha}  \right )^\frac{\alpha}{2}  (\gamma_{\delta}' ) ^{1-\frac{\alpha}{2}} +o(1) }  \Big )^ m   \nonumber \\
&
\leq  \exp(-n^{ 1 - \gamma_\delta' - \rho -(1-\delta)^\alpha \frac{2}{\alpha-2} \left (1 - \frac{2}{\alpha}  \right )^\frac{\alpha}{2}  (\gamma_{\delta}' ) ^{1-\frac{\alpha}{2}} +o(1)}) \leq 
  \exp( -n^{1 - (1-\delta)^2  - \rho +o(1)}),
\end{align}
where we used $\gamma_\delta' = (1-\delta)^2 \left ( 1 -\frac{2}{\alpha} \right )$  to simplify the exponent.
Since $\gamma_{\delta}' = (1-\delta)^2\left ( 1 - \frac{2}{\alpha} \right ) < (1-\delta)^2$,  \eqref{580} and  \eqref{581} show that the dominant term in \eqref{582} is $e^{-n^{1 - (1-\delta)^2  - \rho +o(1)}}$. By taking $\rho>0$  sufficiently small enough, we obtain the matching upper bound.

\qed

\section{Heavy-tailed weights} \label{sec:heavy}
{In this section, we prove Theorems \ref{thm:heavy_upper} and \ref{thm:heavy_lower} in Sections \ref{6.1} and \ref{6.2} respectively. As before, for notational brevity, we define  $ \lambda_\alpha : = \lambda^{\textup{heavy}}_\alpha = ( \log n )^\frac{1}{\alpha}.$}

\subsection{The upper tail} \label{6.1}
Let us first recall the theorem that we will prove in this section.
Recall that  for $\theta>1$ and  the integer $k\geq 2$, we defined the following function
\begin{align*}
\phi_{\theta} (k) = \sup_ {f = (f_1,\cdots,f_k): \norm{f}_1=1} \sum_{i,j\in [k], i \neq j} |f_i| ^{\theta}  |f_j|^{\theta}.
\end{align*}
\heavyupper*

\subsubsection{Lower bound for the upper tail}
  
As mentioned in the idea of proof section, we lower bound the large deviation probability by having high edge-weights on a suitable size of clique.
In the case when $\alpha < 1$, it turns out to suffice to only consider a clique of size 2 (i.e. an  edge). When $1 < \alpha < 2$ on the other hand, we also need to consider the possibility of larger cliques {appearing in $X$}. The lower bound then follows by optimizing over the clique size.

We first note that using 
 $\phi_{\beta/2}(2) = 2^{1-\beta}$ (see  \eqref{clique 2} in Lemma \ref{lemma 02}), 
 \begin{align} \label{formula}
\psi_{\alpha,\delta}(2) = -1 +  \frac{1}{2}  (1+\delta)^\alpha \phi_{\beta/2}(2)^{1 -\alpha}  = -1 +  \frac{1}{2}  (1+\delta)^\alpha   \Big(\frac{1}{2^{\beta-1}}\Big)^{1-\alpha} = (1+\delta)^\alpha - 1 ,
\end{align}
where we used the conjugacy relation  $\frac{1}{\alpha}+\frac{1}{\beta}=1$ in the last identity.

We establish the lower bound by separately proving
\begin{align} \label{147}
\limsup_{n\to \infty}  -  \frac{  \log \P( \lambda_1 (Z) \geq (1+\delta) \lambda_{\alpha } )  }{\log n}
 \leq    \psi_{\alpha,\delta} (2) =   (1+\delta)^\alpha - 1
\end{align}
and for any $k\geq 3$,
\begin{align} \label{148}
\limsup_{n\to \infty}  -  \frac{  \log \P( \lambda_1 (Z) \geq (1+\delta) \lambda_{\alpha } )  }{\log n}
 \leq     \psi_{\alpha,\delta} (k).
\end{align}

\textbf{Single large edge-weight.}
We consider the scenario that there is a large edge-weight, which provides the bound \eqref{147}.
Let us define the event that the number edges in the random graph $X$ is not unusually small:
\begin{align} \label{event f}
\mathcal{M}:=\Big \{|E(X)| \geq \frac{d(n-1)}{4}\Big\}.
\end{align}
Then,
\begin{align*}  
\P  \big  ( \lambda_{1}(Z) \geq (1+\delta) \lambda_{\alpha } \big ) &  \geq  \E \Big[ \P   \big ( \lambda_{1}(Z) \geq (1+\delta ) \lambda_{\alpha }  \mid  X  \big )  \1_{\mathcal{M}} \Big] 
\end{align*}

Since the number of missing edges $\binom{n}{2} - E(X)$ has a  distribution $\Binom\left (\binom{n}{2}, 1 - \frac{d}{n} \right )$, by Lemmas \ref{rel_ent_bern} and \ref{binomial_tails} about the relative entropy  and binomial tail estimates respectively, there is a constant $c>0$ such that 
{\begin{align}  \label{prob f}
\mathbb{P} \big (\mathcal{M}^c \big ) = \P \left ( \binom{n}{2} - |E(X)|  >  \binom{n}{2} - \frac{ d(n-1) }{4} \right )  \leq e^{-\binom{n}{2} I_{ 1 - \frac{d}{n} } \left ( 1- \frac{d}{2n} \right )}\leq  e^{-cn}.
\end{align}}
Also, under the event $\cM$,  by the independence of edge-weights,
\begin{align*}  
\P   \left ( \max_{(i,j)\in E(X)}|Z_{ij}|  \geq (1+\delta ) \lambda_{\alpha }  \mid  X  \right )\geq  1  -  \left ( 1 - C_1 n^{-(1+\delta)^\alpha }  \right )^{ \frac{d(n-1)}{4} } \geq n^{1-(1+\delta)^\alpha + o(1)}.
\end{align*}
Combining these two bounds, we obtain  \eqref{147}.
This already concludes the proof in the case $0< \alpha \leq 1$.

\textbf{Large edge-weights on a bigger clique.}
{Next, we establish the lower bound \eqref{148}  in the case $ 1 < \alpha < 2$}.  {For any  $k\geq 3$}, using the result and notation from Lemma \ref{equality_L_p}, take $k_1,k_2\geq 0$ with $k_1+k_2\leq k$, $x,y\geq 0$ and the $k\times k$ matrix
  $A = (a_{ij})_{i,j \in [k]}$ given by
\begin{align} \label{A}
a_{ij}=
\begin{cases}
x^2 \qquad &  i\neq j, i,j\in \{1,\cdots,k_1\}=: V_1 ,\\
y^2  \qquad  &  i\neq j,  i,j\in \{k_1+1,\cdots,k_1+k_2\} =:V_2 ,\\
xy \qquad &  i\in V_1, j\in V_2 \text{ or } i\in V_2,j\in V_1,\\
0 \qquad &\text{otherwise},
\end{cases}
\end{align}
which achieves the  equality in \eqref{case1}, i.e.
\begin{align} \label{222}
 \lambda_1(A) = \phi_{ \frac{\alpha}{2(\alpha-1)}} (k)^{\frac{\alpha-1}{\alpha}}  \norm{A}_{\alpha}= \phi_{\frac{\beta}{2}} (k)^{\frac{\alpha-1}{\alpha}}  \norm{A}_{\alpha}.
\end{align}

Conditioned on the event that $X$ contains a clique of size $k$ denoted  by $H$, let  $V(H) := \{v_1,\cdots,v_{k}\}$. Now consider the event that
\begin{equation} \label{y}
Y_{v_i v_j} \geq  
\begin{cases}
  \frac{1}{\lambda_1(A)} (1+\delta) {\lambda_{\alpha }} x^2 \quad  & i\neq j, i,j  \in V_1, \\ 
 \frac{1}{\lambda_1(A)} (1+\delta) \lambda_{\alpha } y^2 \quad  &  i\neq j, i,j  \in V_2 ,\\
 \frac{1}{\lambda_1(A)}(1+\delta) \lambda_{\alpha } xy \quad & i\in V_1, j\in V_2 \text{ or } i\in V_2,j\in V_1 , \\
 0 \quad & \text{otherwise}.
\end{cases}
\end{equation}
By the distribution of the edge-weights as defined in \eqref{weibull},  {the conditional probability} of this event is lower bounded by
\begin{align} \label{145}
 n^{-(1+\delta)^\alpha \frac{1}{\lambda_1(A)^\alpha} \left (\binom{k_1}{2} x^{2\alpha} + \binom{k_2}{2} y^{2\alpha} + k_1 k_2 x^\alpha y^\alpha \right) + o(1)} & =    n^{  -   (1+\delta)^\alpha  \frac{ ||A||_{\alpha}^\alpha}{2\lambda_1(A)^\alpha}     +    o(1)} \nonumber \\
 & \overset{\eqref{222}}{=} n^{-\frac{1}{2}  (1+\delta)^\alpha \phi_{\beta/2}(k)^{1 -\alpha} + o(1)}.
\end{align}
{Also, under this event, we have
$\lambda_1 (Z) \geq  \lambda_1(Z|_{H}) \geq  (1+\delta) \lambda_{\alpha }$,  {since  $Z|_H$ is entrywise greater than  or equal to the matrix $\frac{(1+\delta) \lambda_\alpha}{\lambda_1(A)}  A$, having non-negative entries, whose largest eigenvalue is  $ (1+\delta) \lambda_{\alpha }$.}
}
By Lemma \ref{lemma clique} the probability that $X$ contains a  clique of size $k\geq 3$ is lower bounded by $C n^{- \binom{k}{2} +k }$. Therefore, combining this  with \eqref{145},   for any $k\geq 3$,
\begin{align} \label{146} 
\limsup_{n\to \infty}  -  \frac{  \log \P( \lambda_1 (Z)  \geq (1+\delta) \lambda_{\alpha } )  }{\log n}
 \leq   \binom{k}{2} - k +  \frac{1}{2}  (1+\delta)^\alpha \phi_{\beta/2}(k) ^{1 -\alpha} =\psi_{\alpha,\delta}(k).
\end{align}  
 
 \qed

\subsubsection{Upper bound for the upper tail}
{As in the light-tailed   case, we decompose $Z = Z^{(1)} + Z^{(2)}$ with a negligible part  $Z^{(2)}$. However, the analysis of $Z^{(1)}$ will be significantly different since the governing structures will be distinct.  
 
We first present  a counterpart of Lemma \ref{lem:bulk_negligible}.  The proof is  almost identical, so we omit it.}

\begin{lemma} \label{heavy_lem:bulk_negligible}
For $\delta>0$,
\begin{equation}
\liminf_{n \to \infty} \frac{-\log \P (\lambda_1  (Z^{(2)})  \geq \e^{\frac{1}{\alpha}} (1+\delta) \lambda_{\alpha } )  }{\log n} \geq (1+\delta)^2 -1 .
\end{equation}
\end{lemma}
 
{The results in Section \ref{connectivity}  provide the structural properties of $Z^{(1)}$. The following key proposition, a counterpart of Proposition \ref{prop:eigenvalue_components_Z^1_2},  states a bound on the largest eigenvalue of  such  networks. Recall that $ \lambda_{\alpha }   =  (  \log n)^{\frac{1}{\alpha}} $.}

{A key distinction between Proposition  \ref{prop:eigenvalue_components_Z^1_2} and  this proposition     is that   the former claims the smallness of $\lambda_1(A)$  under the condition  that the  maximum degree is $o(\frac{\log n}{\log \log n})$ and edge-weights are light. Whereas the latter provides the bound on    $\lambda_1(A)$  for heavy-tailed weights  under the weaker condition on the maximum degree, $O(\frac{\log n}{\log \log n})$.}

{
\begin{proposition} \label{keyprop}
For any $k,n \in \mathbb{N}$   and constants $\e, c_1,c_2,c_3 >0$, let  $G=(V,E,A)$ be a   connected network ($A = (a_{ij})_{i,j\in V}$ is a conductance matrix) whose  maximum clique size is $k$ and which satisfies
\begin{enumerate}
\item $d_1(G) \leq c_1   \frac{\log n}{\log \log n} $,
\item  $|V| \leq   \frac{c_2}{\e}  \frac{\log n}{\log \log n} $,
\item $|E| \leq |V|+c_3$.
\end{enumerate}
Suppose that the edge-weights   are given by i.i.d. Weibull distributions  with a shape parameter $0<\alpha<2$ conditioned to be  greater than $ (\e \log \log n)^{\frac{1}{\alpha}} $ in absolute value. Let $\rho>0$ and $0<\xi<\frac{1}{2}$ be constants, then:
\begin{enumerate}[1.]
\item In the case $1<\alpha<2$,  let $\beta$ be the conjugate of $\alpha$ and set  $\tau :=  (c_3 +2 )^{ \frac{1}{2\beta}  }\xi^{\frac{1}{2}}$. Then,
\begin{align}\label{keydisp}
\mathbb{P} \left ( \lambda_1(A) \geq  \rho^{\frac{1}{\alpha}} \lambda_{\alpha }  \right )  \leq      n^{- 2^{-\alpha}  \tau^{-\alpha  }\rho+  c_2    +  o(1)} +n^{-  2^{-1} \phi_{\beta/2} (k)^{1-\alpha }(1-\tau )^\alpha    \rho + c_1 \xi^{-2}   \varepsilon    + o(1)}   ,
\end{align}
where  $\phi_{\beta/2}$ is defined in \eqref{varphi}.

\item  In the case $0<\alpha\leq 1$, 
\begin{align} \label{keydisp2}
\mathbb{P} \left ( \lambda _1(A) \geq  \rho^{\frac{1}{\alpha}} \lambda_{\alpha } \right )  \leq       n^{- 2^{-\alpha}  \xi^{-\alpha  }\rho   +  c_2    + o(1)}  +   n^{-\rho (1- \xi)^\alpha  + c_1  \xi^{-2}  \e   +  o(1)} . 
\end{align}
\end{enumerate}

\end{proposition}
}

The expression on the right hand side is technical  but the constants $\varepsilon,\xi$ will be suitably chosen  so that  the dominant behaviors  in the case  $1<\alpha \leq 2$ and $0<\alpha \leq 1$ are $n^{- 2^{-1}  \phi_{\beta/2} (k)^{1-\alpha }    \rho }$ and $n^{-\rho}$ respectively.
Note that in the case $1<\alpha<2$, since $\phi_{\beta/2}(k)$ is non-decreasing in $k$ (see \eqref{uniform} in  Lemma \ref{lemma 02}), the upper bound \eqref{keydisp} gets worse as $k$ increases.  Whereas, the upper bound \eqref{keydisp2}  in the case $0<\alpha\leq 1$ does not depend on $k$.

\begin{proof}[Proof of Proposition \ref{keyprop}]
 
{
Let $f$ be a (random) top eigenvector.  For  a constant $ \xi \in (0,\frac{1}{2})$, define the subsets $ V_S, V_L, E_S, E_L$   as  \eqref{400} and \eqref{401}  in Proposition \ref{prop:eigenvalue_components_Z^1_2}.  Since $|V_L| \leq \frac{1}{\xi^2}$, by the condition (1),  
\begin{align} \label{362}
|E_L| \leq  \frac{c_1}{\xi^2}  \frac{\log n}{\log \log n}.
\end{align}
{
As in \eqref{413}, we write  $\lambda_1(A) = 2  \lambda_S +2  \lambda_L$. Then, for any $0<\tau<1$,
\begin{align} \label{363}
\mathbb{P}& \left ( \lambda_1(A) \geq  \rho^{\frac{1}{\alpha}} \lambda_{\alpha }  \right )   \leq   \mathbb{P} \left ( 2\lambda_S\geq  \tau \rho^{\frac{1}{\alpha}} \lambda_{\alpha }  \right )  + \mathbb{P} \left ( 2 \lambda_L\geq  (1-  \tau ) \rho^{\frac{1}{\alpha}} \lambda_{\alpha } \right ) .
\end{align}
How we proceed from here depends on the value of $\alpha$, but in both cases the parameter $\tau$ will be chosen in such way that it is costly for $\lambda_S$ to be large.} {In the case  $1<\alpha<2$ we apply  H\"older's inequality to bound $\lambda_S$ and $\lambda_L$, while in the case $0<\alpha\leq 1$ we use the monotonicity of $\ell^p$ norms.}

\textbf{Case 1. $\bm{1 < \alpha <  2}$.}
We first estimate $\lambda_S$. As in \eqref{416},
\begin{align*} 
 \sum_{(i,j)\in E_S } |f_i|^\beta |f_j|^\beta   \leq  \xi^{\beta}+  (c_3+1) \xi^{2\beta}\leq   (c_3+2) \xi^\beta,
\end{align*}
where the first term $\xi^\beta$ is obtained as an application of \eqref{lemma 36} with $\theta = \frac{\beta}{2} >1$.

Note that by assumptions, $|E| \leq |V| + c_3 \leq \frac{c_2}{\e} \frac{\log n}{\log \log n} + c_3$. Hence,
setting $\tau :=   (c_3 +2  )^{ \frac{ 1}{2\beta} }\xi^{\frac{1}{2}} $, by  H\"older's inequality,  }
$$
\lambda_S  \leq \bigg (\sum_{(i,j)\in E_S}  |f_i|^\beta |f_j|^\beta \bigg )^{\frac{1}{\beta} }  \bigg(\sum _{(i,j)\in E_S}  |a_{ij}|^\alpha \bigg)^{\frac{1}{\alpha} }     \leq  \tau^2  \bigg(\sum _{(i,j)\in E_S}  |a_{ij}|^\alpha \bigg)^{\frac{1}{\alpha}}   \leq  \tau^2  \bigg(\sum _{(i,j)\in E}  |a_{ij}|^\alpha \bigg)^{ \frac{1}{\alpha} }. $$
By a tail estimate for the sum of Weibull random variables (the bound \eqref{212} with $a= \frac{\rho}{2^\alpha \tau^\alpha} , b=\frac{c_2}{\e}, c=c_3$),
\begin{align} \label{365}
\mathbb{P}\left ( 2 \lambda_S \geq \tau     \rho^{\frac{1}{\alpha}} \lambda_{\alpha }   \right ) 
&\leq 
\mathbb{P}\bigg ( \sum _{(i,j)\in E} |a_{ij}|^\alpha \geq  \frac{\rho}{2^\alpha \tau^\alpha  }\log n  \bigg )  \leq  n^{-2^{-\alpha} \tau^{-\alpha  } \rho +  c_2    +  o(1)} .
\end{align}
 
Next, we estimate $\lambda_L$. By the definition of $\varphi_{\beta/2}$ in  \eqref{varphi} (where sup is taken over $\norm{f}_1=1$ whereas our vector $f$ satisfies $\norm{f}_2=1$) and since $\varphi_{\beta/2}=\phi_{\beta/2}$ (see Lemma \ref{same}),
\begin{align*}
\sum_{(i,j)\in E_L}  |f_i|^\beta |f_j|^\beta   \leq  \sum_{(i,j)\in E}  |f_i|^\beta |f_j|^\beta    \leq \frac{1}{2}\varphi_{\beta/2} (k)= \frac{1}{2 }\phi_{\beta/2} (k).
\end{align*}
Hence, by  H\"older's inequality,  
\begin{align} \label{366}
\lambda_L \leq \bigg ( \sum_{(i,j)\in E_L}  |f_i|^\beta |f_j|^\beta \Big)^{\frac{1}{\beta}}   \bigg (\sum _{(i,j)\in E_L}  |a_{ij}|^\alpha \bigg)^{\frac{1}{\alpha}}  \leq \bigg (\frac{1}{2} \phi_{\beta/2}(k)\bigg )^{\frac{1}{\beta}}   \bigg (\sum _{(i,j)\in E_L}  |a_{ij}|^\alpha \bigg )^{\frac{1}{\alpha}}. 
\end{align} 
Since the  number of possible subsets that  a random subset $E_L$ can take is bounded by $  |V|^{\left \lfloor {  \frac{1}{\xi^2} }   \right  \rfloor }  = n^{o(1)}$, by a union bound and the tail estimate \eqref{212} with
\begin{align*}
a= 2^{\frac{\alpha}{\beta}} \phi_{\beta/2} (k)^{-\frac{\alpha}{\beta}}   (1 -  \tau )^\alpha   \frac{\rho}{2^\alpha  } ,\quad b= \frac{c_1}{\xi^2},\quad c=0
\end{align*}
(recall the bound of $|E_L|$ in  \eqref{362}), we have
\begin{align} \label{367}
\mathbb{P} \left ( 2 \lambda_L \geq    (1-  \tau ) \rho^{\frac{1}{\alpha}} \lambda_{\alpha }  \right )  &  \leq 
\mathbb{P} \bigg (  \sum _{ (i,j)\in E_L}  |a_{ij}|^\alpha   \geq 2^{\frac{\alpha}{\beta}} \phi_{\beta/2} (k)^{-\frac{\alpha}{\beta}}   (1 -  \tau )^\alpha   \frac{\rho}{2^\alpha  } \log n \bigg ) \nonumber \\
 & \leq n^{-  2^{-1}  \phi_{\beta/2}(k)^{1-\alpha }(1-\tau )^\alpha    \rho + c_1 \xi^{-2}  \varepsilon   + o(1)}   ,
\end{align}
where we used $ 2^{\frac{\alpha}{\beta}} \cdot 2^{-\alpha} = 2^{-1}$ and $- \frac{\alpha}{\beta} = 1-\alpha$ by a conjugacy relation $\frac{1}{\alpha}+\frac{1}{\beta}=1$.

Plugging the bounds \eqref{365} and \eqref{367} into \eqref{363} gives the desired bound.

\textbf{Case 2. $\bm{0<\alpha \leq 1}$.} {In this case we use \eqref{363} with $\tau = \xi$.}
We first estimate $\lambda_S$. 
Since $|f_if_j| \leq  \xi^2   $ when $|f_i|,|f_j| <\xi$, by the monotonicity of $\ell^p$ norms (note that $0<\alpha\leq 1$),
\begin{align*}
  \sum_{(i,j)\in E_S}  a_{ij} f_if_j   \leq      \xi ^2   \sum_{(i,j)\in E_S}  |a_{ij}| \leq      \xi^2    \sum_{(i,j)\in E}  |a_{ij}|\leq    \xi ^2  \Big ( \sum_{(i,j)\in E}  |a_{ij}|^\alpha \Big)^{\frac{1}{\alpha}}.
\end{align*}
Thus,   recalling $|E|   \leq \frac{c_2}{\e} \frac{\log n}{\log \log n} + c_3$, by a tail estimate  \eqref{212} with $a= \frac{\rho}{(2\xi )^\alpha}, b=\frac{c_2}{\e},c=c_3$,
\begin{align} \label{371}
\mathbb{P}(2  \lambda_S \geq  \xi  \rho^{\frac{1}{\alpha}} \lambda_{\alpha }   )&\leq \mathbb{P} \bigg ( \sum _{i<j, (i,j)\in E} |a_{ij}|^\alpha \geq \frac{\rho}{(2\xi )^\alpha  } \log n    \bigg ) \leq  n^{- 2^{-\alpha}  \xi^{-\alpha  }\rho + c_2  + o(1)}.
\end{align}

Next, we estimate $\lambda_L$. Since $|f_if_j| \leq  \frac{1}{2}   $  when $\norm{f}_2=1$,  again by the  monotonicity of $\ell^p$ norms,
\begin{align*}
 \sum_{(i,j)\in E_L}  a_{ij} f_if_j   \leq   \frac{1}{2}  \sum_{(i,j)\in E_L}  |a_{ij}  |  \leq  \frac{1}{2} \bigg  (    \sum_{(i,j)\in E_L}  |a_{ij}  |^\alpha  \bigg )^{\frac{1}{\alpha}}.
\end{align*}
Thus, as in \eqref{367},
\begin{align} \label{372}
\mathbb{P} \left (2  \lambda_L \geq   (1  - \xi   )   \rho^{\frac{1}{\alpha}} \lambda_{\alpha } \right   )&\leq \mathbb{P}\bigg ( \sum _{(i,j)\in E_L} |a_{ij}|^\alpha \geq  (1- \xi )^\alpha \rho \log n    \bigg )  \leq  n^{-\rho (1- \xi )^\alpha  + c_1 \xi^{-2} \e    +  o(1)} . 
\end{align}
Therefore, by plugging the bounds \eqref{371} and \eqref{372} into \eqref{363}, the proof of \eqref{keydisp2} is concluded.

\end{proof}

With all this preparation, we are now ready to prove the upper bound for the upper tail.

\begin{proof}[Proof of the upper bound of the upper tail]
{By a decomposition   $Z = Z^{(1)} + Z^{(2)}$, 
setting $\delta'$ as}
\begin{align}\label{del}
1+\delta = (1+\delta')  +\e^{\frac{1}{\alpha}} (1+\delta),
\end{align}
we have
\begin{align} \label{851}
 \P \big ( \lambda_1(Z) \geq  (1+\delta) \lambda_{\alpha } \big )  
 \leq
   \P \left ( \lambda_1 \big   (Z^{(1)} \big )  \geq  (1+\delta')  \lambda_{\alpha } \right ) 
 +
  \P  \left ( \lambda_1 \big ( Z^{(2)} \big ) \geq \e^{\frac{1}{\alpha}}(1+\delta) \lambda_{\alpha } \right ).
\end{align}
By Lemma \ref{heavy_lem:bulk_negligible},
\begin{align}
  \P \left (\lambda_1 \big (Z^{(2)} \big )  \geq \e^{\frac{1}{\alpha}} (1+\delta) \lambda_{\alpha } \right )  \leq n^{1 - (1+\delta)^2 + o(1)},
\end{align} which implies that  $Z^{(2)}$ is spectrally negligible. Thus, it suffices to focus on  the spectral behavior of  $Z^{(1)}$.
Let $C_1,\cdots,C_m$ be the connected components of $X^{(1)}$ and
 let $\mathcal{H}$ be the event defined by
\begin{align} \label{tree}
\mathcal{H}:= \big  \{|\{\ell = 1,\cdots,m: C_\ell  \ \textup{not tree} \}|  <  \log n \big \}.
\end{align}
Then, by Lemma \ref{lem:excess_edges} and the Van-den Berg-Kesten  (BK) inequality \cite{bk},  for large enough $n$,
 {
$$
\mathbb{P}( \mathcal{H} ^c) \leq  C^{\log n} (\log n)^{- 2\varepsilon  \log n} \leq    (\log n)^{- \varepsilon  \log n} 
$$}
 
Setting $\delta_0 := (1+\delta)^\alpha-1$, 
{define  the  event $\cF_0$ measurable with respect to $X^{(1)}$}  which guarantees that all components $C_1,\cdots,C_m$ satisfy the conditions of Proposition \ref{keyprop} and that there are only few components which are not trees:
\begin{align*}
\cF_0:=\cD_{2\delta_0}   \cap \cC_{\e,2\delta_0} \cap \cE_{2\delta_0} \cap \mathcal{H},
\end{align*}
where $\cD_{2\delta_0}  ,\cC_{\e,2\delta_0} $ and $ \cE_{2\delta_0}$ are the events defined in Lemmas \ref{lemma 33}, \ref{lem:biggest_component} and \ref{lem:excess_edges} respectively  (applied to the graph  $X^{(1)}$).
 {By the discussion above as well as the results in these lemmas,  for large enough $n$,
\begin{align} \label{129}
\mathbb{P}(\cF_0^ c ) \leq n^{-2 \delta_0+o(1)}.
\end{align}}

 {Conditioned on $X^{(1)}$,} let $Z^{(1)}_\ell$ be the network $Z^{(1)}$ restricted to $C_\ell$ and denote by $k_\ell$ the size of the maximal clique in $C_\ell$.
We consider the case $1<\alpha<2$ and $0<\alpha\leq 1$ separately,  {since the maximum clique size turns out to be only relevant  in the former case.}

\textbf{Case 1: $\bm{1<\alpha<2}$.} Let $\beta>2$ be the conjugate of $\alpha$. Under the event $\cF_0$, in order to control  $\lambda_1  ( Z^{(1)}_\ell )$ for each $\ell=1,\cdots,m$,  we apply Proposition \ref{keyprop}  with 
\begin{align} \label{setting}
c_1=c_2=1 + 2\delta_0,\quad c_3=2\delta_0,\quad  \xi := \e^{\frac{1}{4}} ,\quad \tau = (2\delta_0 + 2)^{\frac{1}{2\beta}} \e^{\frac{1}{8}},\quad  \rho = (1+\delta')^\alpha
\end{align}
(recall that  $\delta_0 = (1+\delta)^\alpha-1$).
Observing that for  small enough $\e>0$,  the first term in the bound  \eqref{keydisp} is negligible compared to the second term,
\begin{align} \label{127}
\mathbb{P} \left (\lambda_1 \big ( Z^{(1)}_\ell \big ) \geq  (1+\delta')  \lambda_{\alpha }   \mid  X^{(1)} \right ) 
\1_{\cF_0}    
\leq 
 n^{-  2^{-1}  \phi_{\beta/2} (k_\ell)^{1-\alpha }(1-\tau )^\alpha     (1+\delta')^\alpha + (1+2\delta_0)   \varepsilon^{1/2}    + o(1)} .
\end{align}

The argument for a component  now depends on the size of its maximal clique, i.e. whether $k_\ell \geq 3$ or $k_\ell=2$.  For this define
\begin{align}
\nonumber
I := \{\ell = 1,\cdots,m: &\,\, k_\ell \geq 3\},\quad J:= \{\ell = 1,\cdots, m: k_\ell=2\},
\end{align} 
and let $ \bar{k} := \max \{k_1,\cdots,k_m\}.$  Since the maximum size of clique in any tree is equal to 2, under the event  $\mathcal{H}$ defined in \eqref{tree},
\begin{align} \label{bound I}
|I| \leq \log n.
\end{align}
Then, by \eqref{127} and using the fact that $\phi_{\beta/2}(k_\ell) \leq \phi_{\beta/2}(\bar{k})$ (recall that $\phi_{\beta/2}$ is non-decreasing, see \eqref{uniform} in Lemma \ref{lemma 02}),  for $\ell\in I$,
\begin{align}  \label{141}
\mathbb{P} \left ( \lambda_1 \big ( Z^{(1)}_\ell \big ) \geq  (1+\delta')  \lambda_{\alpha }    \mid X^{(1)} \right )  \1_{\cF_0} 
&\leq 
  n^{-  2^{-1}  \phi_{\beta/2}(\bar{k})^{1-\alpha }(1-\tau )^\alpha     (1+\delta')^\alpha + (1+2\delta_0)   \varepsilon^{1/2}    + o(1)}    \nonumber \\
   & \leq n^{-  2^{-1}  \phi_{\beta/2}(\bar{k})^{1-\alpha }   (1+\delta)^\alpha    + \zeta_1+ o(1)}  
\end{align}
for some $\zeta_1 = \zeta_1(\e)$ with $\lim_{\e \rightarrow 0} \zeta_1 = 0$. Note that the last inequality follows from the fact that $\lim_{\e \rightarrow 0} \delta' = \delta$, $\lim_{\e \rightarrow 0} \tau = 0$ (recall the definition of $\delta'$ and $\tau$ in \eqref{del} and \eqref{setting} respectively) and the uniform boundedness of the quantity $\phi_{\beta/2}(\bar{k})^{1-\alpha }  $ ({recall the monotonicity property  of $\phi_{\beta/2}$ by  Lemma \ref{lemma 02} \eqref{uniform}  and the fact $\alpha>1$}).

Similarly, for $\ell\in J$, by \eqref{127} again,
\begin{align}  \label{142}
\mathbb{P} \left ( \lambda_1 \big ( Z^{(1)}_\ell \big ) \geq  (1+\delta')  \lambda_{\alpha }    \mid X^{(1)} \right )  \1_{\cF_0} 
\leq  n^{-  2^{-1}  \phi_{\beta/2}(2)^{1-\alpha }   (1+\delta)^\alpha    + \zeta_2+ o(1)}   
\end{align} 
for some $\zeta_2= \zeta_2(\e)$ with $\lim_{\e \rightarrow 0} \zeta_2 = 0$.

Note that since $X^{(1)}$ is distributed as  $ \calG_{n,q} $ with $q \leq p=\frac{d}{n}$, Lemma \ref{lemma clique} implies that  {for  any $k \geq 3$}, the probability that  
$X^{(1)}$  contains a clique of size $k$ is bounded by $  d^{{k \choose 2}}n^{ -{k \choose 2} + k} $.
 Thus, by  \eqref{129}, \eqref{141} and  \eqref{142} combined with a union bound, 
\begin{align}    \label{149}
\mathbb{P}& \left ( \lambda_1 \big ( Z^{(1)} \big ) \geq  (1+\delta')  \lambda_{\alpha }  \right   )  \nonumber\\
  &\leq 
  C   \log n  \cdot   \sum_{k=3}^n   d^{{k \choose 2}} n^{ -{k \choose 2} + k   } \cdot  n^{-  2^{-1}  \phi_{\beta/2}(k)^{1-\alpha }   (1+\delta)^\alpha    + \zeta_1+ o(1)}  \nonumber   \\
  &+ Cn\cdot n^{-  2^{-1}  \phi_{\beta/2}(2)^{1-\alpha }   (1+\delta)^\alpha    + \zeta_2+ o(1)}      +Cn^{-2\delta_0 + o(1) } ,
\end{align}
where the multiplicative factors $\log n $ and $n$ arise from \eqref{bound I} and the fact $|J| \leq n$ respectively.  We analyze each term above.

{Recalling the definition of $\psi_{\alpha,\delta}(k)$ in  \eqref{psi},  the exponent of $n$ in each summation is less than $- \min_{k\geq 3} \psi_{\alpha,\delta}(k) + \zeta_1 + o(1)$.  By a straightforward argument, one can deduce that the first term is bounded by $
 n^{- \min_{k\geq 3} \psi_{\alpha,\delta}(k) + \zeta_1 + o(1) }
$. In addition, by the first identity in  \eqref{formula}, 
the second term  is nothing other than
$ n^{-  \psi_{\alpha,\delta}(2) + \zeta_2 + o(1) }$. Furthermore,  since 
$
 \psi_{\alpha,\delta}(2) \overset{\eqref{formula}}{=}  (1+\delta)^\alpha-1 =\delta_0
$,
 the  last term  is  bounded by $n^{-2  \psi_{\alpha,\delta}(2) + o(1)}$.
}

Therefore, by combining the  above bounds together,
  there exists $\zeta = \zeta(\e) > 0$ with $\lim_{\e \ro} \zeta  = 0$  such that the  RHS  in \eqref{149} is bounded by $n^{-\min_{k\geq 2} \psi_{\alpha,\delta}(k) + \zeta + o(1) }$. 
  By taking $\e>0$ sufficiently small, we conclude the proof.

\textbf{Case 2: $\bm{0 < \alpha \leq 1}$.}  
We apply Proposition \ref{keyprop} as before. The dominating term in  the bound  \eqref{keydisp2} is the second term for  small enough $\e>0$. For each $\ell=1,\cdots,m$,
\begin{align}  
\mathbb{P} \Big (\lambda_1 \big (Z^{(1)}_\ell \big ) \geq  (1+\delta')  \lambda_{\alpha }   \mid  X^{(1)} \Big)  \1_{\cF_0}  \leq  n^{-  ( 1+\delta')^\alpha (1- \e^{1/4})^\alpha  + (1+2\delta_0) \e^{1/2}    +  o(1)} \leq  n^{-(1+\delta)^\alpha + \zeta' + o(1)} 
\end{align}
for some  $\zeta' = \zeta'(\e) >0$ with $\lim_{\e \ro} \zeta'  = 0$.
One can then  conclude the proof by applying a union bound over at most $n$ many connected components $C_1,\cdots,C_m$.
\end{proof}

\begin{remark} \label{rem:str_upper_light}
Although we do not pursue proving this formally,  with some additional work, {one may be able to prove the following structure theorem}: Let 
\begin{align*}
\bar{k} := \argmin_{k \in \mathbb{N}_{\geq 2}} \psi_{\alpha,\delta}(k).
\end{align*} 
Then, conditioned on the upper tail event $\{\lambda_1(Z) \geq (1+\delta) \lambda_\alpha\}$, with high probability, there exists a clique of size close to $\bar{k}$ in $X$ with high edge-weights on it. 

{
Also recall from Section  \ref{sec 2.3} that the conditional structure given atypically large $\lambda_1(Z)$ is different in the Gaussian and heavy-tailed edge-weight cases. In the former case, the optimal size of the clique tends to infinity as the amount of deviation $\delta$ goes to infinity.
whereas, in the latter case, it stays bounded.}
 
\end{remark}

\subsection{The lower tail} \label{6.2}
We start by recalling the theorem we prove in this section.

\heavylower*

\subsubsection{Lower bound of the lower tail}
 
{We start by defining the $X$-measurable event} $\mathcal{B}_1$ as
\begin{align*}
\mathcal{B}_1 := \left  \{ \lambda_1\left (X \right ) \leq 2\frac{(\log n)^\frac{1}{2}}{(\log \log n)^{\frac{1}{2}} } \right \}.
\end{align*}
Recalling the event $\mathcal{H} $ introduced in \eqref{tree}, we define the event $\cF_1$ measurable with respect to $\big \{ X,X^{(1)} \big \}$:
{\begin{align*}
\cF_1:=  \cD_{\delta}  \cap \cC_{\e, \delta} \cap \cE_{\delta} \cap \mathcal{H}  \cap \cB_1.
\end{align*}}
By Theorem \ref{thm:eigenvalues_bbg}, $\lim_{n\to \infty} \mathbb{P}(\mathcal{B}_1) = 1$. Combining this with a previous argument to derive  \eqref{129},  for large enough $n$,
\begin{align} \label{860}
 \P \big ( \cF_1 \big ) \geq \frac{1}{2}.
\end{align} 
Since $\lambda_1 (Z) \leq \lambda_1 (Z^{(1)}) + \lambda_1 (Z^{(2)})$,
setting
\begin{align} \label{delta''}
\delta'' :=  \delta + \varepsilon^{\frac{1}{\alpha}},
\end{align} we have
\begin{align} \label{861}
\P & \big ( \lambda_1(Z) \leq (1-\delta) \lambda_{\alpha }   \big ) \geq 
\E \left  [ \P \left  ( \lambda_1 \big (Z^{(1)} \big ) \leq (1-\delta'') \lambda_{\alpha } , \;  \lambda_1 \big (Z^{(2)} \big ) \leq   \varepsilon^{\frac{1}{\alpha}}  \lambda_{\alpha }   \mid  X, X^{(1)}  \right ) \1_{\cF_1} \right  ].
\end{align}
 Since $|Y_{ij}^{(2)}|\leq  (\e \log \log n)^{\frac{1}{\alpha}} $,   under the event $\mathcal{B}_1$, hence under the event $\mathcal{F}_1$, for large enough $n$,
$$\lambda_1 \left ( Z^{(2)} \right  ) \leq  2 \frac{(\log n)^\frac{1}{2}}{(\log \log n)^{\frac{1}{2}} } \cdot   (\e \log \log n)^{\frac{1}{\alpha}} \leq   \varepsilon^{\frac{1}{\alpha}} \lambda_{\alpha } $$
(recall that $\lambda_\alpha =(\log n)^{\frac{1}{\alpha}}$ and $\alpha<2$).
 Thus, using the  conditional independence of $X$ and $Z^{(1)}$   given $X^{(1)}$, under the event $\cF_1$,
 \begin{align} \label{862}
 \P & \left ( \lambda_1 \big ( Z^{(1)} \big ) \leq (1-\delta'') \lambda_{\alpha } ,   \lambda_1 \big ( Z^{(2)} \big ) \leq  \varepsilon^{\frac{1}{\alpha}}  \lambda_{\alpha }    \mid  X, X^{(1)} \right )  \nonumber\\ 
 &= \P \left ( \lambda_1 \big (Z^{(1)} \big ) \leq (1-\delta'') \lambda_{\alpha }     \mid  X, X^{(1)}  \right ) 
= \P \left  ( \lambda_1 \big ( Z^{(1)} \big ) \leq (1-\delta'') \lambda_{\alpha }   \mid   X^{(1)} \right  ).
\end{align}   
Therefore, applying this to \eqref{861},
\begin{align} \label{11}
\P  \big ( \lambda_1(Z) \leq (1-\delta) \lambda_{\alpha }   \big ) \geq  \mathbb{E} \left [ \P   \left ( \lambda_1 \big ( Z^{(1)} \big ) \leq (1-\delta'') \lambda_{\alpha }   \mid   X^{(1)}  \right ) \1_{\cF_1}\right].
\end{align}
Let $C_1,\cdots,C_m$ be the connected components of $X^{(1)}$ and   $Z^{(1)}_\ell$ be the restriction of $Z^{(1)}$ to $C_\ell$ for $\ell = 1,\cdots,m$.  
Let $k_\ell$ be the size of the maximal clique in $C_\ell$.

{By a similar reasoning as in the upper tail case, we separately consider $1<\alpha<2$ and $0<\alpha\leq 1$.}

\textbf{Case 1: $\bm{1<\alpha<2}$.} 
Let $\beta>2$ be the conjugate of $\alpha$.
Under the event $\cF_1$, in order to control  $\lambda_1  ( Z^{(1)}_\ell )$ for each $\ell=1,\cdots,m$,  we apply Proposition \ref{keyprop}  with 
\begin{align}  \label{setting2}
c_1=c_2=1 + \delta,\quad c_3=\delta,\quad  \xi := \e^{\frac{1}{4}} ,\quad \tau = (\delta+2)^{\frac{1}{2\beta}} \e^{\frac{1}{8}},\quad  \rho = (1-\delta'')^\alpha.
\end{align}
The dominating term in  the bound \eqref{keydisp} is the second term for   small enough $\e>0$. Thus,
\begin{align}  \label{900}
\mathbb{P} \left (\lambda_1 \big ( Z^{(1)}_\ell \big ) \geq  (1 - \delta'')  \lambda_{\alpha }   \mid  X^{(1)} \right ) 
\1_{\cF_1}    
\leq 
 n^{-  2^{-1}  \phi_{\beta/2} (k_\ell)^{1-\alpha }(1-\tau )^\alpha     (1-\delta'')^\alpha + (1+\delta)   \varepsilon^{1/2}    + o(1)} .
\end{align}
As in the proof of upper bound for upper tails,  we proceed differently for the components depending on the size of  maximal cliques:
\begin{align}
\nonumber
I := \{\ell = 1,\cdots,m: &\,\, k_\ell \geq 3\},\quad J:= \{\ell = 1,\cdots, m: k_\ell=2\},
\end{align} 
and let $ \bar{k} := \max \{k_1,\cdots,k_m\}.$    To bound \eqref{900}, we use  the fact that  there is a constant $c=c(\beta)>0$ such that $\phi_{\beta/2}(k) \leq c $ for all $k\geq 2$ (see \eqref{uniform} in  Lemma \ref{lemma 02}). 
Thus, for $\ell \in I$,
\begin{align} \label{901}
\P \left  ( \lambda_1 \big ( Z_\ell^{(1)} \big ) \geq (1-\delta'') \lambda_{\alpha }   \mid   X^{(1)} \right  ) \1_{\cF_1} &\leq  n^{-  2^{-1}  c^{1-\alpha }(1-\tau )^\alpha     (1-\delta'')^\alpha + (1+\delta)   \varepsilon^{1/2}    + o(1)}  \nonumber \\
&  \leq  n^{-  2^{-1}  c^{1-\alpha }  (1-\delta)^\alpha + \zeta_1+ o(1)} 
\end{align}
for some $\zeta_1   = \zeta_1(\e) $ with $\lim_{\e \ro} \zeta_1=0$. The last inequality follows from the fact that $\lim_{\e \rightarrow 0} \delta'' = \delta$  and $\lim_{\e \rightarrow 0} \tau = 0$ (see the definition of $\delta''$ and $\tau$ in \eqref{delta''} and \eqref{setting2} respectively). 

In addition, for  $\ell\in J$,  using the fact
$\phi_{\beta/2}(2) = 2^{1-\beta}$ (see \eqref{clique 2} in  Lemma \ref{lemma 02}),
by \eqref{900} again,
\begin{align} \label{902}
\P \left ( \lambda_1 \big ( Z_\ell^{(1)} \big ) \geq (1-\delta'') \lambda_{\alpha }   \mid   X^{(1)} \right   )\1_{\cF_1}  \leq n^{-  (1- \tau)^\alpha (1-\delta'')^\alpha + (1+\delta)   \varepsilon^{1/2}  + o(1)} \leq  n^{-    (1-\delta)^\alpha +\zeta_2+ o(1)} 
\end{align}
for some $\zeta_2   = \zeta_2(\e) $ with $\lim_{\e \ro} \zeta_2=0$, where we used $2^{-1}\cdot  (2^{1-\beta})^{1-\alpha} = 1$ in the first inequality.

{Thus,  by \eqref{901} and \eqref{902} together with the independence of edge-weights across different components, under the event $\cF_1$,}
\begin{align*}
 \P  \left ( \lambda_1 \big ( Z^{(1)} \big ) \leq (1-\delta'') \lambda_{\alpha }  \mid   X^{(1)} \right ) &\geq  \left (1-   n^{-  2^{-1}  c^{1-\alpha }  (1-\delta)^\alpha + \zeta_1+ o(1)} \right )^{\log n} \left  (1  -  n^{-    (1-\delta)^\alpha +\zeta_2+ o(1)} \right )^n
 \\
 &\geq  e^{- n^{1-(1-\delta)^\alpha + \zeta_2 + o(1)}},
\end{align*}
where the powers $\log n$ and $n$ come from the fact that $|I| \leq \log n$ and $|J| \leq n$ respectively.
Therefore, applying this  and  \eqref{860} to \eqref{11},
\begin{align*}
\P \big ( \lambda_{1} (Z) \leq (1 - \delta)  \lambda_{\alpha } \big ) \geq  \frac{1}{2}e^{- n^{1-(1-\delta )^\alpha +  \zeta_2 +  o(1)}}.
\end{align*}
By taking  sufficiently small $\e>0$,  we conclude the proof.

\textbf{Case 2: $\bm{0<\alpha\leq 1}$.} 
We apply Proposition \ref{keyprop} as in the case $1<\alpha<2$.
{As mentioned after Proposition \ref{keyprop}, the dominating term in the bound \eqref{keydisp2} is the second term. Hence, for each $\ell=1,\cdots,m$,}
\begin{align*}  
\mathbb{P} \left ( \lambda_1(Z^{(1)}_\ell) \geq (1-\delta'')  \lambda_{\alpha }   \mid  X^{(1)} \right )  \1_{\cF_1}  
<
 n^{-   ( 1-\delta'')^\alpha (1- \e^{ 1/4 })^\alpha  +  (1+\delta) \e^{1/2}   +  o(1)} \leq n^{-    (1-\delta)^\alpha +\zeta_3+ o(1)} 
\end{align*}
for some $\zeta_3   = \zeta_3(\e) $ with $\lim_{\e \ro} \zeta_3=0$.
Thus, under the event $\cF_1$,
\begin{align*}
 \P \left  ( \lambda_1 \big ( Z^{(1)} \big ) \leq (1-\delta'') \lambda_{\alpha }  \mid   X^{(1)} \right )  &\geq  \left (1-   n^{-    (1-\delta)^\alpha +\zeta_3+ o(1)} \right )^n  \geq  e^{- n^{1-(1-\delta )^\alpha +  \zeta_3 +  o(1)}}.
\end{align*}
By the similar reasoning as before,  we conclude the proof. \qed

\subsubsection{Upper bound for the lower tail}
 
Recall the event $\mathcal{M} = \left \{|E(X)| \geq \frac{d(n-1)}{4} \right \}$ defined in \eqref{event f}. Since $\lambda_1(Z)  \geq \max_{(i,j)\in E(X)}|Z_{ij}|  = \max_{(i,j)\in E(X)}|Y_{ij}| $ (see  Lemma \ref{graphs:max_entry}),
\begin{align*}
 \P \big ( \lambda_1(Z) \leq (1-\delta) \lambda_{\alpha } \big )
 \leq \E \left [ \P \left ( \max_{(i,j)\in E(X)}|Y_{ij}| \leq (1-\delta) \lambda_{\alpha } \mid  X \right ) \1_{\mathcal{M}} \right ]  + \P \left ( \mathcal{M}^c \right ).
\end{align*} 
Since
$
\P (  |Y_{ij}|  > (1-\delta) \lambda_{\alpha }) \geq C_1 n^{-(1-\delta)^\alpha}
$ for any $i\neq j$,
we have
\begin{align*}
 \P  \left ( \max_{(i,j)\in E(X)}|Y_{ij}|  \leq (1-\delta) \lambda_{\alpha } \mid  X \right ) \1_{\mathcal{M}} \leq  \left ( 1 - C_1 n^{-(1-\delta)^\alpha} \right  )^{\frac{d(n-1)}{4}}  \leq 
e^{ -   n^{1 - (1-\delta)^\alpha +o(1)}}.
\end{align*}
Combining this with the bound  $\P(\cM^c) \leq e^{-cn}$ obtained in \eqref{prob f}, we conclude the proof.
 
\qed

\section{Appendix} \label{sec:app}

\subsection{Tails of  the sum of random variables}
In this section, we state two key lemmas about the tail of a sum of i.i.d. Weibull random variables and their conditioned version. {Throughout this section, we assume that $\{Y_i\}_{i=1,2,\cdots}$ are  i.i.d.  random variables such that for $t>0$,}
\begin{align} \label{dist:appendix}
\frac{C_1}{2} e^{-   t^\alpha} \leq \P\big (Y_i \geq  t \big ) \leq \frac{C_2}{2} e^{-   t^\alpha} \quad  \text{ and } \quad  \frac{C_1}{2} e^{-   t^\alpha} \leq \P \big (  Y_i \leq  -t \big ) \leq \frac{C_2}{2} e^{-   t^\alpha}.
\end{align} 
This implies that 
\begin{align*}
 C_1 e^{-  t^\alpha} \leq \P (|Y_i|  \geq  t)  \leq C_2 e^{-  t^\alpha} .
\end{align*}
{Also, for our applications, as has appeared several times in this paper, we define, for  $\e>0$,   $ \tilde{Y}_i $ as the random variable  $Y_i$  conditioned to be greater than $  (\e \log \log n)^{ \frac{1}{\alpha}}$ in absolute value.}

First, we introduce a tail bound for the sum of  squares of i.i.d.  Weibull random variables having a lighter tail than the  Gaussian distribution, i.e. $\alpha>2$. {Recall that $ \lambda^{\textup{light}}_{\alpha } $ denotes the typical value of the largest eigenvalue,  defined in \eqref{typ1}:}
\begin{align*}
\lambda^{\textup{light}}_{\alpha}  
= \left (\frac{2}{\alpha} \right )^\frac{1}{\alpha} \left ( 1- \frac{2}{\alpha} \right )^{\frac{1}{2} - \frac{1}{\alpha}} \frac{(\log n)^\frac{1}{2}}{(\log \log n)^{\frac{1}{2}-\frac{1}{\alpha}}}.
\end{align*}

\begin{lemma}  \label{lemma_prob_sums_light}
Assume that $\alpha>2$.  
Then, for any $t > k \geq 2$, 
\begin{equation} \label{sum_bounds_light_tail}
C_1^k e^{-  t^\frac{\alpha}{2}k^{1-\frac{\alpha}{2}}}
\leq \P \left ( Y_1^2 + \dots + Y_k^2 \geq t \right )
\leq C_2^k \left ( \frac{2 e t }{ k} \right )^{k} e^{-   (t - k)^\frac{\alpha}{2}k^{1-\frac{\alpha}{2} }}.
\end{equation}
In particular, assume that   
$t = d^2 \left ({\lambda^{\textup{light}}_{\alpha }} \right )^2 + o \left ( \frac{ \log n}{\log \log n} \right )$ and  $k =b \frac{\log n}{\log \log n} + o \left ( \frac{\log n}{\log \log n} \right ) $ for some constants $b, d > 0$. Then, 
\begin{equation} \label{common_tail_terms:upper}
\lim_{n \to \infty} - \frac{\log \P \left (Y_1^2 + \dots + Y_k^2 
\geq 
t \right ) }{\log n} 
= 
d^\alpha \frac{2}{\alpha-2} \left (1 - \frac{2}{\alpha}  \right )^\frac{\alpha}{2} b^{1-\frac{\alpha}{2}}  
\end{equation}
and
 \begin{align} \label{cond1}
\lim_{n \to \infty} - \frac{\log \P  \left ( \tilde{Y}_1^2 + \dots + \tilde{Y}_k^2 
\geq   t  \right ) }{\log n} 
= d^\alpha \frac{2}{\alpha-2} \left (1 - \frac{2}{\alpha}  \right )^\frac{\alpha}{2} b^{1-\frac{\alpha}{2}}  -b    \e .
\end{align}
(recall that $\tilde{Y}_i$ is a  conditioned version of $Y_i$).

Finally, in the case
$
  t =  d^2 \left ({\lambda^{\textup{light}}_{\alpha }} \right )^2 + o \left ( \frac{\log n}{\log \log n} \right )$ and $k =o(1) \frac{ \log n}{\log \log n},$
we have
   \begin{align}   \label{cond2}
 \lim_{n \to \infty} - \frac{\log \P \left (Y_1^2 + \dots + Y_k^2 
\geq 
t \right ) }{\log n} 
 = \lim_{n \to \infty} - \frac{\log \P \left  ( \tilde{Y}_1^2 + \dots + \tilde{Y}_k^2 
\geq   t  \right ) }{\log n}  = 
\infty.
  \end{align}

\end{lemma}

{The particular choices of $t$ and $k$ considered in  \eqref{common_tail_terms:upper}-\eqref{cond2}  appear in our applications in Section \ref{sec:light}.}

The next crucial result is about the  tail estimate for the i.i.d. sum of  $\alpha$th-power of \emph{conditioned} Weibull random variables for any $\alpha>0$.

\begin{lemma} \label{chi tail}
Suppose that $\alpha,\e>0$. Then,
there exists a constant $C>0$ depending only on $C_1,C_2$ such that   the following holds: For any $L>m$,
\begin{align} \label{211}
\mathbb{P} \left (|\tilde{Y}_1|^\alpha+\cdots+|\tilde{Y}_m|^\alpha  \geq   L \right )  \leq    C^m  e^{-  L} e^{  m} \left ( \frac{L}{m} \right)^m    e^{\varepsilon     m \log \log n }.
\end{align}
In particular, assume that $ m \leq b   \frac{\log n}{\log \log n} + c $ and $L = a \log n$  for some constants $a,b,c>0$. Then,
\begin{align} \label{212}
\mathbb{P} \left (|\tilde{Y}_1|^\alpha+\cdots+|\tilde{Y}_m|^\alpha \geq    a\log n \right ) \leq    n^{-  a +\varepsilon   b + o(1)}.
\end{align}

\end{lemma}

\subsubsection{Tails of sums of light-tailed random variables}

We verify Lemma \ref{lemma_prob_sums_light} in this section. We first state the following easy lemma regarding the discretization of a tail.

\begin{lemma} \label{discretize_tail}
Suppose that 
  $k \in \mathbb{N}$ and $t>k$. Then, 
\begin{equation} 
\big \{ x_1 + \cdots + x_k \geq t, x_1,\cdots,x_k\geq 0 \big \} \subseteq 
\bigcup_{
\substack{
(t_1, \dots, t_k) \in \mathbb{N}_{\geq 0}^k\\
t_1 + \cdots + t_k = \lfloor t \rfloor - k + 1
}} 
\big \{x_1 \geq t_1 , \dots , x_k \geq t_k \big \} .
\end{equation}
Here, 
$\mathbb{N}_{\geq 0}$ denotes the set of non-negative integers.
\end{lemma}
\begin{proof}
 
 Since $0\leq x-\lfloor x \rfloor <1$,
 we have  $x_1 - \lfloor x_1 \rfloor + \dots + x_k - \lfloor x_k \rfloor < k$. Using this fact,  for any non-negative $x_1,\cdots,x_k$,
\begin{align*}
 \big \{ x_1 + \dots + x_k \geq t \big \} 
&    \subseteq \
\big \{ \lfloor x_1 \rfloor + \dots + \lfloor x_k \rfloor > \lfloor t \rfloor - k \big \} \\
&  = \ 
\big \{ \lfloor x_1 \rfloor + \dots + \lfloor x_k \rfloor \geq \lfloor t \rfloor - k + 1 \big \}
\quad \\
& =
\bigcup_{
\substack{
(t_1, \dots, t_k) \in \mathbb{N}_{\geq 0}^{k} \\
t_1 + \dots + t_k = \lfloor t \rfloor - k + 1
}} 
\big \{ \lfloor x_1 \rfloor \geq t_1, \dots, \lfloor x_k \rfloor \geq t_k \big \}.
\end{align*}
Since  $t_i$s are integers, $\lfloor x_i\rfloor \geq t_i $ is equivalent to $ x_i \geq t_i $. Thus, we are done.
\end{proof}

\begin{proof}[Proof of Lemma \ref{lemma_prob_sums_light}]
For the lower bound, note that
\begin{align*}
\P \left (Y_1^2 + \dots + Y_k^2 \geq t \right ) &\geq \P \Big (Y_1^2 \geq \frac{t}{k},  \dots,  Y_k^2 \geq \frac{t}{k}  \Big ) = \P \left (Y_1^2 \geq \frac{t}{k}\right )^k   \geq C_1^k e^{-  t^\frac{\alpha}{2}k^{1-\frac{\alpha}{2}}}.
\end{align*}

The upper bound is obtained as an application of the discretization from Lemma \ref{discretize_tail}. In fact,
\begin{align*} 
\P \left ( Y_1^2 + \dots + Y_k^2 \geq t \right ) & \leq 
\sum 
\P \left ( Y^2_1 \geq  t_1, \dots, Y^2_k \geq  t_k \right )\leq  C_2^k \sum   
e^{-  \sum_{i=1}^k t_i^\frac{\alpha}{2}} ,
\end{align*}
where the summation is taken over $(t_1, \dots, t_k) \in \mathbb{N}_{\geq 0}^k$ with $
t_1 + \cdots + t_k = \lfloor t \rfloor - k +1$. Since the function $f(x) = x^\frac{\alpha}{2}$ with $\alpha > 2$ is convex, by Jensen's inequality,
\begin{equation*} 
\frac{1}{k} \sum_{i=1}^{k} t_i^\frac{\alpha}{2} \geq \left ( \frac{1}{k} \sum_{i=1}^{k} t_i  \right )^\frac{\alpha}{2}  .
\end{equation*}
Hence,  
\begin{align} \label{961}
\P \left (Y_1^2 + \dots + Y_k^2 \geq t \right ) \leq  C_2^k \sum
 e^{-  (\lfloor t \rfloor - k +1)^\frac{\alpha}{2}k^{1-\frac{\alpha}{2} }} <  C_2^k \sum
e^{-  (t- k )^\frac{\alpha}{2}k^{1-\frac{\alpha}{2} }} .
\end{align}
We now bound the number of summands, i.e. the number of tuples $(t_1, \dots, t_k) \in \mathbb{N}_{\geq 0}^k$ such that $t_1 + \cdots + t_k = \lfloor t \rfloor - k +1$. This is known as the number of weak compositions of $\lfloor t \rfloor - k +1$ into $k$ terms, and is given by $\binom{\lfloor t \rfloor - k +1 + k -1}{k-1} = \binom{\lfloor t \rfloor}{k-1}$. Using the fact that $\binom{n}{k} \leq \left ( \frac{en}{k} \right )^k $ and the condition $t>k$, the number of summands is thus bounded by
\begin{align*} 
 \binom{\lfloor t \rfloor}{k-1} \leq \left ( \frac{e \lfloor t \rfloor}{ k-1} \right )^{k-1} \leq \left ( \frac{2 e t }{ k} \right )^{k}.
\end{align*} 
The upper bound is established by applying this to \eqref{961}.

To obtain \eqref{common_tail_terms:upper}, we simply plug the given values into \eqref{sum_bounds_light_tail} and then use that $(\frac{2e t}{k})^k = n^{o(1)}$ and
\begin{align*}
 (t - k)^\frac{\alpha}{2}k^{1-\frac{\alpha}{2} }
 = 
(1+o(1))   d^\alpha  \frac{2}{\alpha-2}  \Big(1-\frac{2}{\alpha}\Big)^{\frac{\alpha}{2}} b^{1-\frac{\alpha}{2}}\log n .
\end{align*}
 Noting in addition that
\begin{align*}
 \P  \left ( \tilde{Y}_1^2 + \dots + \tilde{Y}_k^2 \geq  t \right  ) = \frac{\P \left (Y_1^2 + \dots + Y_k^2 \geq   t \right )}{\P \left ( |Y_1| \geq (\e \log \log n \right )^\frac{1}{\alpha})^k},
\end{align*}
and recalling the tail probabilities of $Y_i$ in \eqref{dist:appendix}, we establish  \eqref{cond1}.

The final statement \eqref{cond2} follows by similar calculations. Note that formally sending $b\rightarrow 0$ in   \eqref{common_tail_terms:upper} and \eqref{cond1} gives \eqref{cond2} (recall that $\alpha>2$).
\end{proof}

{
\begin{remark} \label{condition1}
The proof of Lemma  \ref{lemma_prob_sums_light}  shows that   the lower bound for the probability $\P(Y_1^2 + \dots + Y_k^2 \geq t )$, obtained by forcing  $Y_i^2 \geq \frac{t}{k}$ for all $i=1,\cdots,k$,  is of the same order (at the exponential scale) as the upper bound. 
 
\end{remark}
}

\subsubsection{Tails of sums of $\alpha$th-power of  random variables}

\begin{proof} [Proof of Lemma \ref{chi tail}]

By  the  exponential Chebyshev bound, for any $s>0$,
\begin{align} \label{500}
\mathbb{P} \left (|\tilde{Y}_1|^\alpha+\cdots+|\tilde{Y}_m|^\alpha \geq    L \right )   \leq e^{-  sL }  \mathbb{E}\left [ e^{s|\tilde{Y}_1|^\alpha} \right ]^m.
\end{align}
{
Recalling that $\tilde{Y}_1$ is  a random variable $Y_1$   conditioned to have an absolute value greater than $(\e \log \log n)^{\frac{1}{\alpha}}$,  using a tail bound \eqref{dist:appendix}, there exists a constant $C \geq 1$ such that   for $x>    (\e \log \log n)^{\frac{1}{\alpha}}  $,
\begin{align*}
\mathbb{P} \left ( |\tilde{Y}_1 |> x \right ) \leq   C e^{ \e    \log \log n }    e^{-   x^\alpha } ,
\end{align*}
and for $0\leq x\leq (\e \log \log n)^{\frac{1}{\alpha}} $, $\mathbb{P}(|\tilde{Y}_1 |> x)  = 1$.
Hence,  for any {$0<s< 1$,}
\begin{align*}
\mathbb{E} \left [ e^{s|\tilde{Y}_1|^\alpha} \right ] &= 1  +  \int_0^\infty  e^{sx^\alpha} s \alpha x^{\alpha-1 } \mathbb{P} \left (|\tilde{Y}_1| > x \right ) dx \\
&=   1+ \int_0^{  (\e \log \log n)^{\frac{1}{\alpha}}  } e^{sx^\alpha} s \alpha x^{\alpha-1 }  dx +  C \int_{ (\e \log \log n)^{\frac{1}{\alpha}}   }^\infty  e^{sx^\alpha} s \alpha x^{\alpha-1 }  e^{\e   \log \log n }   e^{-   x^\alpha } dx \\
&\leq      e^{ \e s \log \log n }      +   C e^{ \e    \log \log n }   \frac{s}{ 1-s} e^{- \varepsilon  ( 1-s) \log \log n } \leq   \Big(1+ \frac{Cs}{ 1-s}\Big)  e^{ \e s \log \log n } .
\end{align*}
{Note that the tail bound \eqref{dist:appendix} implies that the moment generating function is infinity for $s\geq 1$.}
 
Applying this to \eqref{500}, {by Chernoff's bound, for any $0<s< 1$,}
\begin{align*}
\mathbb{P} \left ( |\tilde{Y}_1|^\alpha+\cdots+|\tilde{Y}_m|^\alpha \geq    L \right ) \leq    e^{-sL}   \Big ( 1+ \frac{Cs}{ 1-s}  \Big)^m        e^{ \e s m  \log \log n } .
\end{align*}
Now recall that  $L>m$ and set $s: =    1- \frac{m}{  L   } \in (0,1) $ in order to balance the two terms  $e^{-sL}$ and $ (1+ \frac{Cs}{ 1-s}  )^m   $. Then, 
\begin{align*}
\mathbb{P} \left ( |\tilde{Y}_1|^\alpha+\cdots+|\tilde{Y}_m|^\alpha \geq    L \right ) &\leq  e^{-  L} e^{  m} \left (  1+ \frac{C(L-m)}{m}  \right )^m    e^{ \e   m \log \log n } \\
&\leq   C^m  e^{-  L} e^{ m} \left ( \frac{L}{m} \right )^m    e^{\varepsilon  m \log \log n },
\end{align*}
which concludes the proof of \eqref{211}. 
Since the function $x\mapsto (\frac{L}{x})^x$ is increasing on $(0, \frac{L}{e})$, by taking $m=b\frac{\log n}{\log \log n}+c$ and $L=a \log n$ in the RHS of \eqref{211}, we  establish \eqref{212}.
}

\end{proof}

{
\begin{remark} \label{remark 7.5}
The above  lemmas \ref{lemma_prob_sums_light} and \ref{chi tail}, with the same argument, hold as well for a slightly more general class of Weibull distributions satisfying 
\begin{align*}
C_1  t^{-c_1} e^{-\eta t^\alpha} \leq \P (|Y_i| > t) \leq C_2 t^{-c_2} e^{-\eta t^\alpha} 
\end{align*}
with some constants $c_1,c_2\geq 0$.  This can be used to generalize our results as indicated in Remark \ref{poly}. 
\end{remark}
}

\subsection{Relative entropy  results}
 {In this final section, we state  bounds on the relative entropy and a tail of binomial distributions.
We denote by $I_p$ the relative entropy functional
\begin{align}\label{entropy}
I_p(q) := q \log \frac{q}{p} + (1-q) \log \frac{1-q}{1-p}.
\end{align} }

\begin{lemma}\label{rel_ent_bern}
 {There is a universal constant $c>0$ such that} for any $0<p<1$,
\begin{equation}
I_p \left ( \frac{p}{2} \right ) \geq  c p.
\end{equation}
This implies that 
$ 
I_{1-p} \left (1- \frac{p}{2} \right ) \geq  c p$ and for any $0<q\leq \frac{p}{2}$,
$ 
I_p(q) \geq c p.$

\end{lemma}
\begin{proof}
{Using the  inequality $\log(1 + x)   \geq  \frac{x}{x+1}$ for $x \geq 0$ \footnote{$\frac{d}{dx} (\log(1 + x)  -   \frac{x}{x+1} )= \frac{x}{(x+1)^2} \geq 0$ for $x\geq 0$.} },
 {
\begin{align*}
I_p \left ( \frac{p}{2}  \right ) 
& = 
- \frac{p}{2} \log 2 + \frac{2-p}{2} \log \left ( 1 + \frac{p}{2(1 - p)} \right )  
\geq
- \frac{p}{2} \log 2 + \frac{2-p}{2}   \frac{ \frac{p}{2(1 - p)}}{ 1 + \frac{p}{2(1 - p)} } = \frac{p}{2}  (1-\log 2).
\end{align*}
}
The second inequality is a direct consequence of  $I_p \left ( \frac{p}{2}  \right ) = I_{1-p} \left (1- \frac{p}{2} \right )$ and the {last statement uses the fact that $I_p(x)$ is decreasing on $x\in (0,p)$ \footnote{$I_p'(x) = \log (\frac{x}{p}\cdot  \frac{1-p}{1-x}) < 0 $ for $x\in (0,p)$.}.}
\end{proof}

\begin{lemma} [Lemma 3.3 in \cite{lzvariational}]\label{entropy bound}
If $0<p\ll q$ (i.e.  $\frac{q}{p} \to \infty$ as $p,q\rightarrow 0$) and $q \leq 1-p$ then
\begin{align*}
I_p(p+q) = (1+o(1)) q \log \Big(\frac{q}{p}\Big).
\end{align*}
\end{lemma}

The next lemma provides a sharp bound on the tail probability of the Binomial distribution.

\begin{lemma}[Lemma 4.7.2 in \cite{Ash}] \label{binomial_tails}
For $m \in \mathbb{N}$ and $0<q<1$, let $X$ be a random variable that has the distribution $\Binom(m, q)$. Then, for any $q < \theta < 1$,
\begin{equation} \label{tail1}
\frac{1}{\sqrt{8m \theta(1-\theta)} }e^{- m I_q(\theta)}
\leq 
\P \left ( X \geq \theta m \right ) 
\leq 
e^{- m I_q(\theta)}.
\end{equation}
\end{lemma}
{ By applying the above to the random variable $m-X$ and using that $I_{q}(\theta) = I_{1-q}(1-\theta)$ we additionally have that, for  $0< \theta < q < 1$, 
\begin{equation} \label{tail2}
\frac{1}{\sqrt{8m \theta(1-\theta)} }e^{- m I_q(\theta)}
\leq 
\P \left ( X \leq \theta m \right ) 
\leq 
e^{- m I_q(\theta)}.
\end{equation}
}

\bibliography{sparse}
\bibliographystyle{plain}

\end{document}